\documentclass[12pt,reqno]{amsart}
\usepackage[lmargin=1.15in,rmargin=1.15in,tmargin=1in,bmargin=1in]{geometry}

\usepackage{amsthm,amsfonts,amssymb,yfonts} 
\usepackage[utf8]{inputenc}
\usepackage[T1]{fontenc}
\usepackage{mathrsfs} 
\usepackage{enumitem}
\usepackage[colorlinks,unicode,pdfencoding=auto]{hyperref} 
\hypersetup{linkcolor=blue, citecolor=blue}
\usepackage{bookmark}
\usepackage[nobysame,initials]{amsrefs}

\usepackage{mathtools,hhline,multirow}

\let\originalleft\left
\let\originalright\right
\renewcommand{\left}{\mathopen{}\mathclose\bgroup\originalleft}
\renewcommand{\right}{\aftergroup\egroup\originalright}

\usepackage{fancyhdr}
\newtheorem{theorem}{Theorem}
\newtheorem{proposition}[theorem]{Proposition}
\newtheorem{lemma}[theorem]{Lemma}
\newtheorem{corollary}[theorem]{Corollary}

\theoremstyle{definition}
\newtheorem{remark}[theorem]{Remark}
\newtheorem{definition}[theorem]{Definition}

\numberwithin{theorem}{section}
\numberwithin{equation}{section}
\numberwithin{figure}{section}

\AtBeginDocument{%
   \def\MR#1{}
}

\makeatletter
\def\section{%
    \@startsection{section}{1}%
    \z@{3\linespacing\@plus\linespacing}{.5\linespacing}%
    {\normalfont\large\bfseries}%
}
\makeatother

\makeatletter
\def\@seccntformat#1{%
  \protect\textup{\protect\@secnumfont
    \csname the#1\endcsname
\space\space
  }%
}
\makeatother

\makeatletter
\def\subsection{\@startsection{subsection}{2}%
  \z@{1.5\linespacing\@plus.5\linespacing}{-.5em}%
  {\normalfont\bfseries}}
\makeatother

\newcommand{\nn}{\mathbb{N}}

\newcommand{\zz}{\mathbb{Z}}

\newcommand{\fp}{\mathfrak{p}}

\newcommand{\fL}{\mathfrak{L}}

\newcommand{\goL}{\mathfrak{L}}

\newcommand{\xa}{\alpha}
\newcommand{\xb}{\beta}
\newcommand{\xc}{\chi}
\renewcommand{\xd}{\delta} 
\newcommand{\xe}{\epsilon}
\newcommand{\xf}{\phi}

\newcommand{\xh}{\eta}
\newcommand{\xj}{\varphi}
\newcommand{\xk}{\kappa}
\newcommand{\xl}{\lambda}
\newcommand{\xo}{\omega}

\newcommand{\xr}{\rho}
\newcommand{\xs}{\sigma}
\newcommand{\xt}{\tau}

\newcommand{\xvq}{\vartheta}

\newcommand{\xx}{\xi}
\newcommand{\xy}{\psi}
\newcommand{\xz}{\zeta}

\newcommand{\xG}{\Gamma}

\newcommand{\xF}{\Phi}
\newcommand{\xL}{\Lambda}

\newcommand{\xW}{\Omega}

\newcommand{\caB}{\mathcal{B}}

\newcommand{\bfa}{\mathbf{a}}
\newcommand{\bfe}{\mathbf{e}}
\newcommand{\bfo}{\mathbf{o}}
\newcommand{\bfr}{\mathbf{r}}
\newcommand{\bfs}{\mathbf{s}}

\newcommand*{\boldone}{\text{\usefont{U}{bbold}{m}{n}1}}

\renewcommand{\(}{\begin{equation}}
\renewcommand{\)}{\end{equation}}

\renewcommand{\~}[1]{\tilde{#1}}
\renewcommand{\-}[1]{\bar{#1}}
\renewcommand{\^}[1]{\hat{#1}}

\renewcommand{\epsilon}{\varepsilon}

\renewcommand{\Re}{\operatorname{Re}}

\newcommand{\lf}{\left}
\newcommand{\rh}{\right}

\newcommand{\fr}[2]{\frac{#1}{#2}}
\newcommand{\nfr}[2]{#1/#2}
\newcommand{\tf}[2]{\tfrac{#1}{#2}}
\newcommand{\tfr}[2]{\tfrac{#1}{#2}}

\let\midold\mid
\let\nmidold\nmid
\renewcommand{\mid}{\hspace{-0.2em}\midold\hspace{-0.2em}}
\renewcommand{\nmid}{\hspace{-0.15em}\nmidold\hspace{-0.15em}}
\newcommand{\ssnmid}{\hspace{0.1em}\nmidold\hspace{0.1em}}

\newcommand{\ol}{\overline}

\newcommand{\tops}[2]{\texorpdfstring{#1}{#2}}

\let\mod\undefined
\newcommand{\mod}[1]{\,\,(\mathrm{mod}\,\,#1)}
\newcommand{\md}[1]{\,(#1)}

\newcommand{\Saw}[1]{%
  \left(\mkern-3mu\left(#1\right)\mkern-3mu\right)}

\newcommand{\BigSaw}[1]{%
  \Big(\mkern-6mu\Big(#1\Big)\mkern-6mu\Big)}

\newcommand{\sumsub}[1]{\sum_{\substack{#1}}}
\newcommand{\B}{\Big}
\newcommand{\bb}{\bigg}
\newcommand{\ul}{\underline}

\renewcommand{\.}{\mspace{1.5mu}} 
\newcommand{\?}{\mspace{-1.5mu}} 

\newcommand{\psum}{\sideset{}{'}\sum}
\newcommand{\tjac}[2]{(\tfrac{#1}{#2})} 
\newcommand{\tsum}{{\textstyle\sum}}

\newcommand{\floor}[1]{\lfloor #1 \rfloor}

\newcommand{\rt}[1]{\sqrt{#1}}

\begin{document}

\setlength{\baselineskip}{14.5pt}

\title{Periodic vanishings of the Legendre-$17$ signed partition numbers} 
\author[T.\ Daniels]{Taylor Daniels}
\address{Dept.\ of Mathematics, Purdue University, West Lafayette, IN 47907}
\email{daniel84@purdue.edu}
\begin{abstract}
    For $f : \mathbb{N} \to \{0,\pm 1\}$ the $f$-\emph{signed partition numbers} $\fp(n,f)$ are defined to be the weighted partition sums
    \[
        \fp(n,f) = 
        \sum_{\substack{
            x_{1}+\cdots+x_{k} = n \\[1.5\lineskip]
            x_{1} \geq \cdots \geq x_{k} > 0 \\[1.5\lineskip]
            k \geq 1
        }} f(x_{1})f(x_{2})\cdots f(x_{k}).
    \]
    For prime $p > 2$, let $\tjac{\cdot}{p}$ denote the Legendre symbol modulo $p$. 

    The first half of this paper derives Rademacher-style series formulae for the quantities $\fp(n,\pm\tjac{\cdot}{p})$ for $p < 24$ satisfying $p \equiv 1\mod{4}$ (that is, for $p=5,13,17$), and the extensions to general $p \equiv 1 \mod{4}$ are made apparent in our derivations. In the second half of this paper, the series formulae for $\fp(n,\pm\tjac{\cdot}{17})$, as well as various properties of Dedekind sums and their ``character-twisted'' analogues, are used to establish that these two quantities are identically zero on certain (mod $34$)-arithmetic progressions. 
\end{abstract}

\maketitle

\section{Introduction}

The \emph{partitions} of a given $n \in \nn$ are the integer tuples $(x_{1},x_{2},\ldots,x_{k})$ such that $x_{1}+\cdots+x_{k}=n$ and $x_{1} \geq \cdots \geq x_{k} > 0$, where $k \geq 1$ is not restricted. For a given $f : \mathbb{N} \to \{0,\pm 1\}$, we define the $f$-\emph{signed partition numbers} $\fp(n,f)$ to be the ``weighted'' partition sums
    \[
        \fp(n,f) := 
        \sum_{\substack{
            x_{1} + \cdots + x_{k} = n \\[1.5\lineskip]
            x_{1} \geq \cdots \geq x_{k} > 0 \\[1.5\lineskip]
            k \geq 1
        }} f(x_{1})f(x_{2})\cdots f(x_{k}).
    \]
In other words, for fixed $n$, to any partition $\pi = (x_{1},\cdots,x_{k})$ of $n$ we assign the \emph{sign} $f(\pi) := f(x_{1})\cdots f(x_{k})$, and $\fp(n,f)$ is simply the sum $\sum_{\pi} f(\pi)$ with $\pi$ running over all partitions of $n$. 
Equivalently, one may define $\fp(n,f)$ via the series expansion
    \[
        \prod_{n\geq 1} (1 - f(n)x^{n})^{-1} = 1 + \sum_{n\geq 1} \fp(n,f) x^{n} \qquad (\text{for complex $|x|<1$}).
    \]

For prime $p > 2$, the \emph{Legendre-signed partition numbers} are the quantities $\fp(n,\pm\tjac{\cdot}{p})$, where $\tjac{\cdot}{p}$ is the Legendre symbol modulo $p$: that is, if $p \nmid a$ then $\tjac{a}{p}$ is $1$ (respectively $-1$) when $a$ is (respectively is not) congruent to a square modulo $p$, and $\tjac{a}{p} = 0$ when $p \mid a$. Throughout this paper $\xc$ indicates a Legendre symbol $\tjac{\cdot}{p}$ modulo some $p > 2$.

The Legendre-signed partition numbers $\fp(n,\xc)$ were introduced in \cite{Daniels:Legendre}, wherein the behaviours of different $\fp(n,\xc)$ as $n \to \infty$ were examined. For instance, Corollaries 1.4 and 1.6 in \cite{Daniels:Legendre} establish that: \emph{For $p \neq 5$ such that $p \not\equiv 1 \mod{8}$, as $n \to \infty$ one has}
    \[
        \fp(n,\xc) \asymp n^{c_{p}} \exp\lf(\tf{1}{2}\xk_{p}\rt{n}\rh)
    \]
\emph{for some constant $c_{p}$, where $a_{n} \asymp b_{n}$ means that $a_{n} = O(b_{n})$ and $b_{n} = O(a_n)$, and}
    \[
        \xk_{p} := \pi\sqrt{\tf{2}{3}(1-\tf{1}{p})}.
    \]
\emph{In particular, for $p\neq 5$ with $p \not\equiv 1 \mod{8}$, one has}
    \(
    \label{eq:p(n,c)ToInfty}
        \fp(n,\xc) \to \infty \qquad\text{as $n \to \infty$}.
    \)
For comparison, we recall that the ``ordinary'' partition numbers $\fp(n,1)$ satisfy
    \[
        \fp(n,1) \sim \big(4\sqrt{3}n\big)^{-1} \exp\lf(\xk\rt{n}\rh) \qquad\text{with $\xk = \pi\sqrt{\tf{2}{3}}$},
    \]
as shown by Hardy and Ramanujan \cite{Hardy1918asymptotic}, where $a_{n} \sim b_{n}$ indicates $\fr{a_{n}}{b_{n}} \to 1$ as $n \to \infty$.

The cases of $\fp(n,\pm\tjac{\cdot}{5})$ provide a surprising contrast to \eqref{eq:p(n,c)ToInfty}. Namely, as shown in \cite{Daniels:vanishing}, one has
    \begin{alignat}{2}
    \label{eq:p(n,5)Vanishing}
        \fp\lf(n,\tjac{\cdot}{5}\rh) &= 0 \quad &&\text{for all $n \equiv 2 \mod{10}$}, \\
    \label{eq:p(n,-5)Vanishing}
        \fp\lf(n,-\tjac{\cdot}{5}\rh) &= 0 \quad &&\text{for all $n \equiv 6 \mod{10}$}.
    \end{alignat}
In light of \eqref{eq:p(n,c)ToInfty}--\eqref{eq:p(n,-5)Vanishing}, it is natural to consider what primes (if any) with $p \equiv 1 \mod{8}$ also have periodic vanishings in the quantities $\fp(n,\pm\chi)$. In computing $\fp(n,\pm\xc)$ for $n \leq 10000$ and primes $p \leq 2000$, one finds only the single ``candidate'' $p=17$; this paper rigorously establishes these empirical vanishings in $\fp(n,\pm\tjac{\cdot}{17})$, as stated in the following two theorems.

\begin{theorem}
\label{thm:MainVanishing}
    One has
        \(
        \label{eq:p(n,17)Vanishing}
            \fp\lf(n,\tjac{\cdot}{17}\rh) = 0 \qquad\text{for all $n \equiv 17,19,25,27 \mod{34}$.}
        \)
    Equivalently, one has $\fp(n,\tjac{\cdot}{17})=0$ whenever $n$ is odd and $1-24n$ is congruent to a quartic residue modulo 17.
\end{theorem}

\begin{theorem}
\label{thm:DaggerVanishing}
    One has
        \(
        \label{eq:p(n,-17)Vanishing}
            \fp\lf(n,-\tjac{\cdot}{17}\rh) = 0 \qquad\text{for all $n \equiv 11,15,29,33 \mod{34}$.}
        \)
    Equivalently, one has $\fp(n,-\tjac{\cdot}{17})=0$ whenever $n$ is odd and $1-24n$ is congruent to a quadratic-nonquartic residue modulo 17. 
\end{theorem}

Theorems \ref{thm:MainVanishing} and \ref{thm:DaggerVanishing} are proved by first computing Rademacher-style series formulae for $\fp(n,\pm\tjac{\cdot}{17})$, and then showing that the Kloosterman-type sums appearing in the series coefficients vanish on the specified residue classes modulo 34. In the course of our proofs, we find strong evidence for the following conjecture, which is supported by numerical experiments:

\begin{quote}
    The only primes $p > 2$ for which $\fp(n,\pm\tjac{\cdot}{p})$ vanishes on some arithmetic progressions modulo $2p$ (in the sense of \eqref{eq:p(n,5)Vanishing}--\eqref{eq:p(n,-17)Vanishing}) are $5$ and $17$.
\end{quote}

\noindent\textbf{The structure of this paper.}
The first half of this paper, sections \ref{sec:FEQs}--\ref{sec:rademacherSeries}, derives exact series formulae for $\fp(n,\tjac{\cdot}{p})$ by following the classical techniques of Rademacher \cite{Rademacher:PartitionFunction} and Lehner \cite{Lehner:PartitionsMod5}. The majority of the second half, sections \ref{sec:vanishingPrelim}--\ref{sec:vanishingProof}, examines some properties of the coefficients of the series for $\fp(n,\tjac{\cdot}{p})$, ultimately leading to a proof of Theorem \ref{thm:MainVanishing}. These latter sections discuss a large number of basic properties of Dedekind sums and their ``$\chi$-twisted'' analogues.

Section \ref{sec:dagger} gives an abbreviated proof of Theorem \ref{thm:DaggerVanishing}, and lastly section \ref{sec:Transfer} collects some elementary lemmata (used elsewhere in the paper) about a combinatorial quantity determined by the residue of $p$ modulo 8.
\vspace{\baselineskip}

\noindent\textbf{The scope of this paper.}
The derivations of our series formulae for $\fp(n,\pm\xc)$ use classical methods, and the extensions to general $p \equiv 1 \mod{4}$ are apparent in our proofs. On the other hand, even in our simple cases these formulae suffer from significant ``sprawl'' in notation and scope---the general formulae are worse still. Thus, in the hopes of a compromise between generality and simplicity, and out of our particular interest in $\fp(n,\pm\tjac{\cdot}{17})$, throughout this paper we consider only
    \[
        p 
            \quad\text{such that}\quad 
        p < 24 
            \quad\text{and}\quad 
        p \equiv 1 \mod{4}.
    \]
\vspace{-0.8\baselineskip}

\noindent\textbf{Acknowledgements.}
We would like to thank Trevor Wooley for numerous helpful conversations during the course of this research, and thank Ben McReynolds for partial financial support during part of this research. Thanks are also extended to Nicolas Robles for helping to develop early stages of this work, and to Anurag Sahay for helpful proofreads of parts of this paper.

\section{Notation and Definitions}

As the reader knows (or, can imagine), the Taylor series of generating functions (\`{a} la  Rademacher \cite{Rademacher:PartitionFunction}) require extensive---often \mbox{\emph{nightmarish}}---notation (see, e.g., \cites{Hagis:PartitionsPrime,Lehner:PartitionsMod5,Lehmer:SeriesPartition}). Hoping to avoid further ``inflation'', we employ a number of notations (both standard and non-) that we hope blend brevity, clarity, and consistency with the broader literature.

Throughout this paper $p$ is a prime with $p<24$ and $p \equiv 1 \mod{4}$, unless indicated otherwise. We reserve $\xc$ for a generic Legendre symbol $(\tf{\cdot}{p})$, and when $p$ is fixed we write $\xc_{a}$ to indicate $\tjac{a}{p}$; outside of expressions such as $\fp(n,\tjac{\cdot}{p})$, we do not use ``parenthesized'' fractions to denote the Legendre symbol. 
In addition, throughout this paper
    \(
        q = \fr{p-1}{2},
    \)
unless indicated otherwise.

The monikers ``\emph{quadratic}'' and ``\emph{nonquadratic}'' abbreviate ``quadratic residue'' and ``quadratic nonresidue'', respectively, and statements such as ``$x$ is a quadratic (mod $p$)'' are understood to mean ``$x$ is [congruent to] a quadratic [residue] modulo $p$''. 

For integers $a$ and $k$, with $k \geq 1$, let $\{a\}_{k}$ denote the unique integer such that
    \(
    \label{eq:frac(x)(K)}
        a \equiv \{a\}_{k} \mod{k} \qquad\text{and}\qquad 0 \leq \{a\}_{k} < k.
    \)
In addition, for $k \geq 1$ define $\|a\| = \|a\|_{k}$ via
    \(
    \label{eq:kNorm}
        \|a\|_{k} = \begin{cases}
            \{a\}_{k}, & 0 \leq \{a\}_{k} \leq \tf{k}{2}, \\
            k - \{a\}_{k}, & \tf{k}{2} \leq \{a\}_{k} < k,
        \end{cases}
    \)
so that
    \(
        a \equiv \pm\|a\|_{k} \mod{k} \qquad\text{and}\qquad 0 \leq \|a\|_{k} \leq \tf{k}{2}.
    \)

Bold letters $\bfa$ denote subsets of $\{1,\ldots,\fr{p-1}{2}\}$ and have $|\bfa|$ elements. Sums and products such as $\sum_{\bfa} a^{2}$ and $\prod_{\bfa} \cos(\pi a)$ are taken as $a$ runs over $\bfa$, and expressions such as $\|h\bfa\|$ are understood to mean $\{\|ha\|:a\in \bfa\}$. For fixed $p$, we reserve $\bfr$ and $\bfs$ for the subsets of $\{1,\ldots,\fr{p-1}{2}\}$ containing the quadratics and nonquadratics (mod $p$), respectively, \emph{except} in section \ref{sec:Transfer}, where we must abuse notation slightly.

For complex $z$ and $q$, with $|q|<1$, we employ the standard \emph{$q$-Pochhammer} symbols
    \[
        (z;q)_{\infty} = \prod_{n=0}^{\infty} (1-zq^n) \qquad\text{and}\qquad (z_{1},\ldots,z_{k};q)_{\infty} = \prod_{m=1}^{k} (z_{m};q)_{\infty}.
    \]
The relation $a_n \ll b_n$ indicates that $a_n = O(b_n)$ as $n \to \infty$. Such inequalities involving $\xe$ are understood to hold for arbitrary $\xe > 0$, with implicit constant depending on $\xe$. 
\vspace{\baselineskip}

\noindent\textbf{Some key notation.} Here we indicate two items of frequently used notation that may cause confusion if not emphasized here at the beginning. 
\begin{enumerate}[label=(\arabic*)]
\item%
    In this paper $h$, $k$, $H$, and $K$ are always positive integers, and one always has
        \[
            H = hp \qquad\text{and}\qquad K = kp.
        \]
    Thus $H$ and $K$ are always divisible by $p$, but $h$ and $k$ are \emph{not} a-priori coprime to $p$.
\item%
    For $k \geq 1$,
        \(
        \label{eq:SumConventions}
            \sum_{h \md{k}} \quad\text{abbreviates} \quad \sum_{h=0}^{k-1}\, , 
            \qquad\text{and}\qquad 
            \psum_{h\md{k}} \quad\text{abbreviates}\quad \sum_{\substack{h=0\\ (h,k)=1}}^{k-1}.
        \)
    Thus, sums $\sum_{h \md{k}}$ are not always equivalent to $\sum_{h \mod{k}}$. 
\end{enumerate}

For real $x$, the expression $\Saw{x}$ indicates the standard ``sawtooth'' function appearing in the classical \emph{Dedekind sums}, namely
    \(
    \label{eq:Saw}
        \Saw{x} = x - \lfloor x \rfloor - \tf{1}{2} + \tf{1}{2}\xd(x), 
    \qquad\text{where}\qquad 
        \xd(x) = \boldone_{\zz}(x).
    \)

\begin{definition}
    For general $p \equiv 1 \mod{4}$ and integer $k \geq 1$, let $\xj := \fr{p}{(k,p)}$, so that \mbox{$\xj k = \mathrm{lcm}(k,p)$}. Then, for $\bfa \subset \{1,\ldots,\fr{p-1}{2}\}$, define
    \(
    \label{eq:omega-sub-a}
        \xo_{\bfa}(h,k) = \exp\!\bb\{\pi i \sum_{a \in \bfa}
        \sumsub{\mu \md{\xj k} \\ \mu \equiv \pm a \mod{p}} 
        \BigSaw{\fr{h\mu}{k}}\BigSaw{\fr{\mu}{\xj k}}\bb\}.
    \)
\end{definition}

It is easy to show\footnotemark~that: For any integral $h,k,q \geq 1$, one has
    \(
        \xo_{\bfa}(qh,qk) = \xo_{\bfa}(h,k).
    \)
\footnotetext{The proof mimics that of \cite{Rademacher:TheoremsOnDedekindSums}*{Thm.~1}.}%

\begin{definition}
\label{def:phihk}
    For fixed $p$ and for $h,k \geq 1$, let
        \(
        \label{eq:phi(h,k)Omega}
            \xf_{h,k} := \xo_{\bfr}(h,k)\fr{\xo_{\bfs}(2h,k)}{\xo_{\bfs}(h,k)}.
        \)
\end{definition}

Because $\bfr\cup(p-\bfr)$ and $\bfs \cup (p-\bfs)$ exhaust the elements of $\{1,\ldots,p-1\}$ which are quadratics and nonquadratics (mod $p$), respectively, it is easily seen that
    \(
    \label{eq:phi(h,k)SawSum}
    \xf_{h,k} = \exp\!\bigg\{{
        \sumsub{\mu \md{\xj k} \\ \xc_{\mu}=1} \BigSaw{\fr{h\mu}{k}} \BigSaw{\fr{\mu}{\xj k}} 
        + \sumsub{\mu\md{\xj k} \\ \xc_{\mu}=-1} 
        \Big[ \BigSaw{\fr{2h\mu}{k}} - \BigSaw{\fr{h\mu}{k}} \Big] \BigSaw{\fr{\mu}{\xj k}}
    }\bigg\},
    \)
where again $\xj = \fr{p}{(k,p)}$. For integer $a$ and fixed $p$, let
    \(
    \label{eq:B(a)}
        \caB_{a} := 6\{a\}_{p}^{2} - 6p\{a\}_{p} + p^{2}.
    \)
We note that if $\xb_{2}(x)$ is the second \emph{periodic Bernoulli function}, namely
    \[
        \xb_{2}(x) := \begin{cases}
            x^{2}-x+\fr{1}{6}, & 0 \leq x < 1, \\
            \xb_{2}(x-\floor{x}) & \text{otherwise},
        \end{cases}
    \]
then
    \[
        \caB_{a} = 6p^{2}\xb_{2}(a/p).
    \]

\section{Functional Equations}
\label{sec:FEQs}

Fixing $p \equiv 1 \mod{4}$, for complex $|x|<1$ we define $\xF(x)$ as the generating function for $\fp(n,\xc)$, namely
    \(
    \label{eq:Phi(x)}
        \xF(x) := \prod_{a=1}^{p-1} (\chi_{a} x^{a}; x^{p})^{-1} = 1 + \sum_{n \geq 1} \fp(n,\xc)x^{n}.
    \)
Our first step in deriving a series formula for $\fp(n,\xc)$ is the collection of several ``functional equations'' for $\xF(x)$, for $x$ in different parts of the unit disk $|x|<1$. 

\begin{definition}
For $\bfa \subset \{1,\ldots,\fr{p-1}{2}\}$ and $|x|<1$, let
    \(
    \label{eq:Fa(x)Ga(x)}
        F_{\bfa}(x) := \prod_{\bfa}(x^{a},x^{p-a};x^{p})_{\infty}^{-1} \qquad\text{and}\qquad G_{\bfa}(x) := \prod_{\bfa}(-x^{a},-x^{p-a};x^{p})_{\infty}^{-1}.
    \)
\end{definition}

Comparing \eqref{eq:Phi(x)} and \eqref{eq:Fa(x)Ga(x)}, and recalling that $\bfr$ and $\bfs$ are the subsets of $\{1,\ldots,\fr{p-1}{2}\}$ of quadratics and nonquadratics (mod $p$), respectively, evidently
    \(
    \label{eq:Phi(x)=Fr(x)Gs(x)}
        \xF(x) = F_{\bfr}(x)G_{\bfs}(x) = F_{\bfr}(x)\fr{F_{\bfs}(x^{2})}{F_{\bfs}(x)}.
    \)
In light of this then, to derive functional equations for $\xF(x)$ we need only find functional equations for $F_{\bfa}(x)$ and $G_{\bfa}(x)$. These derivations were done by Lehner \cite{Lehner:PartitionsMod5} and Hagis Jr.~\cites{Hagis:PartitionsPrime,Hagis:DistinctSummands}, and our equations for $\Phi(x)$ follow from an exercise in tedious bookkeeping.

For $0 < h \leq k$ with $(h,k)=1$, and complex $z$ with $\Re z > 0$, write
    \(
    \label{eq:Complexx}
        x = \exp\!\big({2\pi i\nfr{h}{k}-\nfr{2\pi z}{k}}\big).
    \)
Writing $x$ this way, our functional equations for $\xF(x)$ have the common form
    \(
    \label{eq:FEQ-PhiGeneral}
        \xF(x) = \xl_{k}\xf_{h,k}\exp(\xy_{h,k}(z))\xW_{h,k}(\~{x}),
    \)
where 
\begin{enumerate}[label=(\arabic*)]
    \item $\xl_{k}$ is a finite quotient of cosecant terms;
    \item $\xf_{h,k}$ is the root of unity defined in \eqref{eq:phi(h,k)Omega};
    \item $\xy_{h,k}$ is some elementary function;
    \item $\xW_{h,k}$ is an analytic function (for complex $|x| < 1$);
    \item $\~{x}$ is a ``transformed'' point depending on $x$ (via $h$, $k$, and $z$) and on $\mathrm{gcd}(k,2p)$.
\end{enumerate}
Lehner and Hagis Jr.~separate their versions of \eqref{eq:FEQ-PhiGeneral} into four cases, based on $\mathrm{gcd}(k,2p)$. As it happens, we must separate each of these cases based on the quadratic ``residuacities'' of $h$ and $k$ modulo $p$, making a total of \emph{eight} cases. In order to state our equations, we must (unfortunately) employ a bevy of notation.

In the remainder of this section, we assume that $(h,k)=1$ (no assumptions are made about $(h,p)$ or $(k,p)$), and we recall that
    \[
        q = \fr{p-1}{2}.
    \]
In addition, given $(h,k)=1$ we fix $\bar{2}$, $\bar{h}$, and $\bar{H}$ (only if $p \nmid k$) satisfying
    \[
        2\bar{2} \equiv h\bar{h} \equiv H\bar{H} \equiv 1 \mod{k}.
    \]

\subsection*{The cases where \tops{$p \mid k$}{p|k}} 
Suppose that $p \mid k$ and $(h,k)=1$, so that $\xc_{h} \neq 0$.

\begin{enumerate}[label=(\arabic*)]
\item%
When $2p \mid k$, one has
    \begin{subequations}
    \renewcommand{\theequation}{\theparentequation\hspace{0.08333em}\roman{equation}}
    \(
    \label{eq:FEQ(2p)(Parent)}
        \xF(x) = \xf_{h,k}\exp\lf\{\fr{\pi q}{6k}\big(\!-\!z^{-1} + z\big)\rh\} 
        \times \begin{cases}
        R^{+}(x') & \text{if $\xc_{h}=+1$}, \\
        R^{-}(x') & \text{if $\xc_{h}=-1$},
        \end{cases}
    \tag{\theparentequation}
    \)
where 
    \(
    \label{eq:x'}
        x' = \exp\!\big({-2\pi i\nfr{\bar{h}}{k}-\nfr{2\pi}{kz}}\big),
    \)
and 
    \begin{align}
        \label{eq:Rplus}
            R^{\pm}(x) &:= \prod_{\bfr} (\pm x^{r},\pm x^{p-r};x^{p})_{\infty}^{-1} \prod_{\bfs} (\mp x^{s},\mp x^{p-s};x^{p})_{\infty}^{-1}.
    \end{align}
    \end{subequations}
\item%
When $(k,2p)=p$, after some basic algebraic manipulations one finds that
    \begin{subequations}
    \renewcommand{\theequation}{\theparentequation\hspace{0.08333em}\roman{equation}}
    \(
    \label{eq:FEQ(p)(Parent)}
    \begin{aligned}
        \xF(x) = \xf_{h,k} \exp\!\bigg\{{
            \fr{\pi q}{6k}\lf[\lf(\fr{3}{q}\xc_{h}\lf(1-\fr{\xc_{2}}{4}\rh)B_{2}(\xc) - \fr{1}{4}\rh)z^{-1} + z\rh]
        }\bigg\}& \\
        & \hspace{-10em} \times \begin{cases}
            S^{+}(y') & \text{if $\xc_{h}=+1$},\\
            S^{-}(y') & \text{if $\xc_{h}=-1$},
        \end{cases}
    \end{aligned}
    \tag{\theparentequation}
    \)
where
    \[
        B_{2}(\xc) = \fr{1}{p}\sum_{\mu=1}^{p-1} \mu^{2} \xc_{\mu},
    \]
    \(
    \label{eq:y'}
        y' = \exp\!\big({-2\pi i(\overline{2h})/k - \nfr{\pi}{kz}}\big),
    \) 
and
    \begin{align}
    \label{eq:Splus}
        S^{+}(x) &:= \prod_{\bfr} (x^{2r},x^{2p-2r};x^{2p})_{\infty}^{-1} \prod_{\bfs} (x^{p+2s},x^{p-2s};x^{2p})_{\infty}^{-1}, \\
    \label{eq:Sminus}
        S^{-}(x) &:= \prod_{\bfr} (x^{p+2r},x^{p-2r};x^{2p})_{\infty}^{-1} \prod_{\bfs} (x^{2s},x^{2p-2s};x^{2p})_{\infty}^{-1}.
    \end{align}
    \end{subequations}
\end{enumerate}

\subsection*{The cases where \tops{$p \nmid k$}{p∤k}} 
Now supposing that $p \nmid k$, for integer $t$ let
    \[
        \xo_{t} := \exp(2\pi i t/p) \qquad\text{and}\qquad \bar{\xo}_{t} := \exp(-2\pi it/p),
    \]
and maintain that $H=hp$ and $K=kp$.

\begin{enumerate}[label=(\arabic*)]
\setcounter{enumi}{2}
\item%
When $(k,2p)=2$, one has
    \begin{subequations}
    \renewcommand{\theequation}{\theparentequation\hspace{0.08333em}\roman{equation}}
    \(
    \label{eq:FEQ(2)(Parent)}
        \xF(x) = \xl_{k}\xf_{h,k}\exp\lf\{{\fr{\pi q}{6k}\Big(\fr{1}{p}z^{-1}+z\Big)}\rh\}
        \times \begin{cases}
            T^{+}(x'') & \text{if $\xc_{k}=+1$}, \\
            T^{-}(x'') & \text{if $\xc_{k}=-1$},
        \end{cases}
    \tag{\theparentequation}
    \)
where
    \(
    \label{eq:lambdaEven}
        \xl_{k} = 2^{-\fr{p-1}{4}} \prod_{\bfr} \csc(\pi\|\bar{k}r\|/p) \prod_{\bfs} \fr{\csc(\pi\|2\bar{k}s\|/p)}{\csc(\pi\|\vphantom{\hat{h}}\bar{k}s\|/p)} \qquad (k\bar{k} \equiv 1 \mod{p}),
    \)
    \(
    \label{eq:x''}
        x'' = \exp\!\big({-2\pi i\bar{H}/k - 2\pi /Kz}\big),
    \)
and
    \begin{align}
    \label{eq:Tplus}
        T^{\pm}(x) &:= \prod_{\bfr} (\pm\xo_{r}x,\pm\bar{\xo}_{r}x;x)_{\infty}^{-1} \prod_{\bfs} (\mp \xo_{s}x,\mp \bar{\xo}_{s} x;x)_{\infty}^{-1}.
    \end{align}
    \end{subequations}
\item%
When $(k,2p)=1$, one has
    \begin{subequations}
    \renewcommand{\theequation}{\theparentequation\hspace{0.08333em}\roman{equation}}
    \(
    \label{eq:FEQ(1)(Parent)}
        \xF(x) = \xl_{k}\xf_{h,k}\exp\lf\{{\fr{\pi q}{6k} \Big(\fr{1}{4p}z^{-1}+z\Big)}\rh\}
        \times \begin{cases}
            U^{+}(y'') & \text{if $\xc_{k}=+1$}, \\
            U^{-}(y'') & \text{if $\xc_{k}=-1$},
        \end{cases}
    \tag{\theparentequation}
    \)
where
    \(
    \label{eq:lambdaOdd}
        \xl_{k} = 2^{-\fr{p-1}{4}} \prod_{\bfr} \csc(\pi \|\bar{k}r\|/p) \qquad (k\bar{k} \equiv 1 \mod{p}),
    \)
    \(
    \label{eq:y''}
        y'' = \exp\!\big({-2\pi i(\ol{2H})/k-\pi /Kz}\big),
    \)
and
    \begin{align}
    \label{eq:Uplus}
        U^{+}(x) &:= \prod_{\bfr} (\xo_{r}x^2, \bar{\xo}_{r} x^{2}; x^{2})_{\infty}^{-1} \prod_{\bfs} (\xo_{s}x, \bar{\xo}_{s} x; x^{2})_{\infty}^{-1}, \\
    \label{eq:Uminus}
        U^{-}(x) &:= \prod_{\bfr} (\xo_{r}x, \bar{\xo}_{r}x; x^{2})_{\infty}^{-1} \prod_{\bfs} (\xo_{s}x^{2}, \bar{\xo}_{s}x^{2}; x^{2})_{\infty}^{-1},
    \end{align}
    \end{subequations}
\end{enumerate}

\begin{remark}
    For consistency between equations \eqref{eq:FEQ-PhiGeneral}--\eqref{eq:FEQ(1)(Parent)}, let
        \(
        \label{eq:lambdaP}
            \xl_{k} := 1 \qquad\text{when}\qquad p \nmid k.
        \)
\end{remark}\vspace{\baselineskip}

\section{An initial formula for \tops{$\fp(n,\xc)$}{p(n,χ)}}
\label{sec:circleSetup}

We now begin the derivation of a series formula for $\fp(n,\xc)$, for fixed $p \equiv 1 \mod{4}$ with $p < 24$. Our work follows that of Hardy and Ramanujan \cite{Hardy1918asymptotic} and Rademacher \cite{Rademacher:PartitionFunction}, so we omit details when appropriate---see, e.g., \cite{andrews1976partitions}*{Ch.~5} for a detailed exposition. Fixing positive integers $n$ and $N$, with $N \geq 2$, we begin with the integral formula
    \[
        \fp(n,\xc) = \fr{1}{2\pi i} \int_{C} \xF(x) x^{-n-1} \,dx,
    \]
where $C$ is the positively oriented circle of radius $\exp(-2\pi N^{-2})$ centered at the origin. We segment the interval $(0,1]$ into disjoint ``Farey arcs''
    \[
        \xi_{h,k} = \B(\tfr{h}{k}-\xvq_{h,k}'\mspace{2mu}, \tfr{h}{k}+\xvq_{h,k}''\B],
    \]
for $1 \leq h \leq k \leq N$ and $(h,k)=1$, and on an arc $\xx_{h,k}$ we write
    \[
        x = \exp\lf\{ 2\pi i\lf(\fr{h}{k}+\xf\rh)- 2\pi N^{-2}\rh\} 
        = \exp\lf\{2\pi i\fr{h}{k} - \fr{2\pi z}{k} \rh\},
    \]
where
    \[
        z := k(N^{-2} - i\xf) \qquad\text{and}\qquad -\xvq'_{h,k} < \xf \leq \xvq''_{h,k}.\vspace{0.5em}
    \]
When there is no risk of confusion, we abbreviate $\xvq_{h,k}'$ and $\xvq_{h,k}''$ to $\xvq'$ and $\xvq''$, respectively.
Then, recalling the summation notations from \eqref{eq:SumConventions}, we have
    \(
    \label{eq:p(n,c)-Sum1}
        \fp(n,\xc) = \sum_{k \leq N} \psum_{h \md{k}} e^{-2\pi inh/k} \int_{-\xvq'}^{\xvq''} \xF(x)\exp(2\pi nz/k) \,d\xf.
    \)

Before separating our integrals by $\mathrm{gcd}(k,2p)$ and using the different functional equations from section \ref{sec:FEQs}, for the moment we use the general relation of form
    \[
        \xF(x) = \xl_{k}\xf_{h,k}\exp(\xy_{h,k}(z))\xW_{h,k}(\~{x}),
    \]
noting that $\xW_{h,k}$ is one of $R^{\pm}$, $S^{\pm}$, $T^{\pm}$, and $U^{\pm}$, and that the form of $\~{x}$ depends on $\mathrm{gcd}(k,2p)$. Using this general equation, the integrals in \eqref{eq:p(n,c)-Sum1} are
    \[
        \int_{-\xvq'}^{\xvq''} \xF(x) \exp(2\pi nz/k) \,d\xf = \xl_{k}\xf_{h,k}\int_{-\xvq'}^{\xvq''} \exp(\xy_{h,k}(z)+2\pi nz/k) \xW_{h,k}(\~{x}) \,d\xf,
    \]
bringing us to the formula 
    \(
    \label{eq:pnchi-RademacherLeqN}
        \fp(n,\xc) = \sum_{k \leq N} \psum_{h \md{k}} \xl_{k}\xf_{h,k}e^{-2\pi inh/k}\int_{-\xvq'}^{\xvq''} \exp(\xy_{h,k}(z)+2\pi nz/k) \xW_{h,k}(\~{x}) \,d\xf.
    \)
Now separating the sum over $k \leq N$ into subsums grouped by $(k,2p)$, and denoting these subsums with $\fp_{(j)}(n,\xc;N)$, we write
    \[
        \fp(n,\xc) = \fp_{(1)}(n,\xc;N) + \fp_{(2)}(n,\xc;N) + \fp_{(p)}(n,\xc;N) + \fp_{(2p)}(n,\xc;N),
    \]
and we separately consider $\fp_{(j)}(n,\xc;N)$ for different $j$. Finally, for later use, let
    \[
        \fp_{(j)}(n,\xc) := \lim_{N \to \infty} \fp_{(j)}(n,\xc;N).
    \]

\subsection{Series formulae for \tops{$\fp_{(1)}(n,\xc)$}{𝔭(1)(n,χ)} and \tops{$\fp_{(2)}(n,\xc)$}{𝔭(2)(n,χ)}}
\label{sec:RadeTermsp1p2}

{ 

\newcommand{\sumNTwo}{\sum_{\substack{k \leq N \\ (k,2p)=2}}}
\newcommand{\sumNTwoP}{\sum_{\substack{k \leq N \\ (k,2p)=2}}}

For ease of exposition, we begin by examining $\fp_{(2)}(n,\xc;N)$ and then $\fp_{(1)}(n,\xc;N)$. Throughout this section we assume that $p \nmid k$ and $(h,k)=1$, and we maintain that $H=hp$ and $K=kp$. 

Suppose that $(k,2p)=2$ and write $x$ as in \eqref{eq:Complexx}. We recall from section \ref{sec:FEQs} that
    \[
        \xF(x) = \xl_{k}\xf_{h,k}\exp\lf\{\fr{\pi q}{6k} \Big( \fr{1}{p}z^{-1}+z \Big)\rh\} T^{\pm}(x''),
    \]
where $\xl_{k}$ is given in \eqref{eq:lambdaEven}, $\xf_{h,k}$ is the root of unity in \eqref{eq:phi(h,k)SawSum}, 
    \[
        x'' = \exp(-2\pi i\nfr{\bar{H}}{k}-\nfr{2\pi}{Kz}),
    \]
and $T^{\pm}(x)$ is defined in \eqref{eq:Tplus}. Writing
    \[
        T^{\pm}(x) = \sum_{m=0}^{\infty} t_{m}^{\pm} x^{m},
    \]
we have
    \[
        T^{\pm}(x'') = \sum_{m=0}^{\infty} t_{m}^{\pm} \exp(-2\pi i\bar{H}m/k-2\pi m/Kz),
    \]
where the ``$\pm$'' here is determined by $\chi_{k}$, not $\chi_{h}$. Using these in \eqref{eq:pnchi-RademacherLeqN}, we deduce that
    \begin{align}
    &\fp_{(2)}(n,\xc;N) \notag\\
        & \qquad = \sumNTwo \psum_{h \md{k}} \xl_{k} \xf_{h,k} e^{-2\pi inh/k} \int_{-\xvq'}^{\xvq''} 
        \exp\lf\{ 
            \fr{\pi q}{6k}\lf(\fr{1}{p}z^{-1} + z\rh) + \fr{2\pi n}{k}z
        \rh\}\notag\\[-1em]
        & \hspace{17em} \times \left(\sum_{m=0}^{\infty} t^{\pm}_{m} \exp\lf\{-\fr{2\pi i\bar{H}m}{k} - \fr{2\pi m}{K}z^{-1}\rh\}\right) \, d\xf \notag\\[0.5em]
    \label{eq:p2(n,chi,N)IntegralSum}
        & \qquad = \sumNTwo \sum_{m=0}^{\infty} \xl_{k} t_{m}^{\pm} \fL_{k}(n,m) \int_{-\xvq'}^{\xvq''} 
        \exp\lf\{ \fr{2\pi}{K}\lf(\fr{p-1}{24}-m\rh)z^{-1} + \fr{2\pi\~{n}}{k}z \rh\} \,d\xf,
    \end{align}
where
    \[
        \goL_{k}(n,m) := \psum_{h \md{k}} \xf_{h,k}\exp\!\big({-2\pi i(nh+m\bar{H})/k}\big) \qquad \big(\text{when $(k,2p)=2$}\big),
    \]
and
    \[
        \~{n} := n + \fr{p-1}{24}.
    \]

We now separate expression \eqref{eq:p2(n,chi,N)IntegralSum} into a main term and an error term, beginning with the latter. Because $\xl_{k}$ is always a product of no more than $\fr{3p}{2}$ quantities $\csc(\pi a/p)$ or $\csc(\pi a/p)^{-1}$, with $0 < a < p$ (as seen in \eqref{eq:lambdaEven}, \eqref{eq:lambdaOdd}, and \eqref{eq:lambdaP}), evidently
    \[
        \xl_{k} \ll 1, \qquad \text{uniformly in $k$}.
    \]
Moreover, as discussed in section \ref{sec:KloostermanBounds}, for $m \geq 0$ and $k$ with $(k,2p)=2$, one has
    \(
    \label{eq:Lk(n,m)(2)Bound}
        \fL_{k}(n,m) \ll n^{1/3}k^{2/3+\xe}.
    \)
Recalling that $z = k(N^{-2} - i\xf)$, one can check that $\Re(\fr{1}{kz}) \geq \tf{1}{2}$ for $-\xvq' \leq \xf \leq \xvq''$, so that when the coefficient of $z^{-1}$ in \eqref{eq:p2(n,chi,N)IntegralSum} is negative, i.e., when $m > \fr{p-1}{24}$, one has
    \[
        \int_{-\xvq'}^{\xvq''} \exp\lf\{\fr{2\pi}{K} \lf(\fr{p-1}{24}-m\rh) z^{-1} + \fr{2\pi\~{n}}{k}z\rh\} \,d\xf 
        \ll \int_{-\xvq'}^{\xvq''} e^{-\pi m/p} \?\cdot\? e^{2\pi n/N^{2}} \,d\xf \ll \fr{e^{2\pi n/N^{2}-\pi m/p}}{kN}.
    \]
Then, because we assume $p < 24$, we have $m > \fr{p-1}{24}$ for all $m \geq 1$, and we deduce that
    \begin{align*}
        & \sumNTwo \sum_{m \geq 1} \xl_{k} t_{m}^{\pm} \fL_{k}(n,m) \int_{-\xvq'}^{\xvq''} \exp\left\{ \fr{2\pi}{K}\lf(\fr{p-1}{24}-m\rh)z^{-1} + \fr{2\pi\~{n}}{k}z \right\} \,d\xf \\
        & \qquad \ll \sumNTwo n^{\nfr{1}{3}}k^{\nfr{2}{3}+\xe} \sum_{m \geq 1} \bigg( t_{m}^{\pm} \fr{e^{2\pi n/N^{2}-\pi m/p}}{kN} \bigg) \\ 
        & \qquad \ll n^{\nfr{1}{3}}e^{2\pi n/N^{2}} N^{-1} \sum_{k \leq N} k^{-\nfr{1}{3}+\xe} \lf(\sum_{m \geq 1} t_{m}^{\pm} e^{-\pi m/p} \rh) \\
        & \qquad \ll n^{\nfr{1}{3}} e^{2\pi n/N^{2}} \left|T^{\pm}(e^{-\pi /p})\right| N^{-\nfr{1}{3}+\xe}.
    \end{align*}

Returning to \eqref{eq:p2(n,chi,N)IntegralSum} then, the sum over $m \geq 1$ there will be part of an error term, and it remains only to consider the $m=0$ term. Observing that $t_{0}^{\pm} = 1$, and writing
    \(
    \label{eq:Lk(n)}
        \fL_{k}(n) = \fL_{k}(n,0) = \psum_{h \md{k}} \xf_{h,k} \exp(-2\pi inh/k),
    \)
equation \eqref{eq:p2(n,chi,N)IntegralSum} reduces to
{%
\setlength{\multlinegap}{25pt}
    \begin{multline}
        \fp_{(2)}(n,\xc;N) 
            = \sumNTwo \xl_{k} \fL_{k}(n) \int_{-\xvq'}^{\xvq''} \exp\left\{\fr{2\pi}{K}\lf(\fr{p-1}{24}\rh)z^{-1} + \fr{2\pi\~{n}}{k}z\right\} \,d\xf \\
            + O\lf(n^{\nfr{1}{3}} e^{2\pi n/N^{2}} N^{-\nfr{1}{3}+\xe}\rh) \notag.
    \end{multline}
}
Then, following the standard arguments of, e.g., \cites{Rademacher:PartitionFunction,andrews1976partitions}, we make the change of variable 
    \[
        \xo := z/k = N^{-2} - i\xf,
    \] 
we let $\mathcal{R}(N,k)$ be the positively oriented rectangle about the origin having vertices $\{\pm N^{-2}-i\xvq',\pm N^{-2}+i\xvq''\}$, and we find that 
    \[
    \begin{aligned}
        \int_{-\xvq'}^{\xvq''} \exp\left\{\fr{2\pi}{K} \lf(\fr{p-1}{24}\rh) z^{-1} + \fr{2\pi\~{n}}{k}z\right\} \,d\xf
    &   = \int_{\mathcal{R}(N,k)} \exp\left\{
        \fr{2\pi}{k^{2}} \lf(\fr{p-1}{24p}\rh) \xo^{-1} + 2\pi\~{n}\xo
        \right\} \,d\xo \\[0.3em]
    &   \qquad \qquad \qquad \qquad \quad 
    + O\left(n^{\nfr{1}{3}}N^{-\nfr{1}{3}+\xe}e^{2\pi n/N^{2}}\right).
    \end{aligned}
    \]
Using a classical formula\footnote{See Lemma \ref{lem:BesselIntegral} at the end of this section.} concerning the \emph{modified Bessel function} $I_{n}(z)$, we find that
    \[
        \fp_{(2)}(n,\xc;N) = \rt{\tf{1}{24}(1-\tf{1}{p})\~{n}^{-1}} \sum_{\substack{k \leq N \\[0.1em] (k,2p)=2 }} 
        \fr{\xl_{k}}{k} \fL_{k}(n) \, 
        I_{1}\?\Big({\fr{4\pi}{k} \rt{\tf{1}{24}(1-\tf{1}{p})\~{n}}}\,\Big) + O(N^{-\fr{1}{3}+\xe}),
    \]
where again $\~{n} = n + \tf{p-1}{24}$; to reduce clutter, we make the following definition.

\begin{definition}
    For prime $p$, let 
        \(
        \label{eq:kappa(p)}
            \xk_{p} := \pi \sqrt{\tf{2}{3}(1-\tf{1}{p})}.
        \)
\end{definition}

Taking $N \to \infty$ and recalling that $\fp_{(2)}(n,\xc;N) \to \fp_{(2)}(n,\xc)$ by definition, we summarize our derivations thus far with the following lemma.

\begin{lemma}
\label{lem:p(2)(n,chi)}
    For $p < 24$ with $p \equiv 1 \mod{4}$, one has
        \(
        \label{eq:p(2)(n,chi)}
            \fp_{(2)}(n,\xc) = \fr{\xk_{p}}{4\pi\sqrt{\~{n}}}
                \sum_{\substack{k = 1 \\[0.1em] (k,2p)=2 }}^{\infty} \fr{\xl_{k}}{k} \fL_{k}(n) I_{1}\?\Big(\fr{\xk_{p}}{k}\rt{\~{n}}\Big),
        \)
    where $\~{n}=n+\tf{p-1}{24}$, $\xl_{k}$ is given in \eqref{eq:lambdaEven}, $\xk_{p}$ is given in \eqref{eq:kappa(p)}, and 
        \[
            \fL_{k}(n) = \psum_{h \md{k}} \xf_{h,k} \exp(-2\pi inh/k).
        \]
\end{lemma}

Turning now to $\fp_{(1)}(n,\xc;N)$, when $(k,2p)=1$ nearly identical computations lead us to the following formula for $\lim_{N\to\infty} \fp_{(1)}(n,\xc;N)$, i.e., for $\fp_{(1)}(n,\xc)$.

\begin{lemma}
\label{lem:p(1)(n,chi)}
    For $p < 24$ with $p \equiv 1 \mod{4}$, one has
        \(
        \label{eq:p1(n,chi)SumReduced}
            \fp_{(1)}(n,\xc) = \fr{\xk_{p}}{4\pi\sqrt{\~{n}}} 
                \sum_{\substack{k = 1 \\[0.15em] (k,2p)=1}}^{\infty} \fr{\xl_{k}}{2k} \fL_{k}(n) I_{1}\?\Big(\fr{\xk_{p}}{2k}\sqrt{\~{n}}\Big),
        \)
    where $\~{n} = n + \tf{p-1}{24}$, $\xl_{k}$ is given in \eqref{eq:lambdaOdd}, $\xk_{p}$ is given in \eqref{eq:kappa(p)}, and
        \[
            \fL_{k}(n) = \psum_{h \md{k}} \xf_{h,k} \exp(-2\pi inh/k).
        \]
\end{lemma}

} 

\subsection{Series formulae for \tops{$\fp_{(p)}(n,\xc)$}{𝔭(p)(n,χ)} and \tops{$\fp_{(2p)}(n,\xc)$}{𝔭(2p)(n,χ)}}
\label{sec:p(p)-and-p(2p)}

We now examine $\fp_{(p)}$ and $\fp_{(2p)}$, beginning with $\fp_{(2p)}(n,\xc;N)$. First, we write
    \[
        R^{\pm}(x) = \sum_{m=0}^{\infty} \xr_{m}^{\pm} x^m,
    \]
where we recall from \eqref{eq:FEQ(2p)(Parent)} that the ``$\pm$'' here is determined by $\xc_{h}$, not $\xc_{k}$. In addition, we recall from \eqref{eq:x'} that the ``transformed point'' $\~{x}$ in \eqref{eq:FEQ(2p)(Parent)} is 
    \[
        x' = \exp({-2\pi i\nfr{\bar{h}}{k}-\nfr{2\pi}{kz}}) \qquad (h\bar{h} \equiv 1 \mod{k}).
    \]

Repeating our arguments from section \ref{sec:RadeTermsp1p2} then, we quickly find that
    \[
    \begin{aligned}
    \fp_{(2p)}(n,\xc;N) =
        \sum_{\substack{k \leq N \\[0.15em] (k,2p)=2p}} 
        \psum_{h \md{k}} \sum_{m=0}^{\infty} & \xl_{k} \xr_{m}^{\xc_{h}}\xf_{h,k} \exp\!\big({-2\pi i(nh+m\bar{h})/k}\big) \\
    & \times \int_{-\xvq'}^{\xvq''} \exp\!\bigg({-\fr{2\pi}{k}\lf(\fr{p-1}{24}+m\rh)z^{-1} + \fr{2\pi\~{n}}{k}z}\bigg) \,d\xf,
    \end{aligned}
    \]
where again $\~{n} = n + \fr{p-1}{24}$. As $\xc_{h}$ depends only on $h \mod{p}$, we group the sum over $h \mod{k}$ by residue classes modulo $p$ to write
    \(
    \label{eq:p2p-ArcIntegralSum}
    \begin{aligned}
        \fp_{(2p)}(n,\xc) = \sum_{\substack{k \leq N \\[0.15em] (k,2p)=2p}} & \psum_{d \md{p}} \sum_{m=0}^{\infty} \xl_{k} \xr_{m}^{\xc_{d}} \fL_{k}(n,m;d) \\ 
        & \qquad \times \int_{-\xvq'}^{\xvq''} \exp\!\bigg({-\fr{2\pi}{k}\lf(\fr{p-1}{24}+m\rh)z^{-1} + \fr{2\pi\~{n}}{k}z}\bigg) \,d\xf,
    \end{aligned}
    \)
where
    \(
    \label{eq:Lk(n,m;d)(2p)}
        \goL_{k}(n,m;d) := \psum_{\substack{h \md{k} \\ h \equiv d \mod{p}}} \xf_{h,k}\exp\!\big({-2\pi i(nh+m\bar{h})/k}\big) \qquad (\text{when $2p \mid k$}),
    \)
In section \ref{sec:KloostermanBounds} (specifically Lemma \ref{lem:Lk(n,m;d)(2p)(Bound)}), we find that
    \[
        \fL_{k}(n,m;d) \ll n^{1/3}k^{2/3+\xe} \qquad \text{when $2p \mid k$},
    \]
\emph{uniformly in} $d$, cf.~inequality \eqref{eq:Lk(n,m)(2)Bound}. 
Because the coefficient of $z^{-1}$ in \eqref{eq:p2p-ArcIntegralSum} is always negative, it follows from arguments like those in section \ref{sec:RadeTermsp1p2} that 
    \[
        \fp_{(2p)}(n,\xc;N) \ll n^{\nfr{1}{3}} e^{2\pi n/N^{2}} N^{-1} \sum_{k \leq N} k^{-\nfr{1}{3}+\xe} \ll N^{-\nfr{1}{3}+\xe}.
    \]
Taking $N \to \infty$, we establish the following lemma (which does not require that $p < 24$).

\begin{lemma}
\label{lem:p2p(n,chi)}
    For $p \equiv 1 \mod{4}$, one has
        \[
            \fp_{(2p)}(n,\xc) = 0.
        \]
\end{lemma}

Turning at last to $\fp_{(p)}(n,\xc;N)$, we expand the series $S^{\pm}(x)$ from \eqref{eq:FEQ(p)(Parent)} as
    \[
        S^{\pm}(x) = \sum_{m=0}^{\infty} \xs_{m}^{\pm} x^m,
    \]
and recall from \eqref{eq:y'} that the transformed point $\~{x}$ in this case is 
    \[
        y' = \exp\lf(-2\pi i(\ol{2h})/k - \pi /kz\rh),
    \]
where $2h\ol{2h} \equiv 1 \mod{k}$. In addition, we recall that
    \[
        B_{2}(\xc) := \fr{1}{p} \sum_{\mu = 1}^{p-1} \mu^{2}\xc_{\mu}.
    \]
Arguments similar to those in the previous cases lead us to deduce that
    \(
    \label{eq:ArcSum-p}
    \begin{aligned}
        \fp_{(p)}(n,\xc;N) = 
            \sum_{\substack{k \leq N \\[0.15em] (k,2p)=p}} 
            & \psum_{\substack{d \md{p}}} 
            \sum_{m=0}^{\infty} \xs_{m}^{\xc_{d}} \fL_{k}(n,m;d) \\
        & \qquad \times \int_{-\xvq'}^{\xvq''} \exp\!\bigg({\fr{2\pi}{k}\lf(\fr{c_{m,d}}{4}\rh)z^{-1} + \fr{2\pi\~{n}}{k}z}\bigg) \,d\xf,
    \end{aligned}
    \)
where again $\~{n} = n + \fr{p-1}{24}$, where
    \(
    \label{eq:Lk(n,m;d)(p)}
        \fL_{k}(n,m;d) := \psum_{\substack{h \md{k} \\ h \equiv d \mod{p}}} \xf_{h,k} \exp\!\big({-2\pi i(nh+m\ol{2h})/k}\big) \qquad (\text{when $(k,2p)=p$}),
    \)
and where
    \(
    \label{eq:c(m,d)}
        c_{m,d} := \xc_{d}\lf(1-\fr{\xc_{2}}{4}\rh)B_2(\xc) - \fr{p-1}{24} - 2m.
    \)

To separate out our error term in \eqref{eq:ArcSum-p}, we must determine the sign of $c_{m,d}$ as a function of $m$ and $d$. 
Letting $L(s,\chi)$ denote the usual \emph{Dirichlet $L$-function}, we note\footnote{See, e.g., \cite{Montgomery2007multiplicative}*{Ex.~10.1.14}} that $B_{2}(\xc) = \pi^{-2}p^{\fr{3}{2}}L(2,\xc)$ for $p \equiv 1 \mod{4}$, and use the crude lower bound $L(2,\xc) \geq \fr{\xz(4)}{\xz(2)}(1+p^{-2}) = \fr{\pi^{2}}{15}(1+p^{-2})$, to deduce that
    \[
        \lf(1-\fr{\xc_{2}}{4}\rh)B_{2}(\xc) - \fr{p-1}{24} > 0 \qquad \text{for $p \equiv 1 \mod{4}$},
    \]
even without the restriction that $p < 24$.
Thus $c_{m,d} < 0$ if $\xc_{d} = -1$, so that the nonquadratics (mod $p$) contribute only error terms to \eqref{eq:ArcSum-p}. Moreover, since $c_{m,d}$ is constant across $d$ such that $\xc_{d}=1$, we let
    \[
        c_{m} := c_{m,1} = \lf(1-\fr{\xc_{2}}{4}\rh)B_2(\xc) - \fr{p-1}{24} - 2m,
    \] 
and ``wrap up'' the sum over quadratics $d \mod{p}$ in \eqref{eq:ArcSum-p} to write
{%
\setlength{\multlinegap}{25pt}
    \begin{multline}
        \fp_{(p)}(n,\xc;N) = 
            \sum_{\substack{k \leq N \\[0.15em] (k,2p)=p}}  
            \sum_{m=0}^{\infty} \xs_{m}^{+} \fL_{k}^{+}(n,m) \int_{-\xvq'}^{\xvq''} \exp\!\bigg({\fr{2\pi}{k}\lf(\fr{c_{m}}{4}\rh)z^{-1} + \fr{2\pi\~{n}}{k}z}\bigg) \,d\xf \\
         + O\lf(n^{\nfr{1}{3}}e^{2\pi n/N^{2}}N^{-\nfr{1}{3}+\xe}\rh),\notag 
    \end{multline}
}
where
    \[
        \fL_{k}^{+}(n,m) := \psum_{\substack{d \md {p} \\ \xc_{d} = 1}} \fL_{k}(n,m;d) = \psum_{\substack{h \md{k} \\ \xc_{h} = 1}} \xf_{h,k} \exp\!\big({-2\pi i(nh+m\ol{2h})/k}\big).
    \]

Accepting for the moment that $\fL_{k}^{+}(n,m) \ll n^{1/3}k^{2/3+\xe}$ when $(k,2p)=p$, we argue as is done for \eqref{eq:p(2)(n,chi)} and \eqref{eq:p1(n,chi)SumReduced}, and take $N \to \infty$ to establish the following lemma.

\begin{lemma}
\label{lem:pp(n,chi)}
For $p < 24$ satisfying $p \equiv 1 \mod{4}$, one has
    \(
    \label{eq:p(p)(n,chi)}
        \fp_{(p)}(n,\xc) 
            = \fr{1}{\rt{\~{n}}}\sum_{\substack{k=1 \\[0.1em] (k,2p)=p}}^{\infty}
            \sum_{\substack{m = 0 \\ c_{m} \geq 0}}^{\infty} \fr{\rt{c_{m}}}{2k} \xs_{m}^{+} \fL_{k}^{+}(n,m) \mspace{2mu}I_{1}\mspace{-2mu}\Big(\fr{2\pi}{k}\sqrt{c_{m}\~{n}}\Big),
    \)
where 
    \[
        \~{n} = n + \fr{p-1}{24}, \qquad
        c_{m} = \lf(1-\fr{\xc_{2}}{4}\rh)B_2(\xc) - \fr{p-1}{24} - 2m,
    \]

    \[
       \fL_{k}^{+}(n,m) 
       = \psum_{\substack{h \md{k} \\ \xc_{h} = 1}} 
       \xf_{h,k} \exp\!\big({-2\pi i(nh+m\ol{2h})/k}\big),
    \]
with $2h\ol{2h} \equiv 1 \mod{k}$, and with $\xs^{+}_{m}$ defined via the series
    \[
        S^{+}(x) = \prod_{\bfr} (x^{2r},x^{2p-2r};x^{2p})_{\infty}^{-1} 
        \prod_{\bfs} (x^{p+2s},x^{p-2s};x^{2p})_{\infty}^{-1} 
        = \sum_{m=0}^{\infty} \xs_{m}^{+} x^{m}.
    \]
\end{lemma}

\begin{remark}
    For $p < 24$, the sum over $m$ in \eqref{eq:p(p)(n,chi)} only runs over $0 \leq m \leq 2$, but as $p$ grows, this range grows substantially; indeed, the relation $B_{2}(\xc) = \pi^{-2}p^{\fr{3}{2}} L(2,\xc)$ and the crude bounds
        \[
            \tf{\pi^{2}}{15}(1+p^{-2}) \leq L(2,\xc) \leq \tf{\pi^2}{6}(1-p^{-2})
        \]
    show that, as $p \to \infty$, one has
        \[
            \lf(1-\fr{\xc_{2}}{4}\rh)B_{2}(\xc) - \fr{p-1}{24} \asymp p^{\fr{3}{2}}.
        \]
    Thus, for large $p$ we expect that the sum over $m$ such that $c_{m} > 0$ will have ``length'' on the order of $p^{3/2}$.
\end{remark}

We finish this section with the following  classical\footnotemark~fact (normalized for our purposes) about the modified Bessel function $I_{n}(z)$.
\footnotetext{Niven \cite{Niven:OnACertainPartitionFunction}*{p.~360} provides a convenient reference---though he cites Watson \cite{Watson:TheoryBesselFunctions}*{pp.~79 and 181}.}

\begin{lemma}
    \label{lem:BesselIntegral}
    Let $a$ and $b$ be real and nonnegative, and let $\mathcal{R}$ be a positively oriented rectangle about the origin. Then
        \[
            \fr{1}{2\pi i} \int_{\mathcal{R}} \exp\left(\fr{2\pi b}{k^2}\xo^{-1} + 2\pi a\xo\right) \,d\xo = %
            \begin{cases}
                \rt{b/a}\mspace{3mu}I_{1}\mspace{-2mu}\big(\fr{4\pi}{k}\rt{ab}\big)/k, & a > 0, \\
                4\pi b/k^2, & a=0.
            \end{cases}
        \]
\end{lemma}

\section{Some crude bounds on \tops{$\fL_{k}$}{\fL k} and \tops{$\fL_{k}^{+}$}{\fL k+}} 
\label{sec:KloostermanBounds}

In this section we show that the sums $\fL_{k}(n)$ and $\fL_{k}^{+}(n,m)$ defined in \eqref{eq:Lk(n)}, \eqref{eq:Lk(n,m;d)(2p)}, and \eqref{eq:Lk(n,m;d)(p)} can be expressed as \emph{Kloosterman sums} in the usual sense. We recall that $\bfr$ and $\bfs$ are the subsets of $\{1,\ldots,\fr{p-1}{2}\}$ containing the quadratics and nonquadratics \mbox{(mod $p$)}, respectively, and that sums over $\bfr$ and $\bfs$ are taken as $r$ and $s$ run over $\bfr$ and $\bfs$, respectively. We recall that $\{\mu\}_{p}$ is defined so that
    \[
        \mu \equiv \{\mu\}_{p} \mod{p} \qquad\text{and}\qquad 0 \leq \{\mu\}_{p} < p,
    \]
and that
    \[
        q = \fr{p-1}{2}.
    \]

Because Dedekind sums and Kloosterman sums are well studied, and the arguments in our four (technically eight) cases conditioned on $\mathrm{gcd}(k,2p)$ are all quite similar to one another (and to those in \cites{Lehner:PartitionsMod5,Hagis:PartitionsPrime,Hagis:DistinctSummands}), we only demonstrate the desired bounds on $\fL_{k}(n,m;d)$ for the case $2p \mid k$, namely for
    \(
    \label{eq:Lk(n,m;d)(2p)-Recall}
        \fL_{k}(n,m;d) = \psum_{\substack{h \md{k} \\ h \equiv d \mod{p}}} \xf_{h,k} 
        \exp\!\big({-2\pi i(nh+m\bar{h})/k}\big).
    \)

\begin{remark}
    The methods and results on Kloosterman sums we use in this section are far from ``best-known''. These ``cruder'' methods are used for their simplicity, and to follow the works of Lehner \cite{Lehner:PartitionsMod5} and Hagis Jr.~\cites{Hagis:PartitionsPrime,Hagis:DistinctSummands}.
\end{remark}

Suppose that $2p \mid k$, and let $12p = fG$, where $f$ is the largest divisor of $12p$ which is coprime to $k$, and note that $p \mid G$ by assumption. Define $\xj$ and $\xG$ via
    \[
        f\xj \equiv 1 \mod{Gk} \qquad\text{and}\qquad Gk\xG \equiv 1 \mod{f}.
    \]
Then, per \cite{Hagis:DistinctSummands}*{pp.~405--407} and the fact that $\xo_{\bfs}(2h,k) = \xo_{\bfs}(h,k/2)$, we have
    \begin{subequations}
    \(
    \label{eq:phi(h,k)(2p)(Kloost)}
    \begin{aligned}
        \xf_{h,k} &= \xo_{\bfr}(h,k)\fr{\xo_{\bfs}(h,k/2)}{\xo_{\bfs}(h,k)} = \exp\lf\{{2\pi i\Big[
            \fr{\xj}{G}(3pq - 6\tsum_{\bfr}\{hr\}_{p}) + \fr{\xj}{Gk}(uh + v\bar{h})
        \Big]}\rh\},
    \end{aligned}
    \tag{\theparentequation}
    \)
where $h\bar{h} \equiv 1 \mod{Gk}$, and
    \begin{align}
    \label{eq:u(2p)}
        u &= -2pq + k\big(6\tsum_{\bfr}r + qk - 3pq\big) \\
    \label{eq:v(2p)}
        v &= - q(\tf{1}{2}k^{2}+2p).
    \end{align}
    \end{subequations}

\begin{lemma}
    \label{lem:Lk(n,m;d)(2p)(Bound)}
    Let $p \equiv 1 \mod{4}$, let $2p \mid k$, and let $(d,p)=1$. One has
        \[
            \fL_{k}^{+}(n,m;d) \ll n^{\nfr{1}{3}}k^{\nfr{2}{3}+\xe},
        \]
    uniformly in $d$.
\end{lemma}

\begin{proof}
    From \eqref{eq:phi(h,k)(2p)(Kloost)} we see that 
        \[
            \xf_{h,k} = \xz_{h,k} \exp\lf({\xj(uh+v\bar{h})\fr{2\pi i}{Gk}}\rh),
        \]
    where $\xz_{h,k}$ only depends on the residue class of $h \mod{p}$, say $h \equiv d \mod{p}$. Since $d$ and $k$ are fixed, we let $\xz_{d} := \xz_{h,k}$ for the remainder of the proof. Because $\xf_{h,k}$ and the exponentials in \eqref{eq:Lk(n,m;d)(2p)-Recall} are $k$-periodic in $h$ (by their respective definitions), we extend the summation range in \eqref{eq:Lk(n,m;d)(2p)-Recall} to $h \mod{Gk}$, and account for this by multiplying by $1/G$.
    Hence
        \(
        \label{eq:Lk(n,m;d)Kloosterman}
            \fL_{k}(n,m;d) = \fr{\xz_{d}}{G} 
            \psum_{\substack{h \md{Gk} \\ h \equiv d \mod{p}}} 
            \exp\!\bigg({\Big[ (\xj u-Gn)h + (\xj v-Gm)\bar{h} \Big]\fr{2\pi i}{Gk}}\bigg),
        \)
    with $u$ and $v$ as in \eqref{eq:u(2p)} and \eqref{eq:v(2p)}, respectively.
    
    By a standard argument using the relation 
        \[
            \fr{1}{p} \sum_{a=1}^{p} \exp(2\pi i(x-y)a/p) = \begin{cases}
                1, & x \equiv y \mod{p}, \\
                0, & \text{otherwise},
            \end{cases} 
        \]
    the restriction $h \equiv d \mod{p}$ in \eqref{eq:Lk(n,m;d)Kloosterman} is removed at the cost of a constant factor, so that
        \[
            \fL_{k}(n,m;d) \ll \Bigg|\,\psum_{h \md{Gk}} \exp\lf({\Big[ (\xj u-Gn)h + (\xj v-Gm)\bar{h} \Big]\fr{2\pi i}{Gk}}\rh)\?\Bigg|.
        \]
    As this latter sum is a ``complete'' Kloosterman sum, Sali\'e's bound \cite{Salie:ZurAbschatzung} implies that
        \(
        \label{eq:Lk(n,m;d)(2p)(Bound)}
            \fL_{k}(n,m;d) \ll (\xj u-Gn,Gk)^{\nfr{1}{3}}k^{\nfr{2}{3}+\xe}.   
        \)

    As $(f,Gk)=1$, $fG=12p$, and $f\xj \equiv 1 \mod{Gk}$, we observe that
        \[
            (\xj u-Gn,Gk) = (f\xj u-fGn,Gk) = (u - 12pn,Gk) \ll (u-12pn,k).
        \]
    In addition, from \eqref{eq:u(2p)} we see that $u = -2pq + kw$ for some integer $w$, so that 
        \[
            (u-12pn,k) = (-2pq-12pn,k) \ll n,
        \]
    and the result follows immediately from this and \eqref{eq:Lk(n,m;d)(2p)(Bound)}.
\end{proof}

\section{Some lemmata on \tops{$\xl_k$}{λk}}
\label{sec:lambda(k)}

We now derive some formulae for the cosecant quantities $\xl_{k}$ from \eqref{eq:lambdaEven}, \eqref{eq:lambdaOdd}, and \eqref{eq:lambdaP}. In this section we let that $k\bar{k} \equiv 1 \mod{p}$ when $(k,p)=1$, and we do not require that $p < 24$, only that $p \equiv 1 \mod{4}$.  Again using $\bfr$ and $\bfs$ to denote the sets of quadratics and nonquadratics (mod $p$) from the range $\{1,\ldots,\fr{p-1}{2}\}$, respectively, we recall that
    \begin{subequations}
    \(
    \label{eq:lambdaOdd(Recall)}
        \xl_{k} = 2^{-\fr{p-1}{4}} \prod_{\bfr} \csc(\pi \|\bar{k}r\|/p) \qquad \text{if $(k,2p)=1$,}
    \)
    \(
    \label{eq:lambdaEven(Recall)}
        \xl_{k} = 2^{-\fr{p-1}{4}} \prod_{\bfr} \csc(\pi\|\bar{k}r\|/p) \prod_{\bfs} \fr{\csc(\pi\|2\bar{k}s\|/p)}{\csc(\pi\|\vphantom{\hat{h}}\bar{k}s\|/p)} \qquad \text{if $(k,2p)=2$},
    \)
and that
    \[
        \xl_{k} = 1 \qquad \text{if $p \mid k$}.
    \]
    \end{subequations}

For fixed $p$ let
    \[
        Q_{\bfr} := 2^{-\fr{p-1}{4}} \prod_{\bfr} \csc(\pi r/p) \qquad\text{and}\qquad
        Q_{\bfs} := 2^{-\fr{p-1}{4}} \prod_{\bfs} \csc(\pi s/p),
    \]
and define $Q$ via
    \[
        Q = \fr{Q_{\bfr}^2}{Q_{\bfs}^2} = \fr{\prod_{\bfs} \sin^{2}(\pi s/p)}{\prod_{\bfr} \sin^{2}(\pi r/p)}.
    \]

\begin{remark}
    \label{rem:cyclotomy-Q}
    Since $\sin(\pi\mu/p) = \sin(\pi(p-\mu)/p)$ for integer $1 \leq \mu < p$, we have
        \[
            Q = \fr{\prod_{\bfs} \sin(\pi s/p)\sin(\pi(p-s)/p)}{\prod_{\bfr} \sin(\pi r/p)\sin(\pi(p-r)/p)},
        \]
    and since $\bfr \cup (p-\bfr)$ and $\bfs \cup (p-\bfs)$ exhaust the quadratics and nonquadratics (mod $p$) in the full range $\{1,\ldots,p-1\}$, respectively, we see that $Q$ is equal to a classical quantity from elementary number theory; see, e.g., \cite{Davenport:MultiplicativeNT}*{p.~10 ff.}.
\end{remark}

Because $\xl_{k} = 1$ whenever $p \mid k$, let $(k,p)=1$. From \eqref{eq:lambdaOdd(Recall)}, evidently
    \(
    \label{eq:lambda(k)(5,Odd)}
        \xl_{k} = \begin{cases}
            Q_{\bfr}, & \xc_{k}=1, \\
            Q_{\bfs}, & \xc_{k}=-1,
        \end{cases} \qquad \text{for odd $k$.}
    \)
For even $k$, we must separately consider the cases $p \equiv 5 \mod{8}$ and $p \equiv 1 \mod{8}$.

\begin{lemma}
    Let $p \equiv 5 \mod{8}$, and let $k = 2^{\nu}m$ with $\nu > 0$ and $(m,p)=1$. Then
        \[
            \xl_{k} = \begin{cases}
                Q^{\fr{1}{2}\xc_{m}}\xl_{m}, & \text{$\nu$ odd,} \\
                Q^{2\xc_{m}}\xl_{m}, & \text{$\nu$ even}.
            \end{cases}
        \]
\end{lemma}

\begin{proof}
    We first show that if $(k,p)=1$, then
        \(
        \label{eq:lambda2kratio}
            \xl_{2k}/\xl_k = \begin{cases}
                Q^{\fr{1}{2}\xc_{k}}, & \text{$k$ odd}, \\
                Q^{-\fr{3}{2}\xc_{k}}, & \text{$k$ even}.
            \end{cases}
        \)
    Let $k$ be odd. Because $2$ is a nonquadratic (mod $p$) when $p \equiv 5 \mod{8}$, from equation \eqref{eq:lambdaEven(Recall)} it follows that 
        \[
            \xl_{2k} = 2^{-\fr{p-1}{4}} \bigg(\prod_{\bfr} \csc(\pi \|\bar{k}r\|/p)\bigg)^{2} \prod_{\bfs} \sin(\pi\|\bar{k}s\|/p),
        \]
    and so, by \eqref{eq:lambdaOdd(Recall)} and \eqref{eq:lambdaEven(Recall)} together, we have
        \[
            \fr{\xl_{2k}}{\xl_{k}} = \fr{\prod_{\bfs} \sin(\pi\|\bar{k}s\|/p)}{\prod_{\bfr} \sin(\pi\|\vphantom{\hat{k}}\bar{k}r\|/p)}.
        \]
    Since $x \mapsto \|\bar{k}x\|$ either fixes or exchanges $\bfr$ and $\bfs$, dependent on whether $k$ is a quadratic or nonquadratic (mod $p$), we see that 
        \[
            \xl_{2k}/\xl_{k} = Q^{\fr{1}{2}\xc_{k}} \qquad \text{for $k$ odd}.
        \]
    Following similar arguments, we quickly find that
        \[
            \xl_{2k}/\xl_{k} = Q^{-\fr{3}{2}\xc_{k}} \qquad \text{for $k$ even},
        \]
    which establishes \eqref{eq:lambda2kratio}.

    Equations \eqref{eq:lambda(k)(5,Odd)} and \eqref{eq:lambda2kratio} let us build up formulae for $\xl_{k}$ for even $k$ by reducing to the \emph{odd part} of $k$: Namely, beginning with $k$ odd and coprime to $p$, we find that
        \[
            \xl_{2k} = Q^{\fr{1}{2}\xc_{k}}\xl_{k}, \quad 
            \xl_{4k} = Q^{-\fr{3}{2}\xc_{2k}}\xl_{2k} = Q^{2\xc_{k}}\xl_{k}, \quad
            \xl_{8k} = Q^{-\fr{3}{2}\xc_{4k}}\xl_{4k} = Q^{\fr{1}{2}\xc_{k}}\xl_{k},
        \]
    and so forth, and the result follows.
\end{proof}

\begin{lemma}
    Suppose that $p \equiv 1 \mod{8}$ and let $(k,p)=1$. One has
        \[
            \xl_{k} = \begin{cases}
                Q_{\bfr}, & \xc_{k}=1, \\
                Q_{\bfs}, & \xc_{k}=-1.
            \end{cases}
        \]
\end{lemma}

\begin{proof}
    Because $p \equiv 1 \mod{8}$, any $\mu$ is a nonquadratic (mod $p$) if and only if $2\mu$ is a nonquadratic (mod $p$). Thus, in \eqref{eq:lambdaEven(Recall)} we deduce that
        \[
            \prod_{\bfs} \fr{\csc(\pi\|2\bar{k}s\|/p)}{\csc(\pi\|\vphantom{\hat{h}}\bar{k}s\|/p)} = \prod_{\bfs} \fr{\csc(\pi\|\bar{k}s\|/p)}{\csc(\pi\|\vphantom{\hat{h}}\bar{k}s\|/p)} = 1,
        \]
    whereby
        \[
            \xl_{k} = 2^{-\fr{p-1}{4}} \prod_{\bfr} \csc(\pi \|\bar{k}r\|/p) \qquad \text{for $k$ odd \emph{or} even,}
        \]
    and the result follows at once.
\end{proof}

\begin{corollary}
    \label{cor:lambda(2k)(1mod8)}
    If $p \equiv 1 \mod{8}$ and $(k,p)=1$, then $\xl_{2k}=\xl_{k}$. 
\end{corollary}

We close this section with a few explicit examples of $Q_{\bfr}$, $Q_{\bfs}$ and $Q$, from which $\xl_{k}$ can be computed using the preceding lemmata. For $p=5$, one has
    \[
        Q_{\bfr} = \sqrt{\tfr{1}{2}(1+\tfr{1}{\sqrt{5}})}, \quad
        Q_{\bfs} = \sqrt{\tfr{1}{2}(1-\tfr{1}{\sqrt{5}})}, \quad\text{and}\quad
        Q = \tf{3+\sqrt{5}}{2}.
    \]
For $p=13$, one has
    \[
        Q_{\bfr} = \sqrt{\tf{1}{2}(1+\tf{3}{\sqrt{13}})}, \quad
        Q_{\bfs} = \sqrt{\tf{1}{2}(1-\tf{3}{\sqrt{13}})}, \quad\text{and}\quad
        Q = \tf{11+3\sqrt{13}}{2}.
    \]
For $p=17$, one has
    \[
        Q_{\bfr} = \sqrt{1+\tf{4}{\vphantom{\hat{k}}\sqrt{17}}}, \quad
        Q_{\bfs} = \sqrt{1-\tf{4}{\vphantom{\hat{k}}\sqrt{17}}}, \quad\text{and}\quad
        Q = 33+8\sqrt{17}.
    \]

\section{The full series for \tops{$\fp(n,\xc)$}{𝔭(n,χ)}}
\label{sec:rademacherSeries}

We now give an exact series formula for $\fp(n,\xc)$, recalling that 
    \[
        \fp(n,\xc) = \fp_{(1)}(n,\xc) + \fp_{(2)}(n,\xc) + \fp_{(p)}(n,\xc) + \fp_{(2p)}(n,\xc),
    \]
where each $\fp_{(j)}(n,\xc)$ includes only terms coming from reduced rationals $\fr{h}{k} \in (0,1]$ with $(k,2p)=j$. 
In section \ref{sec:RadeTermsp1p2}, we found that
    \begin{align}
    \label{eq:p1(n,chi)-Recall}
    \fp_{(1)}(n,\xc) 
        &= \fr{\xk_{p}}{4\pi\sqrt{\~{n}}} 
            \sum_{\substack{k=1 \\[0.15em] (k,2p)=1}}^{\infty} 
            \fr{\xl_{k}}{2k} \fL_{k}(n) I_{1}\?\Big(\fr{\xk_{p}}{2k}\sqrt{\~{n}}\Big), \\
    \label{eq:p2(n,chi)-Recall}
    \fp_{(2)}(n,\xc) 
        &= \fr{\xk_{p}}{4\pi\sqrt{\~{n}}} 
            \sum_{\substack{k=1 \\[0.15em] (k,2p)=2}}^{\infty} 
            \fr{\xl_{k}}{k} \fL_{k}(n) I_{1}\?\Big(\fr{\xk_{p}}{k}\rt{\~{n}}\Big),
    \end{align}
where the $\xl_k$ are determined as in section \ref{sec:lambda(k)},
    \[
        \~{n} = n + \fr{p-1}{24}, \qquad 
        \xk_{p} = \pi\sqrt{\tf{2}{3}(1-\tf{1}{p})},
    \]
and
    \[
        \fL_{k}(n) = \psum_{h \md{k}} \xf_{h,k} \exp(-2\pi inh/k) \qquad \text{for $(k,p)=1$}.
    \]
Next, we recall from Lemma \ref{lem:pp(n,chi)} that
    \[
        \fp_{(p)}(n,\xc) 
            = \fr{1}{\rt{\~{n}}}
            \sum_{\substack{k = 1 \\[0.1em] (k,2p)=p}}^{\infty}
            \sum_{\substack{m = 0 \\[0.1em] c_{m} > 0}}^{\infty} 
            \fr{\rt{c_{m}}}{2k} \xs_{m}^{+} \fL_{k}^{+}(n,m) \, I_{1}\? 
            \Big(\fr{2\pi}{k}\sqrt{c_{m}\~{n}}\Big),
    \]
where
    \[
        c_{m} = \lf(1-\fr{\xc_{2}}{4}\rh)B_2(\xc) - \fr{p-1}{24} - 2m,
    \]
the constants $\xs^{+}_{m}$ are defined via
    \[
        S^{+}(x) = \prod_{\bfr} (x^{2r},x^{2p-2r};x^{2p})_{\infty}^{-1} \prod_{\bfs} (x^{p+2s},x^{p-2s};x^{2p})_{\infty}^{-1} = \sum_{m=0}^{\infty} \xs_{m}^{+} x^{m},
    \]
and
    \[
       \fL_{k}^{+}(n,m) = \psum_{\substack{h \md{k} \\ \xc_{h} = 1}} \xf_{h,k} \exp\!\big({-2\pi i(nh+m\ol{2h})/k}\big) \qquad \text{for $(k,2p)=p$},
    \]
where $2h\ol{2h} \equiv 1\mod{k}$. Lastly, we recall from Lemma \ref{lem:p2p(n,chi)} that
    \[
        \fp_{(2p)}(n,\xc) = 0.
    \]

For $k$ such that $(k,2p)=1$, we note that the summands in \eqref{eq:p1(n,chi)-Recall} and \eqref{eq:p2(n,chi)-Recall} corresponding to $k$ and $2k$ have the same ``$I_{1}$-factor''. Thus, we may combine \eqref{eq:p1(n,chi)-Recall} and \eqref{eq:p2(n,chi)-Recall}, and then separate out the summands corresponding to $k$ such that $4 \mid k$ and $p \nmid k$, to arrive at the formula, valid for $p \equiv 1 \mod{4}$ with $p < 24$,
    \(
    \label{eq:p(n,chi)Full<24}
    \begin{aligned}
    \fp(n,\xc) 
        &   = \fp_{(1)}(n,\xc) + \fp_{(2)}(n,\xc) + \fp_{(p)}(n,\xc) \\ 
        &   = \fr{\xk_{p}}{4\pi\rt{\~{n}}}
            \sum_{\substack{k=1 \\[0.1em] (k,2p)=1}}^{\infty}
            \fr{1}{2k} \Big\{{\xl_{k}\fL_{k}(n) + \xl_{2k}\fL_{2k}(n)}\Big\} I_{1}\?\lf(\fr{\xk_{p}}{2k}\rt{\~{n}}\rh) \\
        &   \qquad + \fr{\xk_{p}}{4\pi\rt{\~{n}}}
            \sum_{\substack{k=1 \\[0.1em] 4\.|\.k,\, p \ssnmid k}}^{\infty} 
            \fr{\xl_{k}}{k} \fL_{k}(n) I_{1}\? \lf(\fr{\xk_{p}}{k}\rt{\~{n}}\rh) \\
        &   \qquad + \fr{1}{\rt{\~{n}}} 
            \sum_{\substack{k=1 \\[0.1em] (k,2p)=p}}^{\infty}\,
            \sum_{\substack{m=0 \\[0.1em] c_{m} \geq 0}}^{\infty} \fr{\sqrt{c_{m}}}{2k} \xs_{m}^{+} \fL_{k}^{+}(n,m) I_{1}\?\Big(\fr{2\pi}{k}\sqrt{c_{m}\~{n}}\Big).
    \end{aligned}
    \)

\subsection{The formula for \tops{$\fp(n,\tjac{\cdot}{17})$}{\fp(n,(.|17))}}

When $p=17$, we compute that
    \[
        B_{2}(\xc) = \fr{1}{p} \sum_{\mu=1}^{16} \mu^{2}\xc_{\mu} = 8,
    \]
whereby
    \[
        c_{m} = \lf(1-\fr{\xc_{2}}{4}\rh)B_2(\xc) - \fr{p-1}{24} - 2m = 2\B(\fr{8}{3}-m\B),
    \]
so that $c_{m} \geq 0$ only for $m=0,1,2$; in particular, we have 
    \[
        c_{0} = \tf{16}{3}, \quad c_{1} = \tf{10}{3}, \quad\text{and}\quad c_{2} = \tf{4}{3}.
    \]
Moreover, we have
    \[
        \xk_{17} = 4\pi \rt{\tf{2}{51}}, \qquad\text{and}\qquad \~{n} = n + \tf{2}{3}.
    \]
The constants $\xs_{m}^{+}$ can be directly computed for small $m$, and we find that
    \[
        S^{+}(x) = \sum_{m \geq 0} \xs_{m}^{+}x^{m} = 1 + x^{2} + x^{3} + 2x^{4} + 2x^{5} + 3x^{6} + 4x^{7} + 6x^{8} + \cdots,
    \]
whereby
    \[
        \xs_0^{+} = 1, \quad \xs_{1}^{+} = 0, \quad\text{and}\quad \xs_{2}^{+} = 1.
    \]

Incorporating these facts and the fact that $\xl_{2k} = \xl_{k}$ for $(k,p)=1$ into formula \eqref{eq:p(n,chi)Full<24}, for $\fp(n,\tjac{\cdot}{17})$ we at last arrive at the formula
    \(
    \label{eq:p(n,17)FullSeries}
    \begin{aligned}
    \fp(n,\tjac{\cdot}{17}) 
        &   = \rt{\tf{2}{51}\~{n}^{-1}} 
        \!\sum_{\substack{k=1 \\[0.1em] (k,2p)=1}}^{\infty} 
        \!\!\fr{\xl_{k}}{2k} \Big\{{\fL_{k}(n) + \fL_{2k}(n)}\Big\} I_{1}\mspace{-2mu} \Big({\tfr{2\pi}{k}\rt{\tfr{2}{51}\~{n}}}\,\Big) \\
        &   \qquad + \rt{\tfr{2}{51}\~{n}^{-1}} 
        \!\sum_{\substack{k=1 \\[0.1em] 4\.|\. k,\, p \ssnmid k}}^{\infty} 
        \!\fr{\xl_{k}}{k} \fL_{k}(n) I_{1}\mspace{-2mu} \Big(\tfr{4\pi}{k}\rt{\tfr{2}{51}\~{n}}\,\Big) \\
        &  \qquad + \rt{\tf{1}{3}\~{n}^{-1}} 
        \!\sum_{\substack{k=1 \\[0.1em] (k,2p)=p}}^{\infty} 
        \!\!\fr{1}{2k} \B\{ {
            4\fL_{k}^{+}(n,0) I_{1}\?\big(\tfr{8\pi}{3k}\rt{3\~{n}}\.\big) 
            + 2\fL_{k}^{+}(n,2) I_{1}\?\big(\tfr{4\pi}{3k}\rt{3\~{n}}\.\big) 
            }\B\}.
    \end{aligned}
    \)
Examining formula \eqref{eq:p(n,17)FullSeries}, because $I_{1}(t) \sim (2\pi t)^{-\fr{1}{2}}e^{t}$ as $t \to \infty$, and because the different $I_1$-terms have different arguments, we see that the vanishing of $\fp(n,\tjac{\cdot}{17})$ is equivalent to certain cancellations and vanishings among the factors $\fL_{k}(n)$ and $\fL_{k}^{+}(n,m)$. Examining $\fL_{k}$ and $\fL_{k}^{+}$ requires vocabulary and theory concerning Dedekind sums, and these topics occupy much of this paper's second half.

\section{Preliminaries for vanishing results}
\label{sec:vanishingPrelim}

Let $p \equiv 1 \mod{4}$, without restriction on the size of $p$. Again let $\bfr$ and $\bfs$ denote the sets of quadratics and nonquadratics (mod $p$) in $\{1,\ldots,\tfrac{p-1}{2}\}$, respectively, and again reserve $r$ and $s$ for elements of $\bfr$ and $\bfs$, respectively. We recall that $h$ and $k$ are always positive integers, that 
    \[
        H=hp \qquad\text{and}\qquad K=kp
    \]
always, and that
    \[
        \sum_{\mu \md {k}} \quad\text{and}\quad \psum_{\mu \md{k}} 
        \qquad\text{indicate}\qquad 
        \sum_{\mu=0}^{k-1} \quad\text{and}\quad
        \sum_{\substack{\mu = 0 \\ (\mu,k)=1}}^{k-1}, \quad\text{respectively}.
    \]

\noindent\textbf{Note:} In the remainder of this paper, let $[x]$ indicate the \emph{integer part} (or \emph{floor}) of $x$, unless indicated otherwise.

\subsection{The Dedekind sums \tops{$s(h,k)$}{s(h,k)} and \tops{$s_{\xc}(h,k)$}{sχ(h,k)}}
We now recall some basic properties of Dedekind sums---for further exposition see, e.g., \cites{Rademacher:TheoremsOnDedekindSums,Rademacher:DedekindSums}. The function $\Saw{x}$ is
    \[
        \Saw{x} = x - [x] - \tf{1}{2} + \tf{1}{2}\xd(x), \qquad\text{where}\qquad \xd(x) := \boldone_{\zz}(x),
    \]
and the \emph{Dedekind sum} $s(h,k)$ and its auxiliary sum $t(h,k)$ are defined via 
    \begin{align}
    \label{eq:s(h,k)}
        s(h,k) &= \sum_{\mu \md{k}} \BigSaw{\fr{h\mu}{k}}\BigSaw{\fr{\mu}{k}}\\
    \intertext{and}
    \label{eq:t(h,k)}
        t(h,k) &= \sum_{\mu \md{k}} \mu \Big[\fr{h\mu}{k}\Big].
    \end{align}
Expanding the quantities $\Saw{x}$ in \eqref{eq:s(h,k)}, it is easily verified that
    \(
    \label{eq:s(h,k)-t(h,k)Formula}
        s(h,k) = \tf{1}{6}h(k-1)(2-\tf{1}{k}) - \tfrac{1}{k}t(h,k) - \tf{1}{4}(k-1) + \tf{1}{4}(d-1),
    \)
where $d = (h,k)$. It is known \cite{Rademacher:TheoremsOnDedekindSums}*{Thm.~1} that 
    \(
    \label{eq:s-Scaling}
        s(qh,qk) = s(h,k) \qquad \text{for any integer $q > 0$},
    \)
and, using this and \eqref{eq:s(h,k)-t(h,k)Formula}, it is straightforward to verify that 
    \(
    \label{eq:t-Scaling}
        t(qh,qk) = q t(h,k) + \tf{1}{12}kq(q-1)\B\{4hk(q+1)-6h-3(k-1)+3(d-1)\B\},
    \)
where again $d = (h,k)$ and $q > 0$ is any integer. 

We now define and examine ``$\xc$-twisted'' versions of $s(h,k)$ and $t(h,k)$, noting that $\chi_{a}$ is again reserved for the Legendre symbol $\tjac{a}{p}$, and remarking that much of what follows applies to general even Dirichlet characters. 

\begin{definition}
    For $k \geq 1$ let $\xj := \fr{p}{(k,p)}$, so that $\xj k = \mathrm{lcm}(k,p)$. Then let
        \(
        \label{eq:s-chi}
            s_{\xc}(h,k) := \sum_{\mu \md{\xj k}} \xc_{\mu} \BigSaw{\fr{h\mu}{k}}\BigSaw{\fr{\mu}{\xj k}},
        \)
    and, in analogy with \eqref{eq:t(h,k)}, let
        \(
        \label{eq:tchi(h,k)}
            t_\xc(h,k) := \fr{1}{\xj} \sum_{\mu \md{\xj k}} \mu \xc_{\mu} \Big[\fr{h\mu}{k}\Big].
        \)
\end{definition}

\begin{remark}
In \cite{Berndt:OnEisensteinSeriesCharacters}*{p.\ 314 ff.}, Berndt defines
        \[
            S_{1}(h,k;\xy) = \sum_{\mu \mod{qk}} \xy(\mu) \Saw{h\mu/k}\Saw{\mu/qk}
        \]
    when $\xy$ is an even primitive character of modulus $q$. Berndt's $S_1$ and our $s_\xc$ clearly agree when $q=p$ and $(k,p)=1$, but ostensibly disagree when $p \mid k$; in fact, it is straightforward to check that they are equivalent in this latter case as well. Despite this equivalence, it is convenient to keep our definition \eqref{eq:s-chi} for $s_{\xc}$, and separately consider the cases where $p \mid k$ and where $p \nmid k$.
\end{remark}

Akin to the derivation of equation \eqref{eq:s(h,k)-t(h,k)Formula}, it is an easy computation to verify that
    \(
    \label{eq:sChi:tChiForm}
        s_{\xc}(h,k) = \fr{h}{k}B_{2}(\xc) - \fr{1}{k}t_{\xc}(h,k),
    \)
where we recall that
    \[
        B_{2}(\xc) = \fr{1}{p} \sum_{\mu \md{p}} \mu^{2}\xc_{\mu}.
    \]
Next, we show that $s_{\xc}$ and $t_{\xc}$ have ``scaling relations'' akin to \eqref{eq:s-Scaling} and \eqref{eq:t-Scaling}.

\begin{lemma}
    For any integer $q>0$, one has
        \begin{align}
        \label{eq:schi-Scaling}
            s_{\xc}(qh,qk) &= s_{\xc}(h,k), \\
        \label{eq:tchi-Scaling}
            t_{\xc}(qh,qk) &= q t_{\xc}(h,k). 
        \end{align}
\end{lemma}

\begin{proof}
    From \eqref{eq:sChi:tChiForm} we see that
        \[
            s_{\xc}(qh,qk) = \fr{h}{k}B_{2}(\xc) - \fr{1}{qk} t_{\xc}(qh,qk),
        \]
    so that \eqref{eq:schi-Scaling} follows at once from \eqref{eq:tchi-Scaling}. This latter relation follows from straightforward but tedious computations that split the sum over $\mu$ (modulo $qk$ or modulo $pqk$) into different arithmetic progressions, depending on $q$ and $k$ and their divisibility (or not) by $p$, so we omit the details.
\end{proof}

\subsection{The periodic Bernoulli numbers}
The following definitions are an abbreviated collection of definitions and basic results on the \emph{periodic Bernoulli numbers}\footnotemark~

\footnotetext{We use \cite{Berndt:ReciprocityDedekind}*{p.~287 ff.} as a convenient reference, but, as stated there, the original results and their generalizations can be found in the same author's works \cites{Berndt:PeriodicAnalogues,Berndt:PeriodicAnaloguesCorrigendum}.}

\begin{definition}
    For $p\equiv 1\mod{4}$ let $B_0(\xc) = B_{1}(\xc) = 0$, and let 
        \[
            B_{2}(\xc) = \fr{1}{p}\sum_{\mu\md{p}} \mu^{2}\xc_{\mu}.
        \]
    In addition, for real $y$ let
        \(
        \label{eq:B1Chi(y)}
            B_{1,\xc}(y) := \tf{1}{2}\xd(y)\xc_{y} - \sum_{0 \leq n \leq y} \xc_{n} \qquad (\xd(y) = \boldone_{\zz}(y)).
        \)
\end{definition}

The function $B_{1,\xc}(y)$ is evidently $p$-periodic, is constant (and integral) on intervals $(n,n+1)$ for $n \in \zz$, and is half-integral when $y \in \zz$. Moreover, one has \cite{Berndt:ReciprocityDedekind} that
    \[
        B_{1,\xc}(y) = \sum_{\mu\md{p}} \xc_{\mu} \BigSaw{\fr{y+\mu}{p}}.
    \]
For nonintegral $y$, one has $B_{1,\xc}(-y)=-B_{1,\xc}(y)$. The following additional property of $B_{1,\xc}$ is useful in section \ref{sec:stilde}.

\begin{lemma}
    \label{lem:B1chi-Sum-0}
    Suppose that $(k,p)=1$. For real $y$, one has
        \(
        \label{eq:B1ChiSum-Cases}
            \sum_{\xl \md{p}} B_{1,\xc}(k\xl+y) = \begin{cases}
                0 & p \nmid k,\\
                pB_{1,\xc}(y) & p \mid k.
            \end{cases}
        \)
\end{lemma}

\begin{proof}
    Because $B_{1,\xc}(x)$ is $p$-periodic, equation \eqref{eq:B1ChiSum-Cases} trivially holds when $p \mid k$, so suppose that $p \nmid k$; for the same reason, no generality is lost in assuming that $y \in [0,p)$. From \eqref{eq:B1Chi(y)} we have
        \(
        \label{eq:B1ChiLemma-0}
            \sum_{\xl=0}^{p-1} B_{1,\xc}(k\xl+y) 
            = \tf{1}{2}\sum_{\xl=0}^{p-1} \xd(k\xl+y)\xc_{k\xl+y} 
            - \sum_{\xl=0}^{p-1}\sum_{\nu=0}^{k\xl+[y]} \xc_{\nu},
        \)
    where $[y]$ is the integer part of $y$. Because $(k,p)=1$, as $\xl$ runs over all residue classes modulo $p$, so too does $k\xl+[y]$, whereby
        \[
            \sum_{\xl=0}^{p-1}\sum_{\nu=0}^{k\xl+[y]} \xc_{\nu} 
            = \sum_{\xl=0}^{p-1}\sum_{\nu=0}^{\xl} \xc_{\nu}.
        \]
    As $\xc$ is even, one can easily see that this latter double sum is zero. The first sum on the right-hand side of \eqref{eq:B1ChiLemma-0} is also zero: indeed, if $y$ is nonintegral then all $\xd(k\xl+y) = 0$, and if $y$ is integral then $k\xl+y$ again runs over all residue classes modulo $p$ as $\xl$ does, yielding $\fr{1}{2}\sum_{\xl=0}^{p-1} \xc_{\xl} = 0$.
\end{proof}

The following congruences of $B_{2}(\xc)$ play important roles in our derivations.

\begin{lemma}
    \label{lem:B2chi-Congruences}
    Let $p \equiv 1 \mod{4}$ with $p \neq 5$. Then $B_2(\xc)$ is integral, and
        \(
        \label{eq:B2Chi(mod8)}
            B_{2}(\xc) \equiv \begin{cases}
                0 \mod{8}, & p\equiv 1 \mod{8},\\
                4 \mod{8}, & p\equiv 5 \mod{8}.
            \end{cases}
        \)
\end{lemma}

\begin{remark}
    In the case $p=5$, one has $B_{2}(\xc) = \tf{4}{5}$.
\end{remark}

\begin{proof}
    We first show that $B_{2}(\chi)$ is an integer by showing that $\sum_{\mu \md{p}}\mu^{2}\xc_{\mu} \equiv 0 \mod{p}$\footnote{%
    We thank Trevor Wooley for showing us an elementary proof of this congruence.}. 
    Let $g$ be a primitive root modulo $p$. Since $\xc_{\mu} \equiv \mu^{(p-1)/2} \mod{p}$, we have
        \[
            pB_{2}(\xc) = \sum_{\mu=1}^{p-1} \mu^2 \xc_{\mu} \equiv \sum_{\mu=1}^{p-1} \mu^{2}\mu^{(p-1)/2} 
            \equiv \sum_{\mu=1}^{p-1} g^{2\mu + \mu(p-1)/2} \mod{p},
        \]
    and, separating $\mu$ by parity, this is
        \[
            \equiv \sum_{\nu=1}^{(p-1)/2} g^{4\nu} + \sum_{\nu=0}^{(p-3)/2} g^{4\nu+2+(p-1)/2} \equiv (g^4 - g^2) \sum_{\nu=0}^{(p-3)/2} g^{4\nu} \mod{p}.
        \]
    Because $p > 5$, we have $g^{4} \not\equiv 1 \mod{p}$ and $g^{2} \not\equiv -1 \mod{p}$, whereby
        \[
            (g^4 - g^2)\sum_{\nu=0}^{(p-3)/2} g^{4\nu} \equiv g^{2}(g^{2}-1)\fr{(g^{4})^{(p-1)/2}-1}{g^{4}-1} \equiv g^{2}\cdot\fr{g^{2(p-1)}-1}{g^{2}+1} \mod{p}.
        \]
    As $g^{p-1} \equiv 1 \mod{p}$, the congruence $pB_{2}(\xc) \equiv 0 \mod{p}$ follows, so that $B_{2}(\xc)$ is indeed an integer.

    Turning now to the congruences modulo 8, since $\xc$ is even we have
        \[
            pB_{2}(\xc) = \sum_{\mu=1}^{p-1} \mu^{2}\xc_{\mu} = \sum_{\mu=1}^{(p-1)/2} (\mu^2 + (p-\mu)^2)\xc_{\mu} = - \sum_{\mu=1}^{(p-1)/2} 2\mu(p-\mu)\xc_{\mu},
        \]
    and, because $4 \mid 2\mu(p-\mu)$ necessarily, reducing modulo 8 we find that
        \[
            - \sum_{\mu=1}^{(p-1)/2} 2\mu(p-\mu)\xc_{\mu} \equiv \sum_{\mu=1}^{(p-1)/2} 2\mu(p-\mu) \equiv \tf{1}{6}p(p-1)(p+1) \mod{8}.
        \]
    Thus $B_{2}(\xc) \equiv \tf{1}{6}(p-1)(p+1) \mod{8}$, and \eqref{eq:B2Chi(mod8)} follows.
\end{proof}

\begin{remark}
    When $p \equiv 3 \mod{4}$, using similar arguments one may easily verify that $B_{2}(\xc)$ is again integral, and that $B_{2}(\xc) \equiv 1 \mod{2}$ in this case.
\end{remark}

\subsection{Some reciprocity laws and congruences}

We now recall some well-known reciprocity relations for $s(h,k)$ and $t(h,k)$. Namely, for $(h,k)=1$ one has
    \[
        s(h,k)+s(k,h) = \tf{1}{12}(\tf{h}{k}+\tf{k}{h}+\tf{1}{hk})-\tf{1}{4},
    \]
and, again for $(h,k)=1$, one has \cite{Rademacher:EineArithmetische}*{p.~221}
    \[
        ht(h,k)+kt(k,h) = \tf{1}{12}(h-1)(k-1)(8hk-h-k-1).
    \]
Unsurprisingly, the reciprocity formulae for $s_{\xc}(h,k)$ are more involved, and we require an additional definition.

\begin{definition}
    For $(a,b)=1$ with $b>1$, let
        \[
            \~{s}_{\xc}(a,b) := \sum_{\mu \md{bp}} \BigSaw{\fr{\mu}{bp}} B_{1,\xc}\Big(\fr{a\mu}{b}\Big) 
            = \sum_{\mu \md{bp}} \sum_{\nu \md{p}} \xc_{\nu} \BigSaw{\fr{\mu}{bp}} \BigSaw{\fr{a\mu+b\nu}{bp}}.
        \]
\end{definition}

Some properties of $\~{s}_{\xc}(a,b)$ are discussed in section \ref{sec:stilde}.

\begin{proposition}[Berndt \cite{Berndt:ReciprocityDedekind}]
    \label{prop:RecipLaws}
    For $(k,p)=1$ and $(h,k)=1$, one has
        \(
        \label{eq:Reciprocity-CoP}
            s_{\xc}(h,k) + \~{s}_\xc(k,h) = \fr{h}{2k} B_2(\xc).
        \)
    For $K = kp$ and $(h,K)=1$, one has
        \(
        \label{eq:Reciprocity-P-Orig}
            s_{\xc}(h,K) + \~{s}_{\xc}(K,h) = \fr{h^2 + \xc_{h}}{2hK}B_2(\xc).
        \)
\end{proposition}

\begin{corollary}
    \label{cor:RecipLaws(Alt)}
    Let $K=kp$ and $(h,K) = 1$. One has
        \[
            \~{s}_{\xc}(K,h) = \xc_{h}s_{\xc}(\^{K},h),
        \]
    where $K\^{K} \equiv 1 \mod{h}$, and so \eqref{eq:Reciprocity-P-Orig} states that
        \(
        \label{eq:Reciprocity-P}
            s_{\xc}(h,K) + \xc_{h} s_{\xc}(\^{K},h) = \fr{h^2 + \xc_{h}}{2hK}B_2(\xc).
        \)
\end{corollary}

\begin{proof}
    Because the summands in $\~{s}(K,h)$ depend only on $\mu$ modulo $ph$, we write
        \[
            \mu \equiv \xl h\-{h} + j p\^{p} \mod{ph}, \quad \text{where $h\-{h} \equiv 1 \mod{p}$ and $p\^{p}\equiv 1 \mod{h}$},
        \]
    and $0 \leq \xl < p$ and $0 \leq j < h$. With this, we have
        \begin{align*}
            \~{s}_{\xc}(K,h) &= 
                \sum_{j\md{h}} \sum_{\nu\md{p}} \sum_{\xl\md{p}} 
                \xc_{\nu} \BigSaw{\fr{\xl h\-{h} + j p\^{p}}{ph}} \BigSaw{\fr{K\xl h\-{h} + Kjp\^{p} + h\nu}{ph}} \\
            &= \sum_{j\md{h}}\sum_{\nu\md{p}} 
                \xc_{\nu}\BigSaw{\fr{Kjp\^{p} + h\nu}{ph}} 
                \sum_{\xl\md{p}} \BigSaw{\fr{\xl\-{h}}{p} + \fr{j\^{p}}{h}}.
        \end{align*}
    Then, because $(\bar{h},p)=1$ and $\sum_{\xl\md{p}} \Saw{\xl/p + x} = \Saw{px}$ for all real $x$, the above is
        \begin{align*} 
            = \sum_{j\md{h}} \sum_{\nu\md{p}} 
                \xc_{\nu}\BigSaw{\fr{Kjp\^{p}+h\nu}{ph}}\BigSaw{\fr{j}{h}}
            &= \sum_{j\md{h}} \sum_{\nu\md{p}} 
                \xc_{\nu} \BigSaw{\fr{jp\^{p}+h\nu}{ph}} \BigSaw{\fr{\^{K}j}{h}} \\
            &= \sum_{j\md{h}}\sum_{\nu\md{p}} 
                \xc_{\-{h}\nu}\BigSaw{\fr{\^{K}j}{h}}\BigSaw{\fr{jp\^{p}+h(\-{h}\nu)}{ph}},
        \end{align*}
    and this latter quantity is equal to $\xc_{h}s_{\xc}(\^{K},h)$ since $(h,p)=1$.
\end{proof}

\section{A lemma on \tops{$\~{s}_{\xc}$}{~sχ}}
\label{sec:stilde}

In light of Proposition \ref{prop:RecipLaws}, we see that congruences of $B_{2}(\xc)$ and $\~{s}_{\xc}$ can be leveraged into congruences of $s_{\xc}$. Because Corollary \ref{cor:RecipLaws(Alt)} lets us ``avoid'' $\~{s}_{\xc}$ when $p \mid K$, we require a lemma on $\~{s}_{\xc}(k,h)$ when $k$ is coprime to $p$.  

\begin{lemma}
    \label{lem:sTilde(mod2)}
    Let $p \equiv 1 \mod{4}$, let $b > 1$, and let $(a,b)=(a,p)=1$. Then $\~{s}_{\xc}(a,b)$ is integral, and moreover
        \[
            \~{s}_{\xc}(a,b) \equiv \tf{1}{2}(\xc_{a}-1) \mod{2}.
        \] 
\end{lemma}

The proof of Lemma \ref{lem:sTilde(mod2)} is somewhat technical, so we begin by first expanding and reducing the formula for $\~{s}_{\xc}(a,b)$. Writing $\mu = \xl b+j$ with $0 \leq j < b$ and $0 \leq \xl < p$, and expanding $\Saw{\mu/bp}$, we use Lemma \ref{lem:B1chi-Sum-0} to deduce that
    \[
    \~{s}_{\xc}(a,b) 
        = \sum_{j \md{b}} \sum_{\xl \md{p}} \lf(\fr{\xl b+j}{bp}-\fr{1}{2}\rh)B_{1,\xc}\lf(a\xl + \fr{aj}{b}\rh) 
        = \sum_{j \md{b}} \sum_{\xl \md{p}} \fr{\xl}{p} B_{1,\xc}\lf(a\xl + \fr{aj}{b}\rh).
    \]
Expanding $B_{1,\xc}(a\xl + \tf{aj}{b})$ and the $\Saw{x}$ terms therein, and again reducing (using the fact that $\sum_{\nu \md{p}} \xc_{\nu} = 0$), we are led to the formula
    \[
        \~{s}_{\xc}(a,b) = \fr{1}{p} \sum_{j \md{b}} \sum_{\xl \md{p}} \sum_{\nu \md{p}} \xl \xc_{\nu} \lf[\fr{a\xl+a(j/b)+\nu}{p}\rh],
    \]
which motivates the following definition.

\begin{definition}
    For fixed $a$ and $p$, with $p \equiv 1 \mod{4}$ and $(a,p)=1$, for real $y$ let
        \[
            S(y) := S(a,p;y) := \sum_{\mu \md{p}} \sum_{\nu \md{p}} \mu \xc_{\nu} \lf[\fr{a\mu+ay+\nu}{p}\rh].
        \]
\end{definition}

Now
    \(
    \label{eq:sTilde-SumS(j/b)Form}
        \~{s}_{\xc}(a,b) = \fr{1}{p} \sum_{j=0}^{b-1} S\B(\fr{j}{b}\B), 
    \)
and we establish Lemma \ref{lem:sTilde(mod2)} by way of three ``sublemmata'' concerning $S(y)$. 

\begin{lemma}
    \label{lem:S(y)Properties}
    Fixing $a$ and $p$ as above, let $S(y) = S(a,p;y)$.
    \begin{enumerate}[label={\normalfont(\arabic*)}]
    \item \label{it:S(y)Integral} One has $S(y) \equiv 0 \mod{p}$ for all $y$, so that $\tf{1}{p}S(y)$ is always integral.
    \item \label{it:S(y)=S(1-y)} If $y$ is not an element of $\fr{1}{a}\zz$, then
        \[
            S(1-y)=S(y).
        \]
    \end{enumerate}
\end{lemma}

Before proving Lemma \ref{lem:S(y)Properties} we provide some context: For the moment, let $(a,p)=1$ and $(a,b)=1$, and assume that Lemma \ref{lem:S(y)Properties} holds. Then none of $\fr{1}{b}, \ldots, \fr{b-1}{b}$ are in $\tf{1}{a}\zz$, whereby $S(\tf{j}{b}) = S(\tf{b-j}{b})$ for $0 < j < b$, and most of the summands $S(\fr{j}{b})$ in \eqref{eq:sTilde-SumS(j/b)Form} cancel each other modulo 2. Thus, in this case we deduce that
    \(
    \label{eq:sTilde(mod2)(Temp1)}
        \~{s}_{\xc}(a,b) \equiv \begin{cases}
            S(0) \hphantom{+S(\fr{1}{2})\,\,\,} \mod{2} & \text{if $b$ is odd}, \\
            S(0) + S(\fr{1}{2}) \mod{2} & \text{if $b$ is even}.
        \end{cases}
    \)

\begin{proof}[Proof of Lemma \ref{lem:S(y)Properties}] 
For part \ref{it:S(y)Integral} let $c := [ay]$, so that $[\fr{a\mu+ay+\nu}{p}]=[\fr{a\mu+c+\nu}{p}]$, and for $0 \leq \mu < p$ let 
    \[
        a\mu+c = q_{\mu}p-r_{\mu} \qquad \text{with $1 \leq r_{\mu} \leq p$}.
    \]
Because $\sum_{\nu\md{p}} \xc_{\nu} = 0$, the sum $S(y)$ is 
    \[
        = \sum_{\mu,\nu \md{p}} \mu \xc_{\nu} \lf[\fr{q_{\mu}p-r_{\mu}+\nu}{p}\rh] 
        = \sum_{\mu,\nu \md{p}} \mu \xc_{\nu} \lf(q_{\mu} + \lf[\fr{\nu-r_{\mu}}{p}\rh]\rh) 
        = (-1) \sum_{\mu=0}^{p-1} \sum_{\nu=0}^{r_{\mu}-1} \mu \xc_{\nu}.
    \]
Since $r_{\mu} \equiv -a\mu-c \mod{p}$, the quantity $r_{\mu}-1$ runs over $\{0,1,\ldots,p\}$ as $\mu$ runs over that same range, so we deduce that
    \[
        S(y) \equiv \-{a} \sum_{\mu=0}^{p-1} \big((r_{\mu}-1)+c+1\big) \sum_{\nu=0}^{r_{\mu}-1} \xc_{\nu} 
        \equiv \-{a} \sum_{\mu=0}^{p-1} (\mu+c+1) \sum_{\nu=0}^{\mu} \xc_{\nu} \mod{p},
    \]
where $a\bar{a} \equiv 1 \mod{p}$.
Discarding the $\mu=0$ and $\nu=0$ summands since they are clearly 0, and then swapping our order of summation, the above is
    \[
        \equiv \-{a} \sum_{\nu=1}^{p-1} \xc_{\nu} \sum_{\mu=\nu}^{p-1} (\mu+c+1) 
        \equiv -\fr{\-{a}}{2} \sum_{\nu=1}^{p-1} \lf(\nu^{2}+(1+2c)\nu\rh)\xc_{\nu} 
        \equiv -\fr{\-{a}}{2} \sum_{\nu=1}^{p-1} \nu^{2}\xc_{\nu} \mod{p}.
    \]
This latter quantity is equal to $-\fr{1}{2}\-{a}p B_{2}(\xc)$, and, since $B_{2}(\xc)$ is even and integral (as shown in Lemma \ref{lem:B2chi-Congruences}), it follows that $S(y) \equiv 0 \mod{p}$, proving part \ref{it:S(y)Integral} of the lemma.

For part \ref{it:S(y)=S(1-y)}, now suppose that $y \not\in \fr{1}{a}\zz$, and observe that
    \[
        S(1-y) = \sum_{\mu=0}^{p-1}\sum_{\nu=0}^{p-1} \mu \xc_{\nu} \lf[\fr{a(\mu+1)-ay+\nu}{p}\rh] = \sum_{\mu=1}^{p}\sum_{\nu=1}^{p-1} (\mu-1)\xc_{\nu}\lf[\fr{a\mu + [-ay] + \nu}{p}\rh].
    \]
Again letting $c := [ay]$ and using the fact that 
    \[
        [-ay] = -1 - [ay] + \xd(ay) = -1-c,
    \] 
we replace $\mu$ and $\nu$ above by $p-\mu$ and $p-\nu$, respectively, and simplify to deduce that 
    \[
        S(1-y) = \sum_{\mu,\nu \md{p}}
            (p-\mu-1) \xc_{\nu} \lf[\fr{-a\mu-c-\nu-1}{p}\rh].
    \]
Again using the identity $[-x] = -1 - [x] + \xd(x)$, we have
    \[
        S(1-y) = \sum_{\mu,\nu \md{p}}
        (\mu+1-p) \xc_{\nu} \lf({\lf[\fr{a\mu+c+\nu+1}{p}\rh] - \xd\lf(\fr{a\mu+c+\nu+1}{p}\rh)}\!\rh),
    \]
and, using the observation that
    \[
        \B[\fr{n+1}{p}\B] = \B[\fr{n}{p}\B] + \xd\B(\fr{n+1}{p}\B) \qquad \text{for integer $n$},
    \]
it follows that
    \[
        S(1-y) = \sum_{\mu,\nu \md{p}} (\mu+1-p)\xc_{\nu}\lf[\fr{a\mu+c+\nu}{p}\rh] = S(y) + (1-p) \sum_{\mu,\nu \md{p}} \xc_{\nu} \lf[\fr{a\mu+c+\nu}{p}\rh].
    \]
From this point, using arguments similar to those for item \ref{it:S(y)Integral}, one can easily show that the latter sum here is 0, and the assertion that $S(1-y) = S(y)$ follows.
\end{proof} 

\begin{remark}
    If $y \in \tf{1}{a}\zz $, say $y=\fr{c}{a}$, it is similarly easy to derive that
        \[
            S(1-y) = S(y) - \sum_{\mu \md{p}} \mu \xc_{a\mu+c}.
        \]
\end{remark}

As discussed leading up to \eqref{eq:sTilde(mod2)(Temp1)}, if $(a,b)=1$ and $(a,p)=1$, then $\~{s}_{\xc}(a,b)$ is equivalent to either $S(0)$ or $S(0)+S(\fr{1}{2})$ modulo $2$. The next lemma shows that in fact \mbox{$\~{s}_{\xc}(a,b) \equiv S(0) \mod{2}$}, regardless of the parity of $b$.

\begin{lemma}
    Fix $p \equiv 1 \mod{4}$, and let $a$ be odd with $(a,p)=1$. Then
        \[
            S(\tfr{1}{2}) \equiv 0 \mod{2}.
        \]
\end{lemma}

\begin{proof} 
Let $a = 2c+1$; since $S(y)$ depends only on $a \mod{p}$, we may assume that $0 \leq c \leq \fr{p-3}{2}$. Reducing modulo 2, we have
    \begin{align}
        & S\B(\fr{1}{2}\B)
        = \sum_{\mu,\nu=0}^{p-1} \mu \xc_{\nu} \left[\fr{a\mu+c+\nu}{p}\right] 
        \equiv \sum_{j=1}^{(p-1)/2} \sum_{\nu=1}^{p-1} \lf[\fr{a(2j-1)+c+\nu}{p}\rh] \notag \\
    \label{eq:S(1/2)(mod2):1}
        & \qquad \qquad \equiv \sum_{\mu,\nu=1}^{p-1} \left[\fr{a\mu+c+\nu}{p}\right] - \sum_{j=1}^{(p-1)/2}\sum_{\nu=1}^{p-1} \left[\fr{2aj+c+\nu}{p}\right] \mod{2}.
    \end{align}
For the first sum in \eqref{eq:S(1/2)(mod2):1}, a straightforward computation (using arguments similar to those of the previous lemma) yields that
    \[
        \sum_{\mu,\nu=1}^{p-1} \lf[\fr{a\mu+c+\nu}{p}\rh] \equiv c \mod{2}.
    \]
For the second sum in \eqref{eq:S(1/2)(mod2):1}, we similarly derive that
    \[
        \sum_{j=1}^{(p-1)/2} \sum_{\nu=1}^{p-1} \left[\fr{2aj+c+\nu}{p}\right] 
        \equiv \sum_{j=1}^{(p-1)/2} \left[\fr{2aj+c}{p}\right] \mod{2}.
    \]
Thus, to conclude that $S(\tf{1}{2}) \equiv 0 \mod{2}$, it suffices to show that 
    \(
    \label{eq:S(1/2)(mod2):4}
        \sum_{j=1}^{(p-1)/2} \left[\fr{2aj+c}{p}\right] \equiv c \mod{2} \qquad (a=2c+1).
    \)

Let $A_j$ denote the $j$-th summand in \eqref{eq:S(1/2)(mod2):4}. 
First, if $a=1$ and $c=0$, then $A_{j} = [\fr{2j}{p}] = 0$ for $1 \leq j \leq \fr{p-1}{2}$, and \eqref{eq:S(1/2)(mod2):4} holds. 
Thus, suppose that $a > 1$, so that $c>0$. Writing
    \[
        \sum_{1\.\leq\.j\.\leq\.\fr{p-1}{2}} 
            \!\! A_{j} = A_{\fr{p-1}{4}} + A_{\fr{p-1}{2}} 
        + \sum_{1 \.\leq\. j \.<\. \fr{p-1}{4}} 
            \!\!\big(A_{j} + A_{\fr{p-1}{2}-j}\big),
    \]
the remainder of our proof amounts to showing that
    \[
        A_{\fr{p-1}{4}} \equiv c \mod{2}, \qquad\qquad 
        A_{\fr{p-1}{2}} \equiv 0 \mod{2},
    \]
and
    \[
        A_{j} \equiv A_{\fr{p-1}{2}-j} \mod{2} \qquad (1 \leq j < \tfr{p-1}{4}).
    \]
It is easy to verify the former two congruences, so we only demonstrate the latter. 

Fixing $1 \leq j < \fr{p-1}{4}$, we find that 
    \[
        A_{\fr{p-1}{2}-j}
        = \B[\fr{2a(\fr{p-1}{2}-j)+c}{p}\B] 
        = a + \B[\fr{-2aj-a+c}{p}\B] 
        \equiv 1 + \left[\fr{-2aj-c-1}{p}\right] \mod{2}.
    \]
Since $[-x]=-1-[x]+\xd(x)$, this implies that 
    \[
        A_{\fr{p-1}{2}-j} \equiv \lf[\fr{2aj+c+1}{p}\rh] + \xd\lf(\fr{2aj+c+1}{p}\rh) \mod{2}.
    \]
Using the fact that $[\fr{n+1}{p}] - \xd(\fr{n+1}{p}) = [\fr{n}{p}]$ with $n=2aj+c$, it follows that
    \[
        A_{j} + A_{\fr{p-1}{2}-j} \equiv \B[\fr{2aj+c}{p}\B] + \B[\fr{2aj+c}{p}\B] \equiv 0 \mod{2},
    \]
as claimed, and the assertion of the lemma follows.
\end{proof} 

Having now established that $\~{s}_{\xc}(a,b) \equiv S(0) \mod{2}$, regardless of the parity of $b$, it remains only to consider $S(0)$.

\begin{lemma}
\label{lem:S(0)Mod2} 
    Fix $p \equiv 1 \mod{4}$ and $(a,p)=1$, and let $S(y) = S(a,p;y)$. Then 
        \[
            S(0) \equiv \tf12(\xc_{a}-1) \mod{2}.
        \]
\end{lemma}

\begin{proof}
Since $\xc_{\nu} \equiv 1 \mod{2}$ for $0 < \nu < p$, we see right away that
    \(
    \label{eq:S0-Odds}
        S(0) 
        \equiv \sum_{j=1}^{(p-1)/2} \sum_{\nu=1}^{p-1} \lf[\fr{a(2j-1)+\nu}{p}\rh] 
        \equiv \sum_{\mu,\nu=1}^{p-1} \lf[\fr{a\mu+\nu}{p}\rh] 
        - \sum_{j=1}^{(p-1)/2}\sum_{\nu=1}^{p-1} \lf[\fr{2aj+\nu}{p}\rh] \mod{2}.
    \)
Letting $\{x\} = x - [x]$ denote the \emph{fractional part} of $x$ for the moment, it follows easily from the coprimality of $a$ and $p$ that
    \begin{align*}
        \sum_{\mu,\nu=1}^{p-1} \lf[\fr{a\mu+\nu}{p}\rh] 
        &= \tf{1}{2}(a+1)(p-1)^2 - \sum_{\mu,\nu=1}^{p-1} \lf\{\fr{a\mu+\nu}{p}\rh\} \\
        &\qquad = \tf{1}{2}(a+1)(p-1)^2 - \tf{1}{2}(p-1)(p-2) \equiv 0 \mod{2},
    \end{align*}
so it remains to consider the latter sum in \eqref{eq:S0-Odds}. This time, for $1 \leq j \leq \fr{p-1}{2}$ we write 
    \[
        2aj = q_j p + r_j, \qquad \text{with $1 \leq r_j < p$},
    \]
noting that $r_j=p$ is not possible, whereby
    \[
        S(0) \equiv \sum_{j=1}^{(p-1)/2}\sum_{\nu=1}^{p-1} 
        \lf[\fr{2aj+\nu}{p}\rh] 
        \equiv \sum_{j=1}^{(p-1)/2}\sum_{\nu=1}^{p-1} \left[\fr{\nu+r_j}{p}\right] 
        \equiv \sum_{j=1}^{(p-1)/2}\sum_{\nu = p-r_j}^{p-1} 1 
        \equiv \sum_{j=1}^{(p-1)/2} r_j \mod{2}.
    \]
On the other hand we have $r_j = 2aj - p[\fr{2aj}{p}] \equiv [\fr{2aj}{p}] \mod{2}$, and we deduce that
    \[
        S(0) \equiv \sum_{j=1}^{(p-1)/2} \lf[\fr{2aj}{p}\rh] \mod{2}.
    \] 
The result then follows from the well-known (see, e.g., \cite{Rademacher:DedekindSums}*{p.~32 ff.}) congruence
    \[
       \sum_{j=1}^{(b-1)/2} \left[\fr{2aj}{b}\right] \equiv \fr{1}{2}\lf(\lf(\fr{a}{b}\rh)-1\rh) \mod{2} \qquad (\text{for $b$ odd and $(a,b)=1$}),
    \]
where in this equality $(\fr{a}{b})$ is the Jacobi symbol of elementary number theory.
\end{proof}

At last, with Lemmata \ref{lem:S(y)Properties}--\ref{lem:S(0)Mod2} established and formula \eqref{eq:sTilde-SumS(j/b)Form} in hand, the proof of Lemma \ref{lem:sTilde(mod2)} is complete.

\section{Vanishings of \tops{$\goL_{k}+\goL_{2k}$}{Lk+L2k} and of \tops{$\goL_{4k}$}{L4k}}
\label{sec:LkOddsVanishing}

We now return to our series formula for $\fp(n,\tjac{\cdot}{17})$, namely
    \(
    \label{eq:p(n,17)Series(Recall)}
    \begin{aligned}
    \fp(n,\tjac{\cdot}{17}) 
        &   = \rt{\tf{2}{51}\~{n}^{-1}} 
        \!\sum_{\substack{k=1 \\[0.1em] (k,2p)=1}}^{\infty} 
        \!\!\fr{\xl_{k}}{2k} \Big\{{\fL_{k}(n) + \fL_{2k}(n)}\Big\} I_{1}\mspace{-2mu} \Big({\tfr{2\pi}{k}\rt{\tfr{2}{51}\~{n}}}\,\Big) \\
        &   \quad + \rt{\tfr{2}{51}\~{n}^{-1}} 
        \!\sum_{\substack{k=1 \\[0.1em] 4 \.|\. k,\, p \ssnmid k}}^{\infty} 
        \!\fr{\xl_{k}}{k} \fL_{k}(n) I_{1}\mspace{-2mu} \Big(\tfr{4\pi}{k}\rt{\tfr{2}{51}\~{n}}\,\Big) \\
        &  \quad + \rt{\tf{1}{3}\~{n}^{-1}} 
        \!\sum_{\substack{k=1 \\[0.1em] (k,2p)=p}}^{\infty} 
        \!\!\fr{1}{2k} \B\{ {
            4\fL_{k}^{+}(n,0) I_{1}\?\big(\tfr{8\pi}{3k}\rt{3\~{n}}\.\big) 
            + 2\fL_{k}^{+}(n,2) I_{1}\?\big(\tfr{4\pi}{3k}\rt{3\~{n}}\.\big) 
            }\B\}
    \end{aligned}
    \)
where $\~{n} = n + \tf{2}{3}$,
    \[
        \fL_{k}(n) = \psum_{h \md{k}} \xf_{h,k} \exp(-2\pi inh/k) \qquad \text{when $p \nmid k$},
    \]
and
    \[
        \fL_{k}^{+}(n,m) = \psum_{\substack{h \md{k} \\ \xc_{h}=1}} \xf_{h,k} \exp\!\B\{{-2\pi i(nh+m\ol{2h})/k}\B\} \qquad \text{when $(k,2p)=p$},
    \]
with $2h\ol{2h} \equiv 1 \mod{k}$. To treat $\fL_{k}$ and $\fL_{k}^{+}$, we recall formula \eqref{eq:phi(h,k)SawSum} for $\xf_{h,k}$ and introduce notation for the exponent there.

\begin{definition}
    Let $(h,k)=1$ and let $\xj := \fr{p}{(k,p)}$, so that $\xj k = \mathrm{lcm}(p,k)$. We define
        \(
        \label{eq:Lambda(h,k)}
            \xL(h,k) = \sum_{\substack{\mu \md{\xj k} \\ \xc_{\mu}=+1}} \BigSaw{\fr{h\mu}{k}}\BigSaw{\fr{\mu}{\xj k}} 
            + \sum_{\substack{\mu \md{\xj k} \\ \xc_{\mu}=-1}} \bigg\{ \BigSaw{\fr{2h\mu}{k}} - \BigSaw{\fr{h\mu}{k}}\bigg\} \BigSaw{\fr{\mu}{\xj k}}.
        \)
\end{definition}

Per the above definition one has $\xf_{h,k} = \exp \pi i \xL(h,k)$, so that
    \(
    \label{eq:Lk(n)(Lambda)}
        \fL_{k}(n) = \psum_{h \md{k}} \exp\!\B\{ \pi i\xL(h,k)-2\pi inh/k \B\} \qquad \text{when $p \nmid k$},
    \)
and
    \[
        \fL_{k}^{+}(n,m) = \psum_{\substack{h \md{k} \\ 
        \xc_{h}=1}} \exp\!\B\{{\pi i\xL(h,k)-2\pi i(nh+m\ol{2h})/k}\B\} \qquad \text{when $(k,2p)=p$,}
    \]
where again $2h\ol{2h} \equiv 1 \mod{k}$. The following formula relating $\xL(h,k)$ to the Dedekind sums $s$ and $s_{\chi}$ follows easily from \eqref{eq:Lambda(h,k)} by using $\tf{1}{2}(1 \pm \xc_{\mu})$ as the indicator functions for the sets $\{\mu : \chi_{\mu} = \pm 1\}$; for emphasis we enunciate it as a lemma.

\begin{lemma}
    For $k \geq 1$ and $(h,k)=1$, one has
        \(
        \label{eq:Lambda(h,k)-sVersion}
            \xL(h,k) = \tf{1}{2}\B\{ 2s_{\xc}(h,k) - s_{\xc}(2h,k) \B\} + \tf{1}{2}\B\{ s(2h,k)-s(2hp,k) \B\}.
        \)
\end{lemma}

We now show that the first sum in \eqref{eq:p(n,17)Series(Recall)} vanishes when $n$ is odd.

\begin{lemma}
    \label{lem:L2kLk}
    Fix $p \equiv 1 \mod{8}$. For odd $k$ with $(k,p)=1$, one has
        \[
            \fL_{2k}(n) = (-1)^{n} \fL_{k}(n).
        \]
\end{lemma}

\begin{proof}
    Let $\varphi$ denote Euler's totient function. Given $(k,2p)=1$, we observe that if $h$ runs over the $\varphi(k)$ invertible residue classes modulo $k$, then $2h+k$ runs over the $\varphi(2k)=\varphi(k)$ invertible residue classes modulo $2k$. From the definition of $\fL_{2k}(n)$ then, we have
        \begin{align*}
            \fL_{2k}(n) &= \psum_{h \md{k}} \exp\!\B\{{\pi i \xL(2h+k,2k) - 2\pi in(2h+k)/2k }\B\} \notag \\
            &= (-1)^n \psum_{h \md{k}} \exp\!\B\{{\pi i \xL(2h+k,2k) - 2\pi inh/k }\B\}.
        \end{align*}
    Comparing this with \eqref{eq:Lk(n)(Lambda)} then, it suffices to show that
        \[
            \xL(2h+k,2k) - \xL(h,k) \equiv 0 \mod{2}
        \]
    for $(h,k)=1$. By formula \eqref{eq:Lambda(h,k)-sVersion}, we have
        \[
            \xL(2h+k,2k) = \fr12\B\{ 2s_{\xc}(2h+k,2k)-s_{\xc}(4h+2k,2k) \B\} + \fr{1}{2}\B\{ s(4h+2k,2k)-s(4hp+2kp,2k) \B\}.
        \]
    Using the $k$-periodicities (in $h$) of $s$ and $s_{\xc}$, and the ``scaling properties'' \eqref{eq:s-Scaling} and \eqref{eq:schi-Scaling}, this reduces to
        \[
            \xL(2h+k,2k) = \tf{1}{2}\B\{ 2s_{\xc}(2h+k,2k)-s_{\xc}(2h,k) \B\} + \tf{1}{2}\B\{ s(2h,k)-s(2hp,k) \B\},
        \]
    and it follows that
        \(
        \label{eq:Lambda-k-2k-Difference}
            \xL(2h+k,2k)-\xL(h,k) = s_{\xc}(2h+k,2k)-s_{\xc}(h,k).
        \)

    Using the reciprocity formula \eqref{eq:Reciprocity-CoP}, the right-hand side of \eqref{eq:Lambda-k-2k-Difference} is
        \begin{align*}
        s_{\xc}(2h+k,2k)-s_{\xc}(h,k) 
            &= \B\{{\fr{2h+k}{4k}B_{2}(\xc) - \~{s}_{\xc}(2k,2h+k)}\B\} 
            - \B\{{\fr{h}{2k}B_{2}(\xc) - \~{s}_{\xc}(k,h)} \B\} \\[0.4em]
            &= \tfr{1}{4}B_{2}(\xc) + \~{s}_{\xc}(k,h) - \~{s}_{\xc}(2k,2h+k),\notag
        \end{align*}
    and so
        \[
            \xL(2h+k,2k)-\xL(h,k) = \tf{1}{4}B_{2}(\xc) + \~{s}_{\xc}(k,h) - \~{s}_{\xc}(2k,2h+k).
        \]
    Thus, we to consider this latter quantity modulo 2. As $(2k,2h+k)=1$ by our assumptions on $h$ and $k$, we apply Lemmata \ref{lem:B2chi-Congruences} and \ref{lem:sTilde(mod2)} to deduce that
        \[
            s_{\xc}(2h+k,2k)-s_{\xc}(h,k) \equiv  \tf{1}{2}\lf(\xc_{k}-1 \rh) - \tf{1}{2}\lf(\xc_{2k}-1 \rh) \equiv \tf{1}{2} \lf(\xc_{k}-\xc_{2}\xc_{k}\rh) \mod{2},     
        \]
    and the result follows because $\xc_{2}=1$ for $p \equiv 1 \mod{8}$.
\end{proof}

\begin{lemma}
\label{lem:L4k}
    Fix $p \equiv 1 \mod{8}$ and let $(k,p)=1$. One has
        \[
            \fL_{4k}(n) = 0 \qquad \text{for odd $n \geq 1$.}
        \]
\end{lemma}

\begin{proof}
    From the observation that
        \[
        \begin{aligned}
            \fL_{4k}(n) &= \psum_{h \md{2k}} 
                \lf( 
                \exp\!\B\{ \pi i\xL(h,4k)-2\pi i\fr{nh}{4k} \B\} 
                + \exp\!\B\{ \pi i\xL(h+2k,4k)-2\pi i\fr{(h+2k)n}{4k} \B\} 
                \rh) \\
            &= \psum_{h \md{2k}} \Big(\exp\!\big\{\pi i\xL(h,4k)\big\} + (-1)^{n}\exp\!\big\{\pi i\xL(h+2k,4k)\big\}\Big)\exp\!\B\{{-\pi i\fr{nh}{2k}}\B\},
        \end{aligned}
        \]
    we see that it suffices to show (for the relevant $h$) that
        \[
            \xL(h+2k,4k) - \xL(h,4k) \equiv 0 \mod{2}.
        \]
    In a manner similar to \eqref{eq:Lambda-k-2k-Difference}, we quickly find that
        \[
            \xL(h+2k,4k) - \xL(h,4k) = s_{\xc}(h+2k,4k) - s_{\xc}(h,4k).
        \]
    Because $(h,2k)=1$ necessarily, clearly $(h+2k,4k) = (h,4k)=1$, and it follows (as in the proof of Lemma \ref{lem:L2kLk}) that $s_{\xc}(h+2k,4k) - s_{\xc}(h,4k)$ equal to
        \[
            \tf{1}{4} B_{2}(\xc) + \~{s}_{\xc}(4k,h) - \~{s}_{\xc}(4k,h+2k) \equiv \tf{1}{2}\big( \xc_{4k} - \xc_{4k} \big) \equiv 0 \mod{2},
        \]
    and the result follows.
\end{proof}

The results of this section show that: for odd $n$, the series \eqref{eq:p(n,17)Series(Recall)} reduces to
    \(
    \label{eq:p(n,17)-Odds}
        \fp(n,\tjac{\cdot}{17}) = \rt{\tf{1}{3}\~{n}^{-1}} \sum_{\substack{k = 1 \\ (k,2p)=p}}^{\infty} \fr{1}{2k} \B\{ 
            4\fL_{k}^{+}(n,0) I_{1}\?\B(\tfr{8\pi}{3k}\rt{3\~{n}}\B) 
            + 2\fL_{k}^{+}(n,2) I_{1}\?\B(\tfr{4\pi}{3k}\rt{3\~{n}}\B) \B\}.
    \)
Thus, to complete the proof of Theorem \ref{thm:MainVanishing}, it remains to examine the sums $\fL_{k}^{+}(n,m)$ for $(k,2p)=1$.

\section{The sums \tops{$\fL_{K}^{+}(n,m)$}{\fL K+(n,m)}}
\label{sec:OddsSetup}

Going forward, to emphasize divisibility by $p$ we use $K=kp$ instead of $k$. Turning to $\fL_{K}^{+}(n,m)$, we recall that
    \(
    \label{eq:LK(Recall)}
        \fL^{+}_{K}(n,m) = \psum_{\substack{h \md{K} \\ \xc_{h}=1}} \exp\!\B\{{\pi i\xL(h,K) - 2\pi i(nh+m\ol{2h})/K}\B\},
    \)
with $x\bar{x} \equiv 1 \mod{K}$. We now work to show that 
    \(
    \label{eq:LKVanishingClaim}
        \fL_{K}^{+}(n,0) = \fL_{K}^{+}(n,2) = 0 \qquad \text{for $n \equiv 0,2,8,10 \mod{17}$},
    \)
from which, in tandem with \eqref{eq:p(n,17)-Odds}, Theorem \ref{thm:MainVanishing} follows. To establish \eqref{eq:LKVanishingClaim}, we show (for use in \eqref{eq:LK(Recall)}) that: for suitable $m$ and $n$, for each invertible $h_{1} \mod{K}$ that is a \mbox{\emph{quartic}} (mod $p$), there is a unique invertible $h_{2} \mod{K}$ which is a quadratic-\emph{nonquartic} \mbox{(mod $p$)} and satisfies
    \[
        \exp\!\bigg\{{\pi i\xL(h_{1},K) - 2\pi i\fr{h_{1}n+\ol{2h}_{1}m}{K} }\bigg\} 
        = (-1)^n \exp\!\bigg\{{ \pi i\xL(h_{2},K) - 2\pi i\fr{h_{2}n+\ol{2h}_{2}m}{K} }\bigg\}.
    \]
This relation is equivalent to
    \[
        \bigg\{{ \xL(h_{1},K) - \fr{2(h_{1}n+\ol{2h}_{1}m)}{K} } \bigg\} - \bigg\{{\xL(h_{2},K) - \fr{2(h_{2}n+\ol{2h}_{2}m)}{K}}\bigg\} \equiv 1 \mod{2}, 
    \]
which is in turn equivalent to
    \(
    \label{eq:DK-PreDef}
        24K \B\{\xL(h_{1},K)-\xL(h_{2},K) \B\}
        - 48 \B\{ (h_{1}-h_{2})n+\bar{2}(\bar{h}_{1}-\bar{h}_{2})m \B\} \equiv 24K \mod{48K}.
    \)

To establish this congruence modulo $48K$, we separate said relation into congruences modulo $16$ and $3K$, and treat these in the next two sections.

\section{Congruences of \tops{$\xL(h,K)$}{𝛬(h,K)} modulo \tops{$3K$}{3K}}

Fixing odd $K=kp$, we first examine $24K\xL(h,K)$ modulo $K$, $3$, and $3K$. Throughout this section we assume that $(h,K)=1$, and we set
    \[
        \xvq := (3,K) = (3,k).
    \]
Recalling from equation \eqref{eq:Lambda(h,k)-sVersion} that
    \(
    \label{eq:Lambda(h,K)Recall-0}
        24K \xL(h,K) = 12K \B\{ 2s_{\xc}(h,K) - s_{\xc}(2h,K) \B\} + 12K \B\{ s(2h,K)-s(2hp,K) \B\},
    \)
we consider this expression in two parts, beginning with $12K\{2s_{\xc}(h,K)-s_{\xc}(2h,K)\}$. For this, we recall (from \eqref{eq:Reciprocity-CoP} and \eqref{eq:Reciprocity-P}) the reciprocity laws
    \begin{subequations}
    \begin{alignat}{2}
    \label{eq:Recip:Coprime(Recall)}
        2k s_{\xc}(h,k) &=  h B_2(\xc) - 2k \~{s}_{\xc}(k,h) \qquad &&(p \nmid k),\\
    \label{eq:Recip:K(Recall)}
        2hKs_{\xc}(h,K) &= (h^{2}+\xc_{h})B_{2}(\xc) - 2hK\xc_{h}s_{\xc}(\^{K},h) \qquad &&(p \mid K),
    \end{alignat}
where $K\^{K} \equiv 1 \mod{h}$. As $s_{\xc}(\^{K},h)$ depends only on $\^{K} \mod{h}$, we may assume that
    \[
        K\^{K} \equiv 1 \mod{2h} \qquad\text{and}\qquad 
        \^{K} > 0.
    \]
    \end{subequations}
As $(h,p)=1$, we use \eqref{eq:Recip:Coprime(Recall)} on $s_{\xc}(\^{K},h)$ in \eqref{eq:Recip:K(Recall)}, and then simplify to find that
    \(
    \label{eq:sChi(h,K)-sTildeFormula}
        2hKs_{\xc}(h,K) = \B( h^{2} + (1-K\^{K})\xc_{h} \B)B_{2}(\xc) + 2\xc_{h}hK\~{s}_{\xc}(h,\^{K}).
    \)

Applying \eqref{eq:sChi(h,K)-sTildeFormula} to $s_{\xc}(h,K)$ and $s_{\xc}(2h,K)$ then, we find that
    \(
    \label{eq:schi-Difference}
    \begin{aligned}
    12hK\B\{2s_{\xc}(h,K)- s_{\xc}(2h,K)\B\}
        & = 3\xc_{h} (4-\xc_{2})(1-\^{K}K)B_2(\xc) \\
            &\qquad + 12\xc_{h}h K  \B\{ 2\~{s}_{\xc}(h,\^{K}) - \xc_{2} \~{s}_{\xc}(2h,\^{K}) \B\}.
    \end{aligned}
    \)
In section \ref{sec:stilde}, we saw that $\~{s}_{\xc}(a,b)$ is integral if $(a,b)=1$ and $(a,p)=1$. Since $\xvq K \mid 3K$ and $h$ is invertible modulo $\xvq K$, it follows that
    \(
    \label{eq:X(modThK)}
        12K\B\{2s_{\xc}(h,K)-s_{\xc}(2h,K)\B\} \equiv 3\xc_{h}\bar{h}(4-\xc_{2})B_{2}(\xc) \mod{\xvq K},
    \)
where $h\bar{h} \equiv 1 \mod{\xvq K}$.

Returning now to \eqref{eq:Lambda(h,K)Recall-0} and examining $12K\{s(2h,K)-s(2hp,K)\}$, we multiply this by $2h$, and then use the scaling relation \eqref{eq:s-Scaling} to write 
    \(
    \label{eq:Lambda(modThK)-sDif2}
    \begin{aligned}
        24hK\B\{s(2h,K) - s(2hp,K)\B\} 
    &=  24hKs(2h,K)- 24hKs(2h,k) \\
    &=  12(2h)(K)s(2h,K) - p\mspace{-0.5mu}\cdot\mspace{-3mu}12(2h)(k)s(2h,k).
    \end{aligned}
    \)
From the theory of Dedekind sums we recall\footnote{See, e.g., \cite{Rademacher:TheoremsOnDedekindSums}*{p.~395 ff.}} that: For $(c,d)=1$, one has
    \[
        12dc \mspace{-0.5mu}\cdot\mspace{-3mu} s(d,c) \equiv d^{2} + 1 \,\,\big(\mathrm{mod}\,(3,c)c\big),
    \]
and this implies that
    \begin{subequations}
    \begin{align*}
        12(2h)(K)s(2h,K) &\equiv 4h^2+1 \mod{\xvq K}, \\
        12(2h)(k)s(2h,k) &\equiv 4h^2+1 \mod{\xvq k}.
    \end{align*}
    \end{subequations}
Applying these congruences in \eqref{eq:Lambda(modThK)-sDif2}, we deduce that
    \(
    \label{eq:Y(modThK)}
        12K\B\{s(2h,K)- s(2H,K)\B\} \equiv (1-p)(2h + \ol{2h}) \mod{\xvq K},
    \)
and, combining \eqref{eq:X(modThK)} and \eqref{eq:Y(modThK)}, we establish the first part of the following lemma.

\begin{lemma}
    \label{lem:Lambda(modThK)}
    Let $p\equiv 1\mod{4}$, let $K=kp$ be odd, and let $(h,K)=1$. Then
        \[
            24K \xL(h,K) \equiv 3\xc_{h}\-{h}(4-\xc_{2})B_{2}(\xc) - (p-1)(2h+\ol{2h}) \mod{\xvq K},
        \]
    where $\xvq=(3,K)$ and $x\-{x} \equiv 1 \mod{\xvq K}$. In addition, if $3 \nmid K$ then
        \(
        \label{eq:Lambda(mod3)}
            24K \xL(h,K) \equiv 0 \mod{3}.
        \)
\end{lemma}

\begin{proof}
    Only \eqref{eq:Lambda(mod3)} remains to be proved, so suppose that $3 \nmid K$ and recall that
        \[
            24K \xL(h,K) = 12K \B\{ 2s_{\xc}(h,K) - s_{\xc}(2h,K) \B\} + 12K \B\{ s(2h,K)-s(2hp,K) \B\}.
        \]
    First, dividing both sides of equation \eqref{eq:schi-Difference} by $h$ and noting that $h \mid (1-K\^{K})$, where $K\^{K} \equiv 1 \mod{2h}$, we immediately see that
        \[
            12K \B\{2s_{\xc}(h,K) - s_{\xc}(2h,K)\B\} \equiv 0 \mod{3}.
        \]
    Second, because $2(3,c)c\mspace{-1mu}\cdot\mspace{-2mu} s(d,c)$ is integral for all $c$ and $d$ (see, e.g., \cite{Rademacher:DedekindSums}*{p.~27}), trivially
        \[
            12K \B\{s(2h,K) - s(2hp,K)\B\} \equiv 0 \mod{3},
        \]
    and \eqref{eq:Lambda(mod3)} follows at once.
\end{proof}

\begin{corollary}
    \label{cor:Lambda17(modThK)}
    Under the conditions of Lemma \ref{lem:Lambda(modThK)}, if $p \equiv 1 \mod{8}$ and $h$ is a quadratic {\normalfont (mod $p$)}, then
        \[
            24K \xL(h,K) \equiv 9\bar{h}B_{2}(\xc) - (p-1)(2h+\ol{2h}) \mod{\xvq K}.
        \]
\end{corollary}

\section{Congruences of \tops{$\xL(h,K)$}{𝛬(h,K)} modulo 16}
\label{sec:Lambda(mod16)}

To determine $24K\xL(h,K)$ modulo 16, we begin with a formula relating this quantity to a ``combinatorial'' quantity $\xt_{\bfe\bfs}(h,K)$, as seen in the following lemma.

\begin{lemma}
    \label{lem:24KLambda(h,K)(mod16)}
    Fix $p\equiv 1\mod{4}$, let $K=kp$ be odd, and let $(h,K)=1$. One has
        \[
            24K\xL(h,K) \equiv 4(\xc_{h}-1) + 2(2h-1)(p-1) + 8\xt_{\bfe\bfs}(2h,K) \mod{16},
        \]
    where
        \[
        \begin{aligned}
            &   \xt_{\bfe\bfs}(h,K) := \\
            &\qquad \#\B\{ 0 < \mu < K : \text{$p \nmid \mu$, $2 \mid \mu$, $\mu$ is a \ul{non}quadratic {\normalfont(mod $p$)}, and $\{h\mu\}_{K}$ is odd} \B\},
        \end{aligned}
        \]
    with $\{h\mu\}_{K}$ defined as in \eqref{eq:frac(x)(K)}.
\end{lemma}

\begin{proof}
Making the assumptions stated, we first recall that
    \(
    \label{eq:Lambda(h,K)(Recall)}
        24K \xL(h,K) = 12K \B\{ 2s_{\xc}(h,K) - s_{\xc}(2h,K) \B\} + 12K \B\{ s(2h,K)-s(2H,K) \B\}.
    \)
Broadly speaking, we use known properties of the sums $s$, $t$, $s_{\xc}$, and $t_{\xc}$ to progressively ``shave down'' \eqref{eq:Lambda(h,K)(Recall)} modulo 16, beginning with the former ``braced'' quantity there.
Recalling from \eqref{eq:sChi:tChiForm} that 
    \(
    \label{eq:s(h,K)-tChiForm}
        Ks_{\xc}(h,K) = hB_{2}(\xc) - t_{\xc}(h,K), 
    \)
we have
    \[
        12K\B\{ 2s_{\xc}(h,K)-s_{\xc}(2h,K) \B\} = 12t_{\xc}(2h,K) - 24t_{\xc}(h,K),
    \]
and, using formula \eqref{eq:s(h,K)-tChiForm} in equation \eqref{eq:sChi(h,K)-sTildeFormula}, we may write
    \[
        24t_{\xc}(h,K) = 12\big(h-(1-K\^{K})/h\big)B_{2}(\xc) - 24K \~{s}_{\xc}(h,\^{K}),
    \]
where $K\hat{K} \equiv 1 \mod{h}$ and $\hat{K} > 0$. Then, recalling from Lemma \ref{lem:B2chi-Congruences} that $4 \mid B_{2}(\xc)$, and recalling the results on $\~{s}_{\xc}$ from Lemma \ref{lem:sTilde(mod2)}, it follows that
    \[
        24t_{\xc}(h,K) \equiv - 12K (\xc_{h}-1) \equiv 4(\xc_{h}-1) \mod{16},
    \]
and, subsequently, that
    \(
    \label{eq:Lambda(h,K)(mod16)(1)}
        12K\B\{2s_{\xc}(h,K)-s_{\xc}(2h,K)\B\} \equiv 12 t_{\xc}(2h,K) + 4(\xc_{h}-1) \mod{16}.  
    \)

Returning now to \eqref{eq:Lambda(h,K)(Recall)} and considering $24K\{ s(2h,K)-s(2H,K)\}$, we recall from equation \eqref{eq:s(h,k)-t(h,k)Formula} that 
    \[
        12K s(h,K) = (K-1)\Big( 2h(2K-1) - 3K \Big) - 12t(h,K),
    \]
and from this it follows that
    \begin{align}
    \label{eq:Lambda(h,K)(mod16)(2)}
        12K\B\{s(2h,K)-s(2H,K)\B\} 
        &= 2h(1-p)(K-1)(2K-1) + 12\B\{ t(2H,K)-t(2h,K) \B\} \notag\\
        &\equiv 12\B\{ t(2H,K)-t(2h,K) \B\} \mod{16}.
    \end{align}
Combining \eqref{eq:Lambda(h,K)(Recall)}, \eqref{eq:Lambda(h,K)(mod16)(1)}, and \eqref{eq:Lambda(h,K)(mod16)(2)} then, we deduce that
    \[
        24K\xL(h,K) \equiv 4(\xc_{h}-1) + 12\big\{{t_{\xc}(2h,K) + t(2H,K) - t(2h,K)}\big\} \mod{16}.
    \]
This provides our first ``reduction'' of $24K\xL(h,K)$ modulo 16. For convenience, let 
    \[
        T(h,K) := t_{\xc}(h,K) + t(H,K) - t(h,K),
    \]
so that the above relation states that
    \(
    \label{eq:Lambda(mod16)Simpler}
        24K\xL(h,K) \equiv 4(\xc_{h}-1) + 12 T(2h,K) \mod{16}.
    \)
To determine $24K \xL(h,K)$ modulo 16 via \eqref{eq:Lambda(mod16)Simpler} then, it suffices to determine a general formula for $T(h,K)$ modulo 4 that we can apply to $T(2h,K)$. 

We recall from \eqref{eq:t(h,k)} and \eqref{eq:tchi(h,k)} that
    \begin{align*}
        t(h,k) = \sum_{\mu \md{k}} \mu\B[\fr{h\mu}{k}\B] 
        \quad\text{and}\quad
        t_\xc(h,k) = \fr{1}{\xj} \sum_{\mu \md{\xj k}} \mu \xc_{\mu}\B[\fr{h\mu}{\xj k}\B],
    \end{align*}
respectively, where $\xj = \fr{p}{(k,p)}$, and we recall the well-known (see, e.g., \cite{MontgomeryNivenZuck:Introduction}*{p.~186}) identity 
    \[
        \sum_{\mu \md{k}} \B[\fr{h\mu}{k}\B] = \fr{(h-1)(k-1)}{2} \qquad \text{for $(h,k)=1$}.
    \]
Turning to $T(h,K)$, we use the scaling relation \eqref{eq:t-Scaling} to observe that
    \[
        t(H,K) = t(hp,kp) \equiv p \?\cdot\? t(h,k) \mod{4},
    \]
whereby
    \begin{align*}
    & t(H,K) - t(h,K) 
    = t(H,K) - \bigg({\sumsub{\mu\md{K}\\ p\ssnmid \mu} \mu \B[ \fr{h\mu}{K} \B] 
    +   \sumsub{\mu\md{K}\\ p\.|\.\mu} \mu \B[ \fr{h\mu}{K} \B] }\bigg)  \\
    & \qquad \equiv  p \?\cdot\? t(h,k) - \sumsub{\mu\md{K}\\ p\ssnmid \mu} \mu \B[ \fr{h\mu}{K} \B] - p \?\cdot\? t(h,k) \equiv   -\sumsub{\mu\md{K}\\ p\ssnmid \mu} \mu \B[\fr{h\mu}{K}\B] \mod{4}.
    \end{align*}
From the definition of $t_{\xc}(h,K)$ then, it follows that
    \(
    \label{eq:T(h,K)(mod4)Sum}
        T(h,K) = t_{\xc}(h,K) + t(H,K) - t(h,K) \equiv \sumsub{\mu\md{K}\\ p \ssnmid \mu} \mu(\xc_{\mu}-1) \B[\fr{h\mu}{K}\B] \mod{4}.
    \)
Now, if $p \nmid \mu$ then $\xc_{\mu}-1$ is even, and therefore \eqref{eq:T(h,K)(mod4)Sum} is equivalent to the congruence
    \[
        T(h,K) \equiv \sumsub{\mu\md{K}\\ (\mu,2p)=1} (\xc_{\mu}-1) \B[\fr{h\mu}{K}\B] \mod{4},         
    \]
which is itself clearly equivalent to the congruence
    \(
    \label{eq:24KLambda-Nres-3t-Sum}
        \tf{1}{2}T(h,K) \equiv \sumsub{\mu\md{K}\\ (\mu,2p)=1} \B(\fr{\xc_{\mu}-1}{2}\B)\B[\fr{h\mu}{K}\B] \mod{2}.
    \)

In \eqref{eq:24KLambda-Nres-3t-Sum} our sum runs over only \emph{odd} $\mu$, but for our derivations it is more convenient to have a sum over \emph{even} $\mu$.
Starting with
\[
        \sum_{\substack{\mu \md{K}\\[0.1em] (\mu,2p)=1}} 
            \!\B(\fr{\xc_{\mu}-1}{2}\B)\B[\fr{h\mu}{K}\B] 
        = \sum_{\substack{\mu \md{K}\\[0.1em] p \ssnmid \mu}} 
            \B(\fr{\xc_{\mu}-1}{2}\B)\B[\fr{h\mu}{K}\B] 
        - \sum_{\substack{\mu \md{K}\\[0.1em] (\mu,2p)=2}} 
            \!\B(\fr{\xc_{\mu}-1}{2}\B)\B[\fr{h\mu}{K}\B], 
    \]
we employ the facts\footnotemark~that
\footnotetext{These equations follow from easy, direct computations by writing $[x]=x-\{x\}$, using the coprimality of $h$ and $K$, and using the fact that  $\sum_{\mu \md{K}} [h\mu/K] = \fr{1}{2}(h-1)(K-1)$.}%
    \[
        \sum_{\substack{\mu \md{K}\\ p \ssnmid \mu}} \xc_{\mu} \B[\fr{h\mu}{K}\B] = 0 
    \qquad\text{and}\qquad
        \sum_{\substack{\mu \md{K}\\ p \ssnmid \mu}} \B[\fr{h\mu}{K}\B] = \fr{(h-1)(p-1)k}{2}
    \]
to deduce that
    \(
    \label{eq:OddToEven(mod2)}
        \sum_{\substack{\mu\md{K}\\ (\mu,2p)=1}} 
        \!\B(\fr{\xc_{\mu}-1}{2}\B)\B[\fr{h\mu}{K}\B] 
    \equiv   \fr{(h-1)(p-1)}{4} + \sum_{\substack{\mu \md{K}\\ (\mu,2p)=2}}
        \!\B(\fr{\xc_{\mu}-1}{2}\B)\B[\fr{h\mu}{K}\B] \mod{2}.
    \)
Now writing $[x]=x-\{x\}$ with $\{x\} \in [0,1)$, one has
    \[
        [h\mu/K] \equiv K[h\mu/K] \equiv h\mu + K\{h\mu/K\} \mod{2}.
    \]
In fact, since $K\{h\mu/K\}$ is equivalent to the integer $\{h\mu\}_{K}$ from \eqref{eq:frac(x)(K)}, we have
    \[
        [h\mu/K] \equiv h\mu + \{h\mu\}_{K} \mod{2},
    \]
and it follows at once from \eqref{eq:OddToEven(mod2)} that
    \(
    \label{eq:T(h,K)/2-EvenSum}
        \tf{1}{2}T(h,K) 
        \equiv \fr{(h-1)(p-1)}{4} + \sumsub{\mu \md{K}\\ (\mu,2p)=2} 
        \B(\fr{\xc_{\mu}-1}{2}\B) \{h\mu\}_{K} \mod{2}.
    \)
The latter sum in \eqref{eq:T(h,K)/2-EvenSum} is evidently congruent, modulo $2$, to
    \[
        \#\Big\{ 0 < \mu < K : \text{$p \nmid \mu$, $2 \mid \mu$, $\mu$ is a nonquadratic (mod $p$), and $\{h\mu\}_{K}$ is odd} \Big\},
    \]
which is $\xt_{\bfe\bfs}(h,K)$, and the assertion of the lemma follows upon replacing $h$ with $2h$ in \eqref{eq:T(h,K)/2-EvenSum} and using the result in \eqref{eq:Lambda(mod16)Simpler}.
\end{proof}

Because determining $\xt_{\bfe\bfs}(h,K)$ modulo 2 is somewhat technical and tedious, for now we state the necessary result, and give a proof in section \ref{sec:Transfer}.

\begin{lemma}
    \label{lem:tau(es)(mod2)}
    Fix $p \equiv 1 \mod{4}$, let $K = kp$ be odd, let $(h,K)=1$, and suppose that $h$ is a quadratic {\normalfont (mod $p$)}. 
    Then
        \[
            \xt_{\bfe\bfs}(h,K) \equiv \begin{cases}
                0 \mod{2} & \text{if $h$ is a quartic {\normalfont(mod $p$)}}, \\
                1 \mod{2} & \text{if $h$ is a quadratic-nonquartic {\normalfont(mod $p$)}}.
            \end{cases}
        \]
\end{lemma}

\begin{corollary}
    \label{cor:Lambda(mod16)(+)}
    Fix $p\equiv 1\mod{8}$, let $K=kp$ be odd, let $(h,K)=1$, and suppose that $h$ is a \underline{quadratic} {\normalfont(mod $p$)}. Then 
        \[
            24K \xL(h,K) \equiv \begin{cases}
                0 \mod{16} & \text{if $2h$ is a quartic {\normalfont(mod $p$)}}, \\
                8 \mod{16} & \text{if $2h$ is a quadratic-nonquartic {\normalfont (mod $p$)}}.
            \end{cases}
        \]
\end{corollary}

\begin{proof}
    This is immediate from Lemmata \ref{lem:24KLambda(h,K)(mod16)} and \ref{lem:tau(es)(mod2)}, since $2$ is a quadratic (mod $p$) if $p \equiv 1 \mod{8}$.
\end{proof}

\section{The vanishing of \tops{$\fL_{K}^{+}(n,m)$}{\fL K+(n,m)}; the proof of Theorem \ref{thm:MainVanishing}}
\label{sec:vanishingProof}

We are at last ready to prove Theorem \ref{thm:MainVanishing}, and for this we recall (from section \ref{sec:OddsSetup}) that it remains to show (for odd $K=kp$) that 
    \(
    \label{eq:LKVanishingClaim(Recall)}
        \fL_{K}^{+}(n,0) = \fL_{K}^{+}(n,2) = 0 \qquad \text{for $n \equiv 0,2,8,10 \mod{17}$}.
    \)
We further recall that we seek pairs $(h_{1},h_{2})$ of quartics and quadratic-nonquartics \mbox{(mod $p$)} for which
    \[
        24K \B\{{\xL(h_{1},K)-\xL(h_{2},K)}\B\} 
        - 48 \B\{{(h_{1}-h_{2})n+\bar{2}(\bar{h}_{1}-\bar{h}_{2})m}\B\} \equiv 24K \mod{48K},
    \]
where $x\bar{x} \equiv 1 \mod{K}$.
The following lemma shows how the residue classes of $1-24n$ and $1-24m$ modulo $17$ affect the existence of the desired pairs $(h_1,h_2)$.

\begin{lemma}
    \label{lem:Lambda(h,K)-exps(mod2)}
    Fix $p=17$, let $K=17k$ be odd, and let $n > 0$ and $m \geq 0$. Fixing $g=3$, so that $g$ is a primitive root {\normalfont (mod $17$)}, suppose that
        \(
        \label{eq:(1-24n)(1-24m)Congs}
        \begin{aligned}
            1-24n &\equiv g^{4\nu+\xa+1} &&\mspace{-15mu}\mod{p}, \\
            1-24m &\equiv \ol{24}(1-g^{4\mu+\xa}) &&\mspace{-15mu}\mod{p},
        \end{aligned} 
        \qquad \text{for some $0 \leq \xa,\mu,\nu \leq 3$.}
        \)
    Then for each invertible $h_{1} \mod{K}$ that is a quartic {\normalfont(mod $p$)}, there is a unique, invertible $h_2 \mod{K}$ which
    is a quadratic-nonquartic {\normalfont(mod $p$)}, and 
    satisfies the relation
        \(
        \label{eq:LambdaDiff(mod2)}
            24K \B\{\xL(h_{1},K)-\xL(h_{2},K) \B\} 
            - 48 \B\{ (h_{1}-h_{2})n+\bar{2}(\bar{h}_{1}-\bar{h}_{2})m \B\} \equiv 24K \mod{48K},
        \)
    where $\-{x}$ denotes an inverse modulo $(3,K)K$.
\end{lemma}

\begin{corollary}
    \label{cor:LK(n,m)Vanishing}
    Under the assumptions of Lemma \ref{lem:Lambda(h,K)-exps(mod2)}, one has $\fL_{K}^{+}(n,m)=0$.
\end{corollary}

\begin{remark}
    Although we fix $p=17$ and $g=3$ in the proof below, to illustrate the presence and roles of these quantities in our proof we continue to write $p$ and $g$.
\end{remark}

\begin{proof}[Proof of Lemma \ref{lem:Lambda(h,K)-exps(mod2)}]

    For convenience we first recall some lemmata:
    \begin{subequations}
    If $\xvq = (3,K)$, $(h,K)=1$, and $h$ is a quadratic (mod $p$), then 
        \(
        \label{eq:Lambda(modThK)(Rec)}
            24K \xL(h,K) \equiv 9\bar{h}B_{2}(\xc) - (p-1)(2h+\ol{2h}) \mod{\xvq K}.
        \)
    If $(h,K)=1$ and $3 \nmid K$, then 
        \(
        \label{eq:Lambda(mod3)(Rec)}
            24K \xL(h,K) \equiv 0 \mod{3}.
        \)
    Lastly, because $2$ is a quadratic-nonquartic (mod $17$), Corollary \ref{cor:Lambda(mod16)(+)} implies that: If $(h,K)=1$ and $h$ is a quadratic (mod $p$), then
        \(
        \label{eq:Lambda(mod16)(Rec)}
            24K \xL(h,K) \equiv \begin{cases}
                8 \mod{16} & \text{if $h$ is a quartic (mod $p$)}, \\
                0 \mod{16} & \text{if $h$ is a quadratic-nonquartic (mod $p$)}.
            \end{cases}
        \)
    \end{subequations}
    
    Now fix $h_{1}$, an invertible quartic (mod $K$), and let $h_2$ be some yet-undetermined invertible quadratic-nonquartic (mod $K$). Because congruence \eqref{eq:Lambda(modThK)(Rec)} has the same form independent of $\xvq=(3,K)$, for simplicity we only demonstrate the case where $3 \nmid K$, so that $\xvq K=K$. 
    Looking at \eqref{eq:LambdaDiff(mod2)}, we let
        \begin{align*}
        D_{K} := 24K \B(\xL(h_{1},K) - \xL(h_{2},K)\B) - 48\B((h_{1}-h_{2})n + \-{2}(\-{h}_{1} - \-{h}_{2})m\B) - 24K,
        \end{align*}
    so that our task is to show that
        \[
            D_{K} \equiv 0 \mod{48K}.
        \]
    From \eqref{eq:Lambda(mod3)(Rec)} we immediately have
        \[
            D_{K} \equiv 0 \mod{3},
        \]
    and from \eqref{eq:Lambda(mod16)(Rec)} it follows that
        \[
            D_{K} \equiv 8 - 0 - 8 \equiv 0 \mod{16},
        \]
    so it remains only to show that $D_{K} \equiv 0 \mod{K}$.

    For the moment letting
        \[
            X := h_1 - h_2, 
            \quad 
            Y := \-{h}_1 - \-{h}_2,
            \quad\text{and}\quad
            Z := 9B_{2}(\xc) - \-{2}(p+1),
        \]
    and using \eqref{eq:Lambda(modThK)(Rec)}, we find that
        \begin{align*}
        D_{K} 
            &\equiv 9YB_{2}(\xc) - (2X+\-{2}Y)p - \-{2}Y + 2(1-24n)X - 24Ym \mod{K} \\
            &\equiv YZ - 2pX + 2(1-24n)X + (1-24m)Y \mod{K}.
        \end{align*}
    Noting that
        \[
            X = h_{1}h_{2}(\-{h}_{2}-\-{h}_{1}) = -h_{1}h_{2}Y,
        \]
    we further reduce
        \begin{align*}
            D_{K} 
            &\equiv Y \Big( Z + 2ph_{1}h_{2} - 2h_{1}h_{2}(1-24n) + (1-24m) \Big) \mod{K} \\
            &\equiv (\-{h}_{2}-\-{h}_{1}) \Big( 2h_{1}h_{2}(1-24n-p) - (1-24m) - Z \Big) \mod{K} \\
            &\equiv (\-{h}_{2}-\-{h}_{1})d_{K}(h_{2})  \mod{K},
        \end{align*}
    where
        \[
            d_{K}(x) := 2h_{1}(1-24n-p)x - (1-24m) - 9B_{2}(\xc) + \-{2}(p+1).
        \]
    Because $m$, $n$, and $h_{1}$ are fixed, our proof is reduced to demonstrating a unique invertible $h_{2} \mod{K}$ which is a quadratic-nonquartic (mod $p$) and satisfies
        \(
        \label{eq:DK=0(modK)}
            D_{K} \equiv (\bar{h}_{2}-\bar{h}_{1})d_{K}(h_{2}) \equiv 0 \mod{K}.
        \)
    Write $K = p^{a}k$ with $a \geq 1$ and $(k,p)=1$. By simply specifying that $h_{2} \equiv h_{1} \mod{k}$, we ensure that
        \[
            D_{K} \equiv (\bar{h}_{2}-\bar{h}_{1})d_{K}(h_{2}) \equiv 0 \mod{k},
        \]
    so it remains to ensure that $D_{K} \equiv 0 \mod{p^{a}}$. 
    
    Under our assumptions on $h_{1}$ and $h_{2}$, let 
        \begin{align*}
            h_{1} \equiv g^{4\xh} \mod{p} 
            \qquad&\text{and}\qquad 
            h_{2} \equiv g^{4\xl+2} \mod{p},
    \intertext{where $0 \leq \eta, \xl \leq 3$, and recall our assumptions that}
            1-24n \equiv g^{4\nu+\xa+1} \mod{p} 
            \qquad&\text{and}\qquad 
            1-24m \equiv \ol{24}(1-g^{4\mu+\xa}) \mod{p}.
        \end{align*}
    Computing (for $p=17$ and $g=3$) that
        \[
            2 \equiv \-{g}^{2} \mod{p},
            \quad
            \ol{24} \equiv 5 \equiv g^{5} \mod{p}, 
            \quad\text{and}\quad
            Z \equiv 12 \mod{p},
        \]
    we have
        \begin{align*}
            d_{K}(h_{2}) \equiv g^{4(\xh+\xl+\nu)+\xa+1} - 5(1-g^{4\mu+\xa}) - 12 
            \equiv \lf(g^{4(\xh+\xl+\nu)} + g^{4(\mu+1)}\rh)g^{\xa+1} \mod{p}.
        \end{align*}
    To make $d_{K}(h_{2}) \equiv 0 \mod{p}$ then, we need
        \[
            g^{4(\xh+\xl+\nu)} \equiv -g^{4(\mu+1)} \equiv g^{4(\mu+3)} \mod{p},
        \]
    so it suffices that $\xh+\xl+\nu \equiv \mu+3 \mod{4}$. That is, it suffices to have
        \(
        \label{eq:eta-Mod4}
            \xl \equiv \mu-\nu-\xh+3 \mod{4}.
        \)
    Thus, selecting $0 \leq \xl \leq 3$ according to \eqref{eq:eta-Mod4}, and selecting some $h_{*} \equiv g^{4\xl+2} \mod{p}$, said $h_{*}$ is indeed a quadratic-nonquartic (mod $p$) such that $d_{K}(h_{*}) \equiv 0 \mod{p}$.

    The derivative
        \[
            d_{K}{\hspace{-0.6em}}'\hspace{0.3em}
            (x) = 2h_{1}(1-24n-p) \equiv 2h_{1}(1-24n) \not\equiv 0 \mod{p},
        \]
    whereby Hensel's lemma provides a unique $h_{2} \mod{p^{a}}$ such that 
        \[
            h_{2} \equiv h_{*} \mod{p} \qquad\text{and}\qquad d_{K}(h_{2}) \equiv 0 \mod{p^{a}}.
        \]
    Having determined $h_{2} \mod{p^{a}}$, and specified that $h_{2} \equiv h_{1} \mod{k}$, the Chinese remainder theorem provides a unique $h_{2} \mod{K}$ having the desired properties and satisfying \eqref{eq:DK=0(modK)}, and the lemma follows.
\end{proof}

\subsection*{The proof of Theorem \ref{thm:MainVanishing}}
Returning now to equation \eqref{eq:p(n,17)-Odds}, we use Corollary \ref{cor:LK(n,m)Vanishing} to validate our claims regarding $\fL_{K}^{+}(n,0)$ and $\fL_{K}^{+}(n,2)$ in \eqref{eq:LKVanishingClaim(Recall)}; again we fix $p=17$ and $g=3$, a primitive root modulo $17$.

When $m=0$ in \eqref{eq:(1-24n)(1-24m)Congs}, we see that $\xa$ must satisfy
    \[
        1 \equiv \overline{24}(1-g^{4\mu+\xa}) \mod{17}, \qquad\text{which implies that} 
        \qquad
        \xa=3.
    \]
Thus, we require $1-24n \equiv g^{4(\nu+1)} \mod{17}$, i.e., that $1-24n$ be a quartic (mod $17$). When $m=2$ we again find that $\xa=3$, and therefore $1-24n$ must be again be a quartic (mod $17$). From Corollary \ref{cor:LK(n,m)Vanishing} then, it follows that
    \(
    \label{eq:LK(n,0)=0}
        \fL_{K}^{+}(n,0) = \fL_{K}^{+}(n,2) = 0 \qquad\text{when $n \equiv 0,2,8,10 \mod{17}$}.
    \)

Per equation \eqref{eq:p(n,17)-Odds}, for odd $n$ we have
    \[
        \fp(n,\tjac{\cdot}{17}) = \rt{\tf{1}{3}\~{n}^{-1}} \sum_{\substack{k = 1 \\ (k,2p)=p}}^{\infty} \fr{1}{2k} \B\{ 4\fL_{k}^{+}(n,0) I_{1}\?\lf(\tfr{8\pi}{3k}\rt{3\~{n}}\rh) 
            + 2\fL_{k}^{+}(n,2) I_{1}\?\lf(\tfr{4\pi}{3k}\rt{3\~{n}}\rh) \B\},
    \]
and, we conclude that 
    \[
        \fp(n,\tjac{\cdot}{17}) = 0 \qquad\text{for}\quad n \equiv \begin{cases}
            1 &\!\!\!\!\mod{2}, \\
            0,2,8,10 &\!\!\!\!\mod{17}.
        \end{cases}
    \]
This is clearly equivalent to 
    \[
        \fp(n,\tjac{\cdot}{17}) = 0 \qquad \text{for}\quad n \equiv 17,19,25,27 \mod{34},
    \]
and the proof of Theorem \ref{thm:MainVanishing} is complete.

\section{The vanishings of \tops{$\fp(n,-\tjac{\cdot}{17})$}{\fp(n,-(.|17))}}
\label{sec:dagger} 

We now turn to $\fp(n,-\tjac{\cdot}{17})$ and prove Theorem \ref{thm:DaggerVanishing}, which states that
    \[
        \fp(n,-\tjac{\cdot}{17}) = 0 \qquad\text{for all $n \equiv 11,15,29,33 \mod{34}$.}
    \]
We note that this congruence condition on $n$ is equivalent to the conditions that $n$ be odd and $1-24n$ be a quadratic-nonquartic (mod $17$). 
The proof of Theorem \ref{thm:DaggerVanishing} is quite similar to that of Theorem \ref{thm:MainVanishing}, so we only illustrate the minor differences between the two. 

Roughly speaking, the functional equations for the generating function
    \[
        \xF^{\dag}(x) = \prod_{a=1}^{p-1} (-\xc_{a}x^{a};x^{p})_{\infty}^{-1} = \sum_{n=0}^{\infty} \fp(n,-\xc)x^{n}
    \]
consistently mirror those for $\xF(x)$, but with the sets $\bfr$ and $\bfs$ exchanged. In particular, we have
    \[
        \xF^{\dag}(x) = \fr{F_{\bfr}(x^{2})}{F_{\bfr}(x)}F_{\bfs}(x), 
        \qquad\text{while}\qquad
        \xF(x) = F_{\bfr}(x)\fr{F_{\bfs}(x^{2})}{F_{\bfs}(x)}.
    \]

\subsection{The Rademacher series expansion for \tops{$\fp(n,-\tjac{\cdot}{17})$}{\fp(n,-(.|17))}}
We first briefly discuss the derivation of a formula like that of \eqref{eq:p(n,17)FullSeries} for $\fp(n,-\tjac{\cdot}{17})$.
We recall that
    \[
        q := \fr{p-1}{2} \qquad\text{and}\qquad \sum_{\mu \in \bfr\,\cup\,\bfs} \caB_{\mu} = \sum_{\mu=1}^{(p-1)/2} \caB_{\mu} = -pq. 
    \]
In analogy with \eqref{eq:phi(h,k)Omega} and \eqref{eq:phi(h,k)SawSum}, let
    \[
        \xf_{h,k}^{\dag} = \fr{\xo_{\bfr}(2h,k)}{\xo_{\bfr}(h,k)}\xo_{\bfs}(h,k) = \exp\{ \pi i \xL^{\dag}(h,k) \},
    \]
where (cf.~\eqref{eq:Lambda(h,k)} and \eqref{eq:Lambda(h,k)-sVersion})
    \(
    \label{eq:Lambda(h,k)(DAG)}
    \begin{aligned}
    \xL^{\dag}(h,k) 
        &:= \sum_{\substack{\mu \md{\xj k} \\ \xc_{\mu}=1}} 
        \bigg\{ \BigSaw{\fr{2h\mu}{k}}- \BigSaw{\fr{h\mu}{k}}\bigg\} 
        \BigSaw{\fr{\mu}{\xj k}}
        + \sum_{\substack{\mu \md{\xj k} \\ \xc_{\mu}=-1}}
        \BigSaw{\fr{h\mu}{k}} \BigSaw{\fr{\mu}{\xj k}} \\
        &= \fr{1}{2}\B\{ s_{\xc}(2h,k) - 2s_{\xc}(h,k) \B\} + \fr{1}{2}\B\{ s(2h,k) - s(2H,k) \B\},
    \end{aligned}
    \)
and again $\xj = \fr{k}{(k,p)}$.
    
When $(k,2p)=p$, the functional equation for $\xF^{\dag}(x)$ is
    \(
    \label{eq:FEQ-Phi(DAG)-p}
        \xF^{\dag}(x) = \xf_{h,k}^{\dag} \exp\!\bigg\{
            \fr{\pi q}{6k}\Big[ \Big({
            -\fr{3}{q}\xc_{h}\lf(1-\fr{\xc_{2}}{4}\rh)B_{2}(\xc) - \fr{1}{4} 
        }\Big)z^{-1} + z \Big]
        \bigg\} \times
        \begin{cases}
            S^{+}(y'), & \text{if $\xc_{h}=-1$},\\
            S^{-}(y'), & \text{if $\xc_{h}=+1$},    
        \end{cases}
    \)
where $S^{\pm}(y')$ is given in \eqref{eq:Splus} and \eqref{eq:Sminus}; note the differences of signs in \eqref{eq:FEQ-Phi(DAG)-p} and \eqref{eq:FEQ(p)(Parent)}. An effect of this sign difference is that: whereas the ``main terms'' of the series $\fp_{(p)}(n,\xc)$ correspond to $h$ which are quadratics \mbox{(mod $p$)}, the analogous ``main terms'' of the series for $\fp_{(p)}(n,-\xc)$ correspond to $h$ which are \emph{non}quadratics (mod $p$). In particular, we derive the analogous formula
    \[
        \fp_{(p)}(n,-\xc) 
            = \fr{1}{\rt{\~{n}}}
            \sum_{\substack{k=1 \\[0.1em] (k,2p)=p}}^{\infty}
            \sum_{\substack{m=0 \\[0.1em] c_{m} > 0}}^{\infty}
            \fr{\rt{c_{m}}}{2k} \xs_{m}^{+} 
            \big(\fL_{k}^{\dag}\big)^{-}(n,m) 
            \,I_{1}\mspace{-2mu}\Big({\tfr{2\pi}{k}\sqrt{c_{m}\~{n}}}\Big),
    \]
where
    \[
        c_{m} = \lf(1-\fr{\xc_{2}}{4}\rh)B_2(\xc) - \fr{p-1}{24} - 2m, 
        \qquad
        \~{n} = n + \fr{p-1}{24},
    \]
    \[
       \big(\fL_{k}^{\dag}\big)^{-}(n,m) = \psum_{\substack{h \md{k} \\ \xc_{h} = -1}} \xf_{h,k}^{\dag} \exp\!\big({-2\pi i(nh+m\overline{2h})/k}\big) \quad \big(\text{for $(K,2p)=p$}\big),
    \]
and the $\xs^{+}_{m}$ are the series coefficients of $S^{+}(x)$.

In parallel fashion to \eqref{eq:lambdaEven} and \eqref{eq:lambdaOdd}, for $k$ coprime to $p$ let
    \begin{subequations}
    \renewcommand{\theequation}{\theparentequation\hspace{0.08333em}\roman{equation}}
    \[
        \xl_{k}^{\dag} := 2^{-\fr{p-1}{4}} \prod_{\bfr} \fr{\csc(\pi\|2\bar{k}r\|/p)}{\csc(\pi\|\vphantom{\hat{h}}\bar{k}r\|/p)} \prod_{\bfs} \csc(\pi\|\bar{k}s\|/p) \qquad \text{if $2 \mid k$,}
    \]
and let
    \[
        \xl_{k}^{\dag} := 2^{-\fr{p-1}{4}} \prod_{\bfs} \csc(\pi \|\bar{k}s\|/p) \qquad \text{if $2 \nmid k$}.
    \]
    \end{subequations}
For the Kloosterman-type sums $\fL_{k}^{\dag}$ and $\big(\fL_{k}^{\dag}\big)^{-}$, the formulae and bounds of section \ref{sec:KloostermanBounds} are readily applied to $\fL_{k}^{\dag}$ and $\big(\fL_{k}^{\dag}\big)^{-}$, since the relevant formulae for these ``dagger sums'' are the same as those for $\fL_{k}$ and $\fL_{k}^{+}$, with $\bfr$ and $\bfs$ exchanged.

Following the same procedures as for \eqref{eq:p(n,chi)Full<24}, we find that, \emph{again for $p < 24$}, one has
    \[
    \begin{aligned}
    \fp(n,-\xc) 
        &   = \fp_{(1)}(n,-\xc) + \fp_{(2)}(n,-\xc) + \fp_{(p)}(n,-\xc) \\ 
        &   = \fr{\xk_{p}}{4\pi\rt{\~{n}}}
            \sum_{\substack{k=1 \\[0.1em] (k,2p)=1}}^{\infty} 
            \fr{1}{2k}
            \Big\{{ \xl_{k}^{\dag}\fL_{k}^{\dag}(n) + \xl_{2k}^{\dag}\fL_{2k}^{\dag}(n) }\Big\} 
            I_{1} \mspace{-2mu} \lf(\fr{\xk_{p}}{2k}\rt{\~{n}}\rh) \\
        &   \qquad + \fr{\xk_{p}}{4\pi\rt{\~{n}}}
            \sum_{\substack{k=1 \\[0.1em] 4\.|\.k,\, p \ssnmid k}}^{\infty}
            \fr{\xl_{k}^{\dag}}{k} \fL_{k}^{\dag}(n) 
            I_{1}\mspace{-2mu} \lf(\fr{\xk_{p}}{k}\rt{\~{n}}\rh) \\
        &   \qquad + \fr{1}{\rt{\~{n}}}
            \sum_{\substack{k=1 \\[0.1em] (k,2p)=p}}^{\infty} \, 
            \sum_{\substack{m=0 \\[0.1em] c_{m} > 0}}^{\infty} 
            \fr{\sqrt{c_{m}}}{2k} \xs_{m}^{-} \big(\fL_{k}^{\dag}\big)^{-}(n,m)
            \,I_{1}\mspace{-2mu} \Big(\fr{2\pi}{k}\sqrt{c_{m}\~{n}}\Big).
    \end{aligned}
    \]
For $p=17$ then, we have
    \[
    \begin{aligned}
    & \fp(n,-\tjac{\cdot}{17}) 
           = \rt{\tf{2}{51}\~{n}^{-1}}
            \sum_{\substack{k=1 \\[0.1em] (k,2p)=1}}^{\infty}
            \fr{\xl_{k}^{\dag}}{2k} \Big\{{ 
                \fL_{k}^{\dag}(n) + \fL_{2k}^{\dag}(n)
            }\Big\} 
            I_{1}\mspace{-2mu} \lf({
                \tfr{2\pi}{k}\rt{\tfr{2}{51}\~{n}}
            }\rh) \\
        &   \qquad + \rt{\tfr{2}{51}\~{n}^{-1}} 
            \sum_{\substack{k=1 \\[0.1em] 4\.|\.k,\, p \ssnmid k}}^{\infty}
            \!\!\fr{\xl_{k}^{\dag}}{k} \fL_{k}^{\dag}(n) I_{1}\mspace{-2mu}\lf({
                \tfr{4\pi}{k}\rt{\tfr{2}{51}\~{n}}
            }\rh) \\
        &  \qquad + \rt{\tf{1}{3}\~{n}^{-1}} 
            \sum_{\substack{k = 1 \\ (k,2p)=p}}^{\infty} 
            \!\!\fr{1}{2k} \B\{{
                4\big(\fL_{k}^{\dag}\big)^{-}(n,0) I_{1}\mspace{-1.5mu} \lf(\tfr{8\pi}{3k}\rt{3\~{n}}\rh) 
            + 2\big(\fL_{k}^{\dag}\big)^{-}(n,2) I_{1}\mspace{-1.5mu} \lf(\tfr{4\pi}{3k}\rt{3\~{n}}\rh) 
            }\B\}.
    \end{aligned}
    \]

\subsection{Congruences of \tops{$\xL^{\dag}(h,K)$}{𝛬†(h,K)}}
\label{sec:LambdaCongruences(DAG)}

\begin{lemma}[cf.~Lem.~\ref{lem:Lambda(modThK)} and Cor.~\ref{cor:Lambda17(modThK)}]
    \label{lem:Lambda(modThK)(DAG)}
    Let $p \equiv 1 \mod{8}$, let $K=kp$ be odd, and let $(h,K)=1$ with $\xc_{h} = -1$. Then
        \[
            24K \xL^{\dag}(h,K) \equiv -9\bar{h}B_{2}(\xc) - (p-1)(2h + \ol{2h}) \mod{\xvq K},
        \]
    where $\xvq = (3,K)$ and $x\-{x} \equiv 1 \mod{\xvq K}$.
    Moreover, if $3 \nmid K$ then
        \[
            24K \xL^{\dag}(h,K) \equiv 0 \mod{3}.
        \]
\end{lemma}

\begin{proof}
    Using \eqref{eq:Lambda(h,k)(DAG)}, the proof is nearly identical to that of Lemma \ref{lem:Lambda(modThK)}.
\end{proof}

Perhaps the most significant difference between $\xL(h,K)$ and $\xL^{\dag}(h,K)$ is in their respective congruences modulo $16$, and in particular the ``dagger-analogues'' of Lemma
\ref{lem:24KLambda(h,K)(mod16)} and Corollary \ref{cor:Lambda(mod16)(+)}. From \eqref{eq:Lambda(h,k)(DAG)} we have
    \[
        24K \xL^{\dag}(h,K) = 12K \big\{ s_{\xc}(2h,K) - 2s_{\xc}(h,K) \big\} + 12K \big\{ s(2h,K) - s(2H,K) \big\},
    \]
and, just as equation \eqref{eq:Lambda(h,K)(Recall)} leads to Lemma \ref{lem:24KLambda(h,K)(mod16)}, this formula for $24K\xL^{\dag}(h,K)$ leads us to the relation
    \(
    \label{eq:24KLambda(16)(DAG)}
        24K \xL^{\dag}(h,K) \equiv 4(\xc_{h}-1) + 2(2h-1)(p-1) + 8\xt_{\bfe\bfr}(2h,K) \mod{16},
    \)
where
    \[
    \begin{aligned}
        &\xt_{\bfe\bfr}(h,K) := \\
        &\qquad \#\Big\{{
            0 < \mu < K : \text{$p \nmid \mu$, $2 \mid \mu$, $\mu$ is a quadratic (mod $p$), and $\{h\mu\}_{K}$ is odd}
        }\Big\}.
    \end{aligned}
    \]

As in section \ref{sec:Lambda(mod16)}, for the moment we simply record the necessary congruences for $\xt_{\bfe\bfr}(h,K)$ modulo 2, and defer the proofs until section \ref{sec:Transfer}.

\begin{lemma}
    \label{lem:tau(er)(mod2)}
    Let $p \equiv 1 \mod{4}$, fix $g$ a primitive root modulo $p$, let $K=kp$ be odd, and let $h$ be a nonquadratic {\normalfont (mod $p$)} with $(h,K)=1$. One has
        \[
            \xt_{\bfe\bfr}(h,K) \equiv \begin{cases}
                1 \mod{2} & \text{if $\bar{g}h$ is a quartic {\normalfont (mod $p$)}}, \\
                0 \mod{2} & \text{if $\bar{g}h$ is a quadratic-nonquartic {\normalfont (mod $p$)}}.
            \end{cases}
        \]
\end{lemma}

\begin{corollary}
    \label{cor:Lambda(mod16)(DAG)}
    Fix $p=17$ and $g=3$, so that $g$ is a primitive root modulo $p$, let $K=17k$ be odd, and let $h$ be a \ul{non}quadratic {\normalfont (mod $p$)} with $(h,K)=1$. One has
        \[
            24K \xL^{\dag}(h,K) \equiv \begin{cases}
                8 \mod{16} & \text{if $\-{g}h$ is a quartic {\normalfont (mod $p$)}}, \\
                0 \mod{16} & \text{if $\-{g}h$ is a quadratic-nonquartic {\normalfont (mod $p$)}}.
            \end{cases}
        \]
\end{corollary}

\begin{proof}
    This follows from equation \eqref{eq:24KLambda(16)(DAG)}, Lemma \ref{lem:tau(er)(mod2)}, and the fact that $2$ is a quadratic-nonquartic (mod $17$).
\end{proof}

\subsection{The proof of Theorem \ref{thm:DaggerVanishing}}

The following lemmata are identical in form to Lemmata \ref{lem:L2kLk} and \ref{lem:L4k}, respectively, as are their respective proofs, mutatis mutandis.

\begin{lemma}
    Let $p \equiv 1 \mod{8}$ and let $(k,p)=1$ with $k$ odd. One has
        \[
            \fL_{2k}^{\dag}(n) = (-1)^{n}\fL_{k}^{\dag}(n) \qquad \text{for $n \geq 1$}.
        \]
\end{lemma}

\begin{lemma}
    Let $p \equiv 1 \mod{8}$ and let $(k,p)=1$. One has
        \[
            \fL_{4k}^{\dag}(n) = 0 \qquad \text{for odd $n \geq 1$.}
        \]
\end{lemma}

We now recall that
    \[
        \big(\fL_{K}^{\dag}\big)^{-}(n,m) = 
        \psum_{\substack{h \md{K} \\ \xc_{h}=-1}} \exp\!\B\{{  
            \pi i\xL^{\dag}(h,K) - 2\pi i(nh+m\ol{2h})/K 
        }\B\},
    \]
where $2h\ol{2h} \equiv 1 \mod{K}$.

\begin{lemma}[cf.~Lem.~\ref{lem:Lambda(h,K)-exps(mod2)} and Cor.~\ref{cor:LK(n,m)Vanishing}]
    Fix $p=17$ and $g=3$, so that $g$ is a primitive root modulo $p$, let $K = 17k$ be odd, and let $n > 0$ and $m \geq 0$. If one has
        \[
        \begin{aligned}
            1-24n &\equiv g^{4\nu+\xa} &&\mspace{-9mu}\mod{p}, \\
            1-24m &\equiv \ol{24}(1-g^{4\mu+\xa+1}) &&\mspace{-9mu}\mod{p},
        \end{aligned}
        \]
    for some $0 \leq \xa,\mu,\nu \leq 3$, then $\big(\fL_{K}^{\dag}\big)^{-}(n,m) = 0$.
\end{lemma}

\begin{proof}
    This is proved using arguments nearly identical to those used for Lemma \ref{lem:Lambda(h,K)-exps(mod2)} and Corollary \ref{cor:LK(n,m)Vanishing}, with Lemma \ref{lem:Lambda(modThK)(DAG)} and Corollary \ref{cor:Lambda(mod16)(DAG)} used in place of Lemma \ref{lem:Lambda(modThK)} and Corollary \ref{cor:Lambda(mod16)(+)}, respectively.
\end{proof}

Mimicking the computations for equation \eqref{eq:LK(n,0)=0} (keeping $K=17k$ odd), when $m=0$ the relation $1 \equiv \ol{24}(1-g^{4\mu+\xa+1}) \mod{17}$ implies that $\xa=2$, so that $1-24n$ must be a quadratic-nonquartic (mod $17$). When $m = 2$ we again find that $\xa=2$, whereby
    \[
        \big(\fL_{K}^{\dag}\big)^{-}(n,0) = \big(\fL_{K}^{\dag}\big)^{-}(n,2) = 0 \qquad \text{when $n \equiv 11,12,15,16 \mod{17}$}.
    \]
Thus, one has
    \[
        \fp(n,-\tjac{\cdot}{17}) = 0 \qquad
        \text{for }
        n \equiv \begin{cases}
            1 & \mod{2}, \\
            11,12,15,16 & \mod{17},
        \end{cases}
    \]
which is equivalent to the statement that
    \[
        \fp(n,-\tjac{\cdot}{17}) = 0 \qquad \text{for $n \equiv 11,15,29,33 \mod{34}$,}
    \]
and the proof of Theorem \ref{thm:DaggerVanishing} is complete.

{ 

\newcommand{\E}{\mathbf{e}}
\renewcommand{\O}{\mathbf{o}}
\newcommand{\R}{\mathbf{r}}
\renewcommand{\S}{\mathbf{s}} 

\newcommand{\ts}[1]{\textstyle{#1}}
\newcommand{\prlim}[1]{\prod\limits_{#1}}
\newcommand{\btPrSet}[2]{\big(\textstyle{\prod\limits_{#1}#2}\big)}
\newcommand{\BtPrSet}[2]{\Big(\,\textstyle{\prod\limits_{#1}#2}\Big)}

\section{Transfer Lemmata}
\label{sec:Transfer}

At last we establish the elementary congruences modulo 2 for $\xt_{\bfe\bfr}$ and $\xt_{\bfe\bfs}$ from Lemmata \ref{lem:tau(er)(mod2)} and \ref{lem:tau(es)(mod2)}; we first recall that 
    \begin{align*}
        &\xt_{\bfe\bfr}(h,K) = \\
        &\qquad \#\Big\{
            0 < \mu < K : \text{$p \nmid \mu$, $2 \mid \mu$, $\mu$ is a quadratic (mod $p$), and $\{h\mu\}_{K}$ is odd}
        \Big\}, \\
    \intertext{and}
        & \xt_{\bfe\bfs}(h,K) = \\
        &\qquad \#\Big\{ 
            0 < \mu < K : \text{$p \nmid \mu$, $2 \mid \mu$, $\mu$ is a nonquadratic (mod $p$), and $\{h\mu\}_{K}$ is odd}
        \Big\},
    \end{align*}
where $\{x\}_{K}$ satisfies
        \[
            x \equiv \{x\}_{K} \mod{p} \qquad\text{and}\qquad 0 \leq \{x\}_{K} < K.
        \]
Out of independent interest, we also derive congruences for analogous $\xt$-quantities having different pairs of conditions on the parity and ``residuacity'' of $\mu \mod{p}$. To illustrate our computations, we first require a moderate notational ``buy-in''.

\begin{enumerate}[label=(\arabic*)]
\item %
    For $k \geq 1$ let 
        \begin{align*}
            \bfe(k) &:= \{ 0 < e < k : \text{$p \nmid e$, $e$ even} \}, \\
            \bfo(k) &:= \{ 0 < o < k : \text{$p \nmid o$, $o$ odd} \},
        \end{align*}
    and further let
        \begin{align*}
            \bfr(k) &:= \{ 0 < r < k : \text{$p \nmid r$, $r$ is a quadratic (mod $p$)} \}, \\
            \bfs(k) &:= \{ 0 < s < k : \text{$p \nmid s$, $s$ is a nonquadratic (mod $p$)} \}.
        \end{align*}
    
    \noindent \textbf{Note:} We have previously used $\bfr$ and $\bfs$ to denote the sets of quadratics and nonquadratics in the range $\{1,\ldots,\fr{p-1}{2}\}$, respectively. To avoid further ``notational sprawl'', in this section (and only this section) we exclusively use $\bfr(k)$ and $\bfs(k)$ as defined above.
\item%
    For $k \geq 1$ let $\bfe\bfr(k)$ and $\bfe\bfs(k)$ denote the sets of even elements of $\bfr(k)$ and $\bfs(k)$, respectively; that is, let 
        \[
            \bfe\bfr(k) := \bfe(k) \cap \bfr(k)
            \qquad\text{and}\qquad
            \bfe\bfs(k) := \bfe(k) \cap \bfs(k).
        \] 
    The sets $\bfo\bfr(k)$ and $\bfo\bfs(k)$ of odd elements are defined similarly.
\item%
    For $K=kp$ odd and $h$ with $(h,K)=1$, from their definitions we have
        \begin{align*}
            \xt_{\bfe\bfr}(h,K) &= \#\{ r \in \bfe\bfr(K) : \text{$\{hr\}_{K}$ is odd} \}, \\
            \xt_{\bfe\bfs}(h,K) &= \#\{ s \in \bfe\bfs(K) : \text{$\{hs\}_{K}$ is odd} \},
        \end{align*}
    and we define $\xt_{\bfo\bfr}(h,K)$ and $\xt_{\bfo\bfs}(h,K)$ in the obvious analogous manner.
\item %
    For fixed, general $p \equiv 1 \mod{4}$, let $g=g(p)$ be a primitive root modulo $p$, and define $i=i(p)$ so that
        \[
            i \equiv g^{\fr{p-1}{4}} \mod{p} \qquad\text{and}\qquad 
            0 < i < p.
        \]
    Finally, let $\xe = \xe(p)$ be the quantity in $\{0,1\}$ satisfying\footnote{That $\xe$ exists is shown in \eqref{eq:pHalfFact(modp)}. Naturally, the choice of $g$ affects the values of $i$ and $\xe$.}
        \(
        \label{eq:eps(p)}
            (\tfrac{p-1}{2})! \equiv (-1)^{\xe} i \mod{p}.
        \)
\end{enumerate}

Before starting our computations, we list the (mod $2$)-congruences of the different $\xt$-quantities in a table. Here we use the notation
    \[
        [p]_{4} := \fr{p-1}{4}.
    \]
As an example, we read from the table that: \emph{Under our assumptions above, if one has $h \equiv g^{4\xh + 3} \mod{p}$ for some $\xh$, then $\xt_{\bfe\bfs}(h,K) \equiv 1+\xe \mod{2}$.}
{
\newcommand{\rowHead}{\smash{\rotatebox[origin=c]{90}{\text{$h \equiv g^{4\xh+*} \md{p}$}}}}
\newcommand{\colHead}{$\xt_{*}(h,K) \mod{2}$}

\begin{table}[h!]
\renewcommand{\arraystretch}{1.3}
\centering
\[
\begin{tabular}{|c|c||c|c|c|c|}
    \cline{3-6}
\multicolumn{2}{c|}{ } & \multicolumn{4}{c|}{\colHead} \\ 
    \cline{3-6}
\multicolumn{2}{c|}{ } & $\bfe\bfr$ & $\bfe\bfs$ & $\bfo\bfr$ & $\bfo\bfs$ \\ 
    \hhline{--====}
\multirow{4}{*}{\rowHead} & $0$ & 0 & 0 & $[p]_{4}$ & $[p]_{4}$ \\
    \cline{2-6}
 & 1 & $1+\xe$ & $\xe$ & $1+\xe+[p]_{4}$ & $\xe+[p]_{4}$ \\
    \cline{2-6}
 & 2 & $1$ & $1$ & $1+[p]_{4}$ & $1+[p]_{4}$ \\
    \cline{2-6}
 & 3 & $\xe$ & $1+\xe$ & $\xe+[p]_{4}$ & $1+\xe+[p]_{4}$ \\
    \hline
\end{tabular}
\]
\caption{Congruences modulo 2 of different $\xt$-quantities.}
\label{tab:Tau}
\end{table}
}
\noindent We now verify that the $\xe=\xe(p)$ in equation \eqref{eq:eps(p)} is indeed well-defined. Fixing $p\equiv 1\mod{4}$ and defining $g$ and $i$ as above, Wilson's theorem implies that
    \[ 
        \bb(\, \ts{\prod\limits_{\substack{\mu=1\\ \mu\text{ even}}}^{p-1} \mu} \,\bb)^{2} 
    =   \bb(\,\ts{\prod\limits_{\substack{\mu=1\\ \mu\text{ even}}}^{p-1} \mu} \,\bb) \bb(\,\ts{\prod\limits_{\substack{\mu=1 \\ \mu\text{ odd}}}^{p-1} (p-\mu)}\bb) 
    \equiv   (-1)^{\fr{p-1}{2}} \B(\,\ts{\prod\limits_{\mu=1}^{p-1} \mu} \B)
    \equiv   -1 \mod{p},
    \]
whereby
    \[
        \ts{\prod\limits_{\substack{\mu=1\\ \mu\text{ even}}}^{p-1}} \mu  \equiv \pm i \mod{p}.
    \]
On the other hand, 
    \[
        \ts{\prod\limits_{\substack{\mu=1\\ \mu\text{ even}}}^{p-1} \mu }
    =   2^{\fr{p-1}{2}}(\tfrac{p-1}{2})!,
    \]
and since $2^{(p-1)/2} \equiv \pm 1 \mod{p}$, it follows that
    \(
    \label{eq:pHalfFact(modp)}
        (\tfrac{p-1}{2})! \equiv \pm i \mod{p}.
    \)
Thus, the $\xe$ in \eqref{eq:eps(p)} is determined by the sign in \eqref{eq:pHalfFact(modp)}. We note that
    \[
        \xe = 0 \qquad \text{when} \qquad  p=17.
    \]

We now verify the various (mod $2$)-congruences stated in Table \ref{tab:Tau}. 
For brevity, products of the forms $\prod_{\.\bfr(k)} r$ and $\prod_{\.\bfe\bfr(k)} r$ are understood to be taken as $r$ runs over $\bfr(k)$ and $\bfe\bfr(k)$, respectively.

\begin{lemma}
\label{lem:TauERES(Quad)}
    Fix $p \equiv 1 \mod{4}$, let $K = kp$ be odd, let $(h,K)=1$, and suppose that $h$ is a quadratic {\normalfont (mod $p$)}. 
    Then
        \(
        \label{eq:TauER(Quad)}
            \xt_{\bfe\bfr}(h,K) \equiv \begin{cases}
                0 \mod{2} & \text{if $h$ is a quartic {\normalfont(mod $p$)}}, \\
                1 \mod{2} & \text{if $h$ is a quadratic-nonquartic {\normalfont(mod $p$)}}.
            \end{cases}
        \)
    Moreover, one has 
        \(
        \label{eq:TauES-TauER(Quad)}
            \xt_{\bfe\bfs}(h,K) \equiv \xt_{\bfe\bfr}(h,K) \mod{2}.
        \)
\end{lemma}

\begin{proof}
    We begin with $\bfe\bfr(K)$, that is, the set of $0 < r < K$ such that
        \[
            (r,p) = 1, \quad \text{$r$ is even}, \quad\text{and}\quad \text{$r$ is a quadratic (mod $p$)},
        \]
    noting that $|\bfe\bfr(K)| = \tf{k(p-1)}{4}$.
    Multiplying these $r$ by $h$ and reducing modulo $K$ yields quantities $0 < \{hr\}_{K} < K$ that are again coprime to $p$, and are again quadratics \mbox{modulo $p$}. 
    Replacing all \emph{odd} such $\{hr\}_{K}$ with $K - \{hr\}_{K}$, we evidently have
        \[
            h^{\fr{k(p-1)}{4}} \BtPrSet{\bfe\bfr(K)}{r}
        \equiv   \BtPrSet{\bfe\bfr(K)}{\{hr\}_{K}}
        \equiv   (-1)^{\xt_{\bfe\bfr}(h,K)} \BtPrSet{\bfe\bfr(K)}{r} \mod{K}.
        \]
    As $h^{(p-1)/2} \equiv 1 \mod{p}$ by assumption, we deduce that
        \[
            (-1)^{\xt_{\bfe\bfr}(h,K)} \equiv h^{\fr{p-1}{4}}\big(h^{\fr{p-1}{2}}\big)^{\fr{k-1}{2}} \equiv h^{\fr{p-1}{4}} \mod{p},
        \]
    and \eqref{eq:TauER(Quad)} follows at once. By simply repeating the above arguments with $\bfe\bfs(K)$ in place of $\bfe\bfr(K)$, mutatis mutandis, one quickly sees that $\xt_{\bfe\bfs}(h,K)$ has the same congruences as $\xt_{\bfe\bfr}(h,K)$ modulo $2$, which establishes \eqref{eq:TauES-TauER(Quad)}.
\end{proof}

We now consider $\xt_{\bfe\bfr}(h,K) \mod{2}$ when $h$ is a nonquadratic (mod $p$).

\begin{lemma}
\label{lem:tau(ER)(NonQuad)}
    Fix $p \equiv 1 \mod{4}$, let $K = kp$ be odd, and maintain the definitions of $g$, $i$, and $\xe$. If $(h,K)=1$ and $h$ is a nonquadratic {\normalfont(mod $p$)}, then 
        \[
            \xt_{\bfe\bfr}(h,K) \equiv \begin{cases}
                1+\xe \mod{2} & \text{if $\bar{g}h$ is a  quartic {\normalfont(mod $p$)}}, \\
                \xe \mod{2} & \text{if $\bar{g}h$ is a quadratic-nonquartic {\normalfont(mod $p$)}}.
            \end{cases}
        \]
\end{lemma}

\begin{proof}
    Since $h$ is a nonquadratic (mod $p$), we have
        \[
            h^{\fr{k(p-1)}{4}} \BtPrSet{\bfe\bfr(K)}{r}
        \equiv   \BtPrSet{\bfe\bfr(K)}{\{hr\}_{K}}
            \equiv (-1)^{\xt_{\bfe\bfr}(h,K)} \BtPrSet{\bfe\bfs(K)}{s} \mod{K}.
        \]
    Multiplying both sides by $\prod_{\bfe\bfr(K)}r$, we have
        \(
        \label{eq:hProd(modK)}
            h^{\fr{p-1}{4}}\big(h^{\fr{p-1}{2}}\big)^{\fr{k-1}{2}} \BtPrSet{\bfe\bfr(K)}{r}^{2}
        \equiv   (-1)^{\xt_{\bfe\bfr}(h,K)} \BtPrSet{\bfe(K)}{e} \mod{K},
        \)
    and we check how this relation reduces modulo $p$.
    First, since
        \[
            \bfe\bfr(K) = \bfe\bfr(p) 
            \cup \lf(\bigcup_{j=1}^{(k-1)/2}\B( 2jp+\bfe\bfr(p) \B)\rh) 
            \cup \lf(\bigcup_{j=1}^{(k-1)/2}\B( (2j-1)p+\bfo\bfr(p) \B) \rh),
        \]
    we have
        \(
        \label{eq:prER^2(modp)}
            \BtPrSet{\bfe\bfr(K)}{r}^{2} 
            \equiv \BtPrSet{\bfe\bfr(p)}{r}^{2} \BtPrSet{\bfr(p)}{r}^{k-1} 
            \equiv (-1)^{\fr{p-1}{4}}\BtPrSet{\bfr(p)}{r} \cdot 1 
            \equiv (-1)^{1 + \fr{p-1}{4}} \mod{p}.  
        \)
    Then, similarly splitting up $\bfe(K)$ as we did $\bfe\bfr(K)$, we have
        \[
            \ts{\prlim{\bfe(K)}e} 
            \equiv \BtPrSet{\bfe(p)}{e} \B(\,\ts{ 
                \prod\limits_{\bfe(p)} e \prod\limits_{\bfo(p)} o 
            }\B)^{\fr{k-1}{2}} 
            \equiv (-1)^{\fr{k-1}{2}} \BtPrSet{\bfe(p)}{e} 
            \equiv (-1)^{\fr{k-1}{2}} 2^{\fr{p-1}{2}} (\fr{p-1}{2})! \mod{p},
        \]
    and, using \eqref{eq:eps(p)}, we deduce that
        \(
        \label{eq:prEK(modp)}
            \ts{\prod\limits_{\bfe(K)} e} \equiv (-1)^{\fr{k-1}{2}+\fr{p-1}{4}+\xe} \?\cdot i \mod{p}.
        \)

    Returning to \eqref{eq:hProd(modK)} and reducing modulo $p$, we note that $h^{(p-1)/2} \equiv -1 \mod{p}$, and apply \eqref{eq:prER^2(modp)} and \eqref{eq:prEK(modp)}, to conclude that
        \begin{align*}
            h^{\fr{p-1}{4}} (-1)^{\fr{k-1}{2}} (-1)^{1+\fr{p-1}{4}}
                &\equiv (-1)^{\xt_{\bfe\bfr}(h,K)} (-1)^{\fr{k-1}{2}+\fr{p-1}{4}+\xe} \?\cdot i \mod{p}, \\
            (-1)^{1+\xe}(\-{g}h)^{\fr{p-1}{4}} 
                &\equiv (-1)^{\xt_{\bfe\bfr}(h,K)} \mod{p},
        \end{align*}
    and from this the result follows.
\end{proof}

With Lemmata \ref{lem:TauERES(Quad)} and \ref{lem:tau(ER)(NonQuad)}, we have established the following entries of Table \ref{tab:Tau}.

{
\newcommand{\rowHead}{\smash{\rotatebox[origin=c]{90}{\text{$h \equiv g^{4\xh+*} \md{p}$}}}}
\newcommand{\colHead}{$\xt_{*}(h,K) \mod{2}$}

\begin{table}[!ht]
\renewcommand{\arraystretch}{1.3}
\centering
\[
\begin{tabular}{|c|c||c|c|c|c|}
    \cline{3-6}
\multicolumn{2}{c|}{ } & \multicolumn{4}{c|}{\colHead} \\ 
    \cline{3-6}
\multicolumn{2}{c|}{ } & $\bfe\bfr$ & $\bfe\bfs$ & $\bfo\bfr$ & $\bfo\bfs$ \\ 
    \hhline{--====}
\multirow{4}{*}{\rowHead} & $0$ & 0 & 0 & & \\
    \cline{2-6}
 & 1 & $1+\xe$ & & & \\
    \cline{2-6}
 & 2 & $1$ & $1$ & & \\
    \cline{2-6}
 & 3 & $\xe$ & & & \\
    \hline
\end{tabular}
\]
\end{table}
}

\noindent Lemma \ref{lem:TauERES(NonQuad)} handles the remaining two entries in the $\bfe\bfs$-column, and Lemma \ref{lem:Tau-OddToEven} shows that the $\bfo\bfr$- and $\bfo\bfs$-columns are determined by the $\bfe\bfr$- and $\bfe\bfs$-columns, respectively.

\begin{lemma}
    \label{lem:TauERES(NonQuad)}
    Fix $p \equiv 1 \mod{4}$, let $K=kp$ be odd, and let $(h,K)=1$. If $h$ is a nonquadratic {\normalfont (mod $p$)}, then
        \[
            \xt_{\bfe\bfr}(h,K)+\xt_{\bfe\bfs}(h,K) \equiv 1 \mod{2}.
        \]
\end{lemma}

\begin{proof}
    First, we find that
        \[
            h^{\fr{k(p-1)}{2}} \BtPrSet{\bfe\bfr(K)}{r}
        \equiv   h^{\fr{k(p-1)}{4}} (-1)^{\xt_{\bfe\bfr}(h,K)} \BtPrSet{\bfe\bfs(K)}{s}
        \equiv   (-1)^{\xt_{\bfe\bfr}(h,K)+\xt_{\bfe\bfs}(h,K)} \BtPrSet{\bfe\bfr(K)}{r} \mod{K},
        \]
    whereby
        \[
            h^{\fr{k(p-1)}{2}} \equiv (-1)^{\xt_{\bfe\bfr}(h,K)+\xt_{\bfe\bfs}(h,K)} \mod{K}.
        \]
    Since $h^{(p-1)/2} \equiv -1 \mod{p}$ by assumption, it follows that
        \[
            (-1)^{\xt_{\bfe\bfr}(h,K)+\xt_{\bfe\bfs}(h,K)} \equiv (-1)^{k} \equiv -1 \mod{p},
        \]
    which gives the result at once.
\end{proof}

\begin{lemma}
\label{lem:Tau-OddToEven}
    Fix $p \equiv 1 \mod{4}$, let $K=kp$ be odd, and let $(h,K)=1$. One has
    \begin{subequations}
        \begin{align}
        \label{eq:OddToEvenS}
            \xt_{\bfo\bfs}(h,K) &\equiv \fr{p-1}{4} + \xt_{\bfe\bfs}(h,K) \mod{2}, \\
    \intertext{and}
        \label{eq:OddToEvenR}
            \xt_{\bfo\bfr}(h,K) &\equiv \fr{p-1}{4} + \xt_{\bfe\bfr}(h,K) \mod{2}.
        \end{align}
    \end{subequations}
\end{lemma}

\begin{proof}
We recall equation \eqref{eq:OddToEven(mod2)}, which states that
    \(
    \label{eq:OddToEven(mod2)(Recall)}
        \sum_{\substack{\mu\md{K}\\ (\mu,2p)=1}} \B(\fr{\xc_{\mu}-1}{2}\B)\B[\fr{h\mu}{K}\B] 
    \equiv   \fr{(h-1)(p-1)}{4} + \sum_{\substack{\mu \md{K}\\ (\mu,2p)=2}}\B(\fr{\xc_{\mu}-1}{2}\B)\B[\fr{h\mu}{K}\B] \mod{2},
    \)
where $[x]$ denotes the integer part of $x$. Again writing $[x] = x-\{x\}$, and observing that for odd $\mu$ we have 
    \[
        K[h\mu/K] \equiv h\mu + K \{h\mu/K\} \equiv h + \{h\mu\}_{K} \mod{2},
    \]
on the left-hand side of \eqref{eq:OddToEven(mod2)(Recall)} we have
    \begin{align*}
        & \sum_{\substack{\mu\md{K}\\ (\mu,2p)=1}} \B(\fr{\xc_{\mu}-1}{2}\B)\B[\fr{h\mu}{K}\B]
        \equiv h\!
        \sum_{\substack{\mu\md{K}\\ (\mu,2p)=1}} \!\B(\fr{\xc_{\mu}-1}{2}\B) + \xt_{\bfo\bfs}(h,K) \\[0.3em]
        &\qquad\qquad \equiv  h|\bfo\bfs(K)| + \xt_{\bfo\bfs}(h,K) 
        \equiv \fr{hk(p-1)}{4} + \xt_{\bfo\bfs}(h,K) \mod{2}.
    \end{align*}
On the right-hand side of \eqref{eq:OddToEven(mod2)(Recall)}, we have
    \[
        \fr{(h-1)(p-1)}{4} + \sum_{\substack{\mu\md{K}\\ (\mu,2p)=2}} \B(\fr{\xc_{\mu}-1}{2}\B)\B[\fr{h\mu}{K}\B]
    \equiv   \fr{(h-1)(p-1)}{4} + \xt_{\bfe\bfs}(h,K) \mod{2},
    \]
and it follows that
    \[
        \xt_{\bfo\bfs}(h,K) \equiv (hk+h-1)\fr{p-1}{4} + \xt_{\bfe\bfs}(h,K) \equiv \fr{p-1}{4} + \xt_{\bfe\bfs}(h,K) \mod{2}.
    \]
This is \eqref{eq:OddToEvenS}, and an identical argument yields \eqref{eq:OddToEvenR}.
\end{proof}

} 

\bibliography{./vanishing17.bib}
\end{document}